\documentclass[10pt,leqno]{amsart}
\usepackage{graphicx}
\usepackage{indentfirst,csquotes}

\usepackage{latexsym}
\usepackage{amsfonts}
\usepackage{float}
\usepackage{listings}
\usepackage{url}

\usepackage{tikz}
\usetikzlibrary{calc}
\usetikzlibrary{shapes.arrows}

\usepackage{amssymb,amsthm,amsmath}
\usepackage{xcolor,paralist,hyperref,fancyhdr,etoolbox}

\newtheorem{theorem}{Theorem}
\newtheorem*{theorem*}{Theorem}
\newtheorem{lemma}{Lemma}
\newtheorem*{lemma*}{Lemma}

\newtheorem{remark}[theorem]{Remark}
\newtheorem*{definition*}{Definition}

\hypersetup{ colorlinks=true, linkcolor=black, filecolor=black, urlcolor=black }

\floatstyle{boxed}

\numberwithin{table}{section}
\numberwithin{figure}{section}
\numberwithin{equation}{section}
\numberwithin{theorem}{section}
\numberwithin{lemma}{section}

\lstdefinelanguage{mycode}{
	morekeywords=[1]{	function, sym, end, if, else, intval, eval, floor, ceil, disp, repmat, error,
	    isspd, num2str, mid, length, format, long, zeros, ones, for,
	    simplifyFraction, subs, int, diff,},
	sensitive=true,
	morecomment=[l]{\%},
	morestring=[d]{'}{'},
	morestring=[d]{'},
	morestring=[b]{"},
}
\allowdisplaybreaks

\begin{document}

\title[Remarkable upper bounds for interpolation error constants on triangles]{Remarkable upper bounds for interpolation error constants on triangles}
\author[K Kobayashi]{Kenta Kobayashi${}^\dagger$}
\date{\today}

\maketitle

\let\thefootnote\relax
\footnotetext{${}^\dagger$Graduate School of Business Administration, Hitotsubashi Univeristy}
\footnotetext{\address{2-1, Naka, kunitachi, Tokyo, 186-8601, Japan}}
\footnotetext{\email{kenta.k@r.hit-u.ac.jp}}

\begin{abstract}
We introduce remarkable upper bounds, which are sharp and given by simple formulas, for the interpolation error constants on triangles.
These constants are crucial for analyzing interpolation errors, particularly those associated with the finite element method.
We prove boundedness via a numerical verification method and asymptotic analysis.
The proof process used here can be applied to various other norm inequalities.
\end{abstract} 

\bigskip

\section{Introduction} \label{Introduction}
The analysis of interpolation error is important in many applications, such as approximation theory and error estimation for the solution of the finite element method.
To estimate interpolation error, we have to obtain the upper bounds of the constants that appear in the corresponding norm inequalities, referred to as the interpolation error constants.

Let $T$ be a given triangle in $\mathbb{R}^2$ and define the function spaces
$V^{1,1}(T),V^{1,2}(T)$, and $V^2(T)$ as follows:
\begin{align*}
	V^{1,1}(T)&=\left\{\varphi\in H^1(T)\; \Big| \; \int_T\varphi\,dxdy=0\right\}, \\
	V^{1,2}(T)&=\left\{\varphi\in H^1(T)\; \Big| \; \int_{\gamma_k}\varphi\,ds=0,\quad k=1,2,3\right\}, \\
	V^2(T)&=\left\{\varphi\in H^2(T)\; \Big| \;\varphi(p_k)=0,\quad k=1,2,3\right\},
\end{align*}
where $p_1,p_2,p_3$ and $\gamma_1,\gamma_2,\gamma_3$ are the vertices and the edges of $T$, respectively.
Under these settings, it is known that the following interpolation error constants
$C_1(T),C_2(T),C_3(T)$, and $C_4(T)$ exist.
\begin{align*}
	C_1(T)&=\sup_{u\in V^{1,1}(T)\setminus 0}\frac{\|u\|_{L^2(T)}}{\|\nabla u\|_{L^2(T)}}, &
	C_2(T)&=\sup_{u\in V^{1,2}(T)\setminus 0}\frac{\|u\|_{L^2(T)}}{\|\nabla u\|_{L^2(T)}}, \\
	C_3(T)&=\sup_{u\in V^2(T)\setminus 0}\frac{\|u\|_{L^2(T)}}{|u|_{H^2(T)}}, &
	C_4(T)&=\sup_{u\in V^2(T)\setminus 0}\frac{\|\nabla u\|_{L^2(T)}}{|u|_{H^2(T)}},
\end{align*}
where $|\cdot|_{H^k(\Omega)}$ denotes the $H^k$ semi-norm (defined later).

Here, we obtain the following remarkable formulas for the upper bounds of the above error constants.
\begin{theorem} \label{formula}
Let $T$ be an arbitrary triangle. Then, 
\begin{align*}
	C_1(T)&<K_1(T)=\sqrt{\frac{A^2+B^2+C^2}{28}-\frac{S^4}{A^2B^2C^2}}, \\
	C_2(T)&<K_2(T)=\sqrt{\frac{A^2+B^2+C^2}{54}-\frac{S^4}{2A^2B^2C^2}}, \\
	C_3(T)&<K_3(T)=\sqrt{\frac{A^2B^2+B^2C^2+C^2A^2}{83}-\frac{1}{24}\left(\frac{A^2B^2C^2}{A^2+B^2+C^2}+S^2\right)}, \\
	C_4(T)&<K_4(T)=\sqrt{\frac{A^2B^2C^2}{16S^2}-\frac{A^2+B^2+C^2}{30}-\frac{S^2}{5}\left(\frac{1}{A^2}+\frac{1}{B^2}+\frac{1}{C^2}\right)},
\end{align*}
hold, where $A,B,C$ are the edge lengths of $T$ and $S$ is the area of $T$.
\end{theorem}

\medskip

As we will show in Section~\ref{Numerical_Results}, the upper bounds obtained by Theorem~\ref{formula} are sufficiently sharp for practical applications.
Moreover, $K_j(T)$ is convenient for practical calculations since these formulas consist of just four arithmetic operations and
the square root.

We now explain the relationship between interpolation error and an interpolation error constant.
For $u\in H^1(T)$, let $\Pi^{(P0)}u$ be a constant function such that the mean value on triangle $T$ is equal to that of $u$; then,
$$
	\|u-\Pi^{(P0)}u\|_{L^2(T)}\le C_1(T)\|\nabla u\|_{L^2(T)}
$$
holds.
In addition, for $u\in H^2(T)$, let $\Pi^{(P1)}u$ be a linear function whose values coincide with those of $u$ at each vertex of triangle $T$;
then, the following estimate holds.
\begin{align*}
	\|u-\Pi^{(P1)}u\|_{L^2(T)}&\le C_3(T)|u|_{H^2(T)}, \\
	\|\nabla(u-\Pi^{(P1)}u)\|_{L^2(T)}&\le C_4(T)|u|_{H^2(T)}.
\end{align*}
$C_2(T)$ is related to Crouzeix-Raviart interpolation: for $u\in H^1(T)$, let $\Pi^{(CR)}u$ be a linear function
such that the average value on each edge of $T$ coincides with that of $u$; then,
\begin{equation}
	\|u-\Pi^{(CR)}u\|_{L^2(T)}\le C_2(T)\|\nabla(u-\Pi^{(CR)}u)\|_{L^2(T)}\le C_2(T)\|\nabla u\|_{L^2(T)} \label{CR}
\end{equation}
holds.
The second inequality in \eqref{CR} can be shown using the divergence theorem and a proof similar to that in Lemma~\ref{projection1},
which will be discussed later. Moreover, for $u\in H^2(T)$, using the fact that $(u-\Pi^{(CR)}u)_x, (u-\Pi^{(CR)}u)_y\in V^{1,1}(T)$
and \eqref{CR}, the following inequalities also hold.
\begin{align*}
	\|u-\Pi^{(CR)}u\|_{L^2(T)}&\le C_1(T)C_2(T)|u|_{H^2(T)}, \\
	\|\nabla(u-\Pi^{(CR)}u)\|_{L^2(T)}&\le C_1(T)|u|_{H^2(T)}.
\end{align*}

\begin{remark}\upshape
Here, we explain how we derived the formulas $K_j(T)$ presented in Theorem~\ref{formula}.
First, we prepared approximate values of $C_j(T)$ with hundreds of concrete triangles.
Then, we tried to find simple polynomial expressions with $S^2$ and symmetric polynomials of $A^2,B^2,C^2$ using the least-square method that
fit the approximated values of $C_j(T)$.
Note that for $(ABC/S)^2$, the leading term of $K_4(T)$, with a coefficient of $1$, is optimal and cannot be decreased.
This term is related to the divergence rate of $C_4(T)$ when $T$ degenerates.
Although the other coefficients of each term in $K_j(T)$ are slightly adjustable, 
we decided to make the coefficients as simple as possible and convenient in practical applications.
\end{remark}

\begin{remark} \upshape
Theorem~\ref{formula} and its proof for certain concrete
triangles have been previously presented~\cite{Kobayashi}. However, the complete proof for
arbitrary triangles is given in this paper for the first time.
\end{remark}

The rest of this paper is organized as follows.  In Section~\ref{PreResult}, 
prior works on the upper bound or the approximation of $C_j(T)$ are reviewed.
In Section~\ref{Def_Pre}, definitions and preliminaries are given.
The relationship between $C_j(T)$ and specific finite-dimensional eigenvalue problems is explained in Section~\ref{RelationFDEP},
and its concrete matrix form is given in Section~\ref{Construction_FDGEP}. 
In Section~\ref{Positive_Definiteness}, we first show that the validation of the first eigenvalues reduces to verifying the positive definiteness of a specific matrix
and explain how this is accomplished using a numerical verification method.
Using the developed method, we show that $(1+\varepsilon)C_j(T) < K_j(T)$
holds for certain small $\varepsilon > 0$ on 12,168 concrete triangles.
In Section~\ref{Continuification}, we show that, when evaluations of interpolation error constants
 have been obtained for two triangles with very close shapes, we can obtain slightly milder evaluations
 on triangles with intermediate shapes between those of the two triangles.
In Section~\ref{Non_Degenerate}, using the results given in the previous two sections, we prove Theorem~\ref{formula} for triangles that are not significantly degenerate.
In Section~\ref{Degenerate}, we first show a kind of monotonicity
of $C_j(T),\; j=1,2,3$ when triangles degenerate.  Using these results
and the results in Section~\ref{Positive_Definiteness} and Section~\ref{Continuification}, $C_j(T) < K_j(T),\; j=1,2,3$ in Theorem~\ref{formula} is proved when triangles degenerate.
In Section~\ref{Asymptotic_Analysis}, $C_4(T) < K_4(T)$ in Theorem~\ref{formula} is proved 
for degenerate triangles using asymptotic analysis. 
In Section~\ref{Numerical_Results}, we give numerical examples that show the well-fitness
of the formulas $K_j(T)$.
In Section~\ref{Circumradius_Condition}, we describe how the expression of $K_4(T)$ is closely related to a property that we call the circumradius condition,
which has significant applications in error analysis for the finite element method.
In Section~\ref{Conclusion}, we give the concluding remarks.

Appendix~\ref{Appendix1} contains lemmas on the deformation of expressions and inequality evaluations
that are too complex to be included in the main text.
In the proofs of some theorems and lemmas, we employed a formula manipulation system and a numerical verification method.
To make the paper as self-contained as possible, we have listed the program codes used in this process in Appendix~\ref{Appendix2}.
The same programs can also be downloaded from the authors' GitHub repository~\cite{GitHub}.

To prove Theorem~\ref{formula}, this study uses many theorems and lemmas.
Their relationship is shown in Fig.~\ref{Block_diagram}.

\begin{figure}[!t]
	\centering
\begin{tikzpicture}[scale=0.85]
\draw[rounded corners](5,17) rectangle (7.1,18.1)node[pos=0.5,align=center]{Lemma\\3.1};
\draw[rounded corners](8,17) rectangle (10.1,18.1)node[pos=0.5,align=center]{Lemma\\3.3};

\draw[rounded corners](5,15) rectangle (7.1,16.1)node[pos=0.5,align=center]{Lemma\\3.2};
\draw[rounded corners](8,15) rectangle (10.1,16.1)node[pos=0.5,align=center]{Theorem\\4.1};

\draw[rounded corners](1,15) rectangle (3.1,16.1)node[pos=0.5,align=center]{Lemma\\A.1};

\draw[rounded corners](1,13) rectangle (3.1,14.1)node[pos=0.5,align=center]{Lemma\\A.2};
\draw[rounded corners](4,13) rectangle (6.1,14.1)node[pos=0.5,align=center]{Lemma\\A.4};

\draw[rounded corners](7,12) rectangle (9.1,13.1)node[pos=0.5,align=center]{Theorem\\7.2};
\draw[rounded corners](11,12) rectangle (13.1,13.1)node[pos=0.5,align=center]{Theorem\\8.1};

\draw[rounded corners](1,11) rectangle (3.1,12.1)node[pos=0.5,align=center]{Lemma\\A.5};
\draw[rounded corners](4,11) rectangle (6.1,12.1)node[pos=0.5,align=center]{Lemma\\7.1};

\draw[rounded corners](7,10) rectangle (9.1,11.1)node[pos=0.5,align=center]{Theorem\\7.1};

\draw[rounded corners](1,9) rectangle (3.1,10.1)node[pos=0.5,align=center]{Lemma\\A.7};
\draw[rounded corners](4,9) rectangle (6.1,10.1)node[pos=0.5,align=center]{Lemma\\A.6};

\draw[rounded corners](7,8) rectangle (9.1,9.1)node[pos=0.5,align=center]{Lemma\\9.1};
\draw[rounded corners](11,8) rectangle (13.1,9.1)node[pos=0.5,align=center]{Theorem\\9.1};

\draw[rounded corners](14,7.4) rectangle (16.1,9.7)node[pos=0.5,align=center]{Theorem\\1.1\\(Main\;\;\\\;\;Result)};

\draw[rounded corners](1,7) rectangle (3.1,8.1)node[pos=0.5,align=center]{Lemma\\A.3};

\draw[rounded corners](4,6) rectangle (6.1,7.1)node[pos=0.5,align=center]{Lemma\\A.8};
\draw[rounded corners](7,6) rectangle (9.1,7.1)node[pos=0.5,align=center]{Lemma\\9.2};

\draw[rounded corners](1,4) rectangle (3.1,5.1)node[pos=0.5,align=center]{Lemma\\A.10};
\draw[rounded corners](4,4) rectangle (6.1,5.1)node[pos=0.5,align=center]{Lemma\\A.11};
\draw[rounded corners](7,4) rectangle (9.1,5.1)node[pos=0.5,align=center]{Lemma\\A.9};
\draw[rounded corners](11,4) rectangle (13.1,5.1)node[pos=0.5,align=center]{Theorem\\10.1};

\draw[rounded corners](5,2) rectangle (7.1,3.1)node[pos=0.5,align=center]{Lemma\\5.1};
\draw[rounded corners](8,2) rectangle (10.1,3.1)node[pos=0.5,align=center]{Theorem\\6.1};

\draw[rounded corners](5,0) rectangle (7.1,1.1)node[pos=0.5,align=center]{Lemma\\5.2};
\draw[rounded corners](8,0) rectangle (10.1,1.1)node[pos=0.5,align=center]{Theorem\\6.2};

\draw[->,thick](7.1,17.55)--(7.9,17.55); 
\draw[->,thick](7.05,17.05)--(7.95,16.15); 
\draw[->,thick](9.05,17)--(9.05,16.2); 
\draw[->,thick](7.1,15.55)--(7.9,15.55); 

\draw[->,thick,rounded corners=10pt](1.05,15.05)--(0.5,14.5)--(0.5,10.6)--(0.95,10.15); 
\draw[->,thick](2.05,15)--(2.05,14.2); 
\draw[->,thick](3.1,13.55)--(3.9,13.55); 
\draw[->,thick](2.05,13)--(2.05,12.2); 
\draw[->,thick](5.05,13)--(5.05,12.2); 
\draw[->,thick](9.1,12.55)--(10.9,12.55); 
\draw[->,thick](6.1,11.7)--(6.9,12.4); 
\draw[->,thick](6.1,11.4)--(6.9,10.7); 
\draw[->,thick](9.1,10.7)--(10.95,12); 
\draw[->,thick](3.1,11.55)--(3.9,11.55); 
\draw[->,thick](3.05,10.05)--(3.95,10.95); 
\draw[->,thick](3.05,8.05)--(3.95,8.95); 
\draw[->,thick](5.05,10.1)--(5.05,10.9); 
\draw[->,thick](9.1,10.4)--(10.95,9.1); 
\draw[->,thick](9.1,8.55)--(10.9,8.55); 
\draw[->,thick](9.1,6.7)--(10.95,8); 
\draw[->,thick](6.1,6.55)--(6.9,6.55); 
\draw[->,thick](3.1,4.55)--(3.9,4.55); 
\draw[->,thick](6.1,4.55)--(6.9,4.55); 
\draw[->,thick](9.1,4.55)--(10.9,4.55); 

\draw[->,thick](7.1,2.55)--(7.9,2.55); 
\draw[->,thick](7.05,2.05)--(7.95,1.15); 
\draw[->,thick](7.1,0.55)--(7.9,0.55); 
\draw[->,thick](7.05,1.05)--(7.95,1.95); 

\draw[->,thick,rounded corners=10pt](9.5,15)--(9.5,6.6)--(10.95,5.15); 
\draw[->,thick,rounded corners=10pt](10,3.1)--(10,10.1)--(11.7,11.9); 
\draw[->,thick,rounded corners=10pt](10,15)--(10,11.4)--(11.7,9.2); 
\draw[->,thick,rounded corners=10pt](10.05,1.05)--(10.5,1.5)--(10.5,6.5)--(11.7,7.9); 
\draw[->,thick,rounded corners=10pt](10.1,15.55)--(11.7,15)--(11.7,13.2); 
\draw[->,thick,rounded corners=10pt](7.05,0.05)--(7.6,-0.5)--(10.5,-0.5)--(11.7,0.7)--(11.7,3.9); 

\draw[->,thick](13.05,12.05)--(14.2,9.8); 
\draw[->,thick](13.1,8.55)--(13.9,8.55); 
\draw[->,thick](12.2,8)--(12.2,5.2); 
\draw[->,thick](13.05,5.05)--(14.2,7.3); 

\end{tikzpicture}
	\caption{Block diagram of theorems and lemmas.}
	\label{Block_diagram}
\end{figure}

\section{Previously reported results} \label{PreResult}
In connection with the finite element method, there are numerous studies on the relation
between $C_4(T)$ and error estimates,
such as those on {\it a priori} error estimates~\cite{BabuskaAziz,BrennerScott,Ciarlet,KobayashiTsuchiya,KobayashiTsuchiya2,
Lehmann,LiuKikuchi,NakaoYamamoto,Zlamal}
 and those on {\it a posteriori} error estimates~\cite{BrennerScott,KikuchiSaito,LiuKikuchi}.

For the explicit upper bound for $C_4(T)$, Arcangeli and Gout~\cite{ArcangeliGout} obtained the following estimates.
$$
	C_4(T)\le \frac{3d(T)^2}{\rho(T)},
$$
where $d(T)$ is the diameter of $T$ and $\rho(T)$ is the diameter of the inscribed circle of $T$.
They also obtained the following upper bound for $C_3(T)$.
\begin{equation*}
	C_3(T)\le 3d(T)^2.
\end{equation*}
Meinguet and Descloux~\cite{MeinguetDescloux} improved this result and obtained
\begin{equation*}
	C_4(T)\le \frac{1.21d(T)^2}{\rho(T)}.
\end{equation*}
Natterer~\cite{Natterer} showed that $C_4(T)$ is bounded in terms of
$C_4(T_{0,1})$, where $T_{0,1}$ is an isosceles right triangle with edge lengths of $1,1$ and $\sqrt{2}$.
Specifically, they showed that
\begin{equation}
	C_4(T)\le C_4(T_{0,1})\cdot\frac{\alpha^2+\beta^2+\sqrt{\alpha^4+2\alpha^2\beta^2\cos2\theta+\beta^4}}
		{\sqrt{2(\alpha^2+\beta^2-\sqrt{\alpha^4+2\alpha^2\beta^2\cos2\theta+\beta^4})}},
\label{NattererEstimation}
\end{equation}
where $\alpha$ and $\beta$ are the lengths of two sides of $T$ and $\theta$ is an included angle (Fig.~\ref{Fig1}).
In the same paper, they proved $C_4(T_{0,1})\le0.81$.

\begin{figure}[t]
	\centering
\begin{tikzpicture}(line width=1.5pt, scale=1.0)
  \coordinate (A) at (0.0,0.0);
  \coordinate (B) at (4.2,1.2);
  \coordinate (C) at (-1.5,3.0);
  \draw[thick] (A) -- node[below]{$\beta$}(B) -- (C) -- node[left]{$\alpha$}(A);
  \coordinate (D) at (0.525,0.15);
  \coordinate (E) at (-0.25,0.5);
  \draw [bend right,thin] (D) to node[pos=0.38,above]{$\theta$} (E) ;
\end{tikzpicture}
	\caption{$\alpha,\beta$, and $\theta$ for triangle $T$.}
	\label{Fig1}
\end{figure}

Nakao and Yamamoto~\cite{NakaoYamamoto} proved that
$$
	C_4(T_{0,1})\le 0.4939
$$
using a numerical verification method. Kikuchi and Liu~\cite{KikuchiLiu2007} proved that
$C_4(T_{0,1})$ is bounded by the maximum positive solution of the transcendental equation for $\mu$:
$$
	\frac{1}{\mu}+\tan\frac{1}{\mu}=0
$$
and showed that
\begin{equation*}
	C_4(T_{0,1})\le 0.49293.
\end{equation*}
Moreover, Liu and Kikuchi~\cite{LiuKikuchi} proved that
\begin{equation}
	C_4(T)\le C_4(T_{0,1})\cdot\frac{1+\cos\theta}{\sin\theta}
\sqrt{\frac{\alpha^2+\beta^2+\sqrt{\alpha^4+2\alpha^2\beta^2\cos2\theta+\beta^4}}{2}}.
\label{LiuKikuchiEstimation}
\end{equation}

Note that the estimation \eqref{LiuKikuchiEstimation} is consistent with the maximum angle condition~\cite{BabuskaAziz}, whereas the estimation \eqref{NattererEstimation} is not.
In fact, if we fix $\beta$ and $\theta$ and let $\alpha\rightarrow0$, the right-hand side of 
\eqref{NattererEstimation} diverges to infinity, whereas the right-hand side of \eqref{LiuKikuchiEstimation}
remains bounded.

$C_1(T)$ is known as the Poincar\'e--Friedrichs constant.
Payne and Weinberger~\cite{PayneWeinberger} obtained
$$
	C_1(T)\le\frac{d(T)}{\pi}.
$$
This estimation is valid for any convex domain. For an arbitrary triangle $T$,
Laugesen and Siudeja~\cite{LaugesenSiudeja} obtained
\begin{equation}
	C_1(T)\le\frac{d(T)}{j_{1,1}},
\label{LaugesenSiudejaEstimation}
\end{equation}
where $j_{1,1}=3.83170597\dots$ denotes the first positive root of the Bessel function $J_1$.

On the other hand, Kikuchi and Liu~\cite{KikuchiLiu2007} proved that
$$
	C_1(T_{0,1})=\frac{1}{\pi}
$$
and
\begin{equation}
	C_1(T)\le C_1(T_{0,1})\sqrt{1+|\cos\theta|}\;\max(\alpha,\beta).
\label{KikuchiLiuEstimation2}
\end{equation}

There are only a few results for $C_2(T)$ itself.
However, $C_2(T)$ is bounded by the so-called Babu\v{s}ka-Aziz constant, whose existence was proved
by Babu\v{s}ka and Aziz\cite[Lemma 2.1]{BabuskaAziz}.
From the upper bound for the Babu\v{s}ka-Aziz constant obtained by Liu and Kikuchi~\cite{LiuKikuchi},
we have
$$
	C_2(T)\le 0.34856\sqrt{1+|\cos\theta|}\;\max(\alpha,\beta).
$$

For most triangles, our formulas $K_j(T)$ give better upper bounds than the preceding results.
The exception is that \eqref{LaugesenSiudejaEstimation}
or \eqref{KikuchiLiuEstimation2} provides a slightly lower value than that provided by $K_1(T)$ for some triangles.

There are some results on computing lower bounds of eigenvalues of elliptic operators~\cite{CarstensenGedicke,LiuOishi,LuoLinXie,Repin,SebestovaVejchodsky}
which can be applied to compute the upper bounds of $C_1(T)$ or $C_2(T)$.
Compared to these results, our method is only applicable to the triangular domain,
but it has the advantage that sharp upper bounds can be obtained through a simple implementation.

\section{Definitions and preliminaries} \label{Def_Pre}
In this section, we provide definitions of symbols and notations, as well as some preliminary lemmas.

\begin{figure}[t]
	\centering
\begin{tikzpicture}(scale=0.8)
  \coordinate [label=below:{$p_1$}](A) at (0.0,0.0);
  \coordinate [label=below:{$p_2$}](B) at (5.0,0.0);
  \coordinate [label=above:{$p_3$}](C) at (1.5,3.0);
  \draw [line width = 0.5pt](A) -- node[below]{$\gamma_3$}(B)
            -- node[pos=0.3,right=1mm]{$\gamma_1$}(C)
             -- node[left=1mm]{$\gamma_2$}(A);
  \coordinate (D) at (3.4,1.4);
  \coordinate (E) at (2.2,2.4);
  \coordinate (F) at (4.4,2.6);
  \draw [->,>=stealth, line width = 1.5pt]  (D) -- node[pos=0.62,below=2mm]{$\nu$}(E);
  \draw [->,>=stealth, line width = 1.5pt]  (D) -- node[pos=0.65,below=2mm]{$n$}(F);
\end{tikzpicture}
	\caption{Notation of vertices and edges of $T$.}
	\label{Fig1.5}
\end{figure}

For a given triangle $T$, let $p_1(T),p_2(T),p_3(T)$ be vertices of $T$ and
$\gamma_1(T), \gamma_2(T), \gamma_3(T)$ be edges $p_2(T)p_3(T),\, p_3(T)p_1(T),\, p_1(T)p_2(T)$, respectively.
Let $n(T)$ be the outer normal unit vector on $\partial T$,
$\nu(T)$ be the unit direction vector that takes a counterclockwise direction through $\partial T$, and
$ds(T)$ be the line element on $\partial T$ (Fig.~\ref{Fig1.5}).
We omit $``(T)"$ if there is no possibility of confusion.
We use Cartesian coordinates $(x,\,y)$ and the usual notation for the $L^2$ norm.
We define the $H^k$ semi-norm $|\cdot|_{H^k(T)}$ as
$
	|u|_{H^k(\Omega)}^2=\sum_{j=0}^k\binom{k}{j}\left\|\frac{\partial^ku}{\partial x^j\partial y^{k-j}}\right\|_{L^2(\Omega)}^2.
$
For functions, subscripts are used to indicate partial derivatives.
With $T_{a,b}$ denoting a triangle whose vertices are $(0,0)$,\;$(1,0)$, and $(a,b)$,
we define
\begin{equation}
	L_j(a,b)=K_j(T_{a,b})^2,\qquad j=1,2,3,4. \label{Ljab}
\end{equation}
In Appendix~\ref{Appendix1}, we use the notation $d_k=1-ka$.

Let $Q_\alpha$ and $Q_\beta$ denote the following polynomial spaces.
\begin{align*}
	Q_\alpha&=\Big\{ a_1(x^2+y^2) + a_2x + a_3y + a_4 \;\Big|\; a_1,\dots,a_4\in \mathbb{R}\Big\}, \\
	Q_\beta&=\Big\{ a_1x^2+a_2xy+a_3y^2 + a_4x + a_5y + a_6 \;\Big|\; a_1,\dots,a_6\in \mathbb{R}\Big\}.
\end{align*}
Note that both $Q_\alpha$ and $Q_\beta$ are invariant under constant shifts and rotations,
and thus they are independent of the choice of coordinates.
Let $\tau$ be a given triangle. We define two kinds of second-order interpolation, 
$\Pi^{(\alpha)}_\tau\varphi$ for $\varphi\in H^1(\tau)$ and
$\Pi^{(\beta)}_\tau\varphi$ for $\varphi\in H^2(\tau)$, on triangle $\tau$ as follows:
\begin{align*}
	&\left\{\begin{aligned}
		\Pi^{(\alpha)}_\tau&\varphi\in Q_\alpha, \\
		\int_{\gamma_k}\Pi^{(\alpha)}_\tau&\varphi\,ds=\int_{\gamma_k}\varphi\,ds,\qquad k=1,2,3, \\
		\iint_\tau \Pi^{(\alpha)}_\tau&\varphi\,dxdy=\iint_\tau\varphi\,dxdy,
	\end{aligned}\right. \\
	&\left\{\begin{aligned}
		\Pi^{(\beta)}_\tau&\varphi\in Q_\beta, \\
		\Pi^{(\beta)}_\tau&\varphi(p_k) = \varphi(p_k),\qquad k=1,2,3, \\
		\int_{\gamma_k}\nabla \Pi^{(\beta)}_\tau&\varphi\cdot n\,ds=\int_{\gamma_k}\nabla\varphi\cdot n\,ds,\qquad k=1,2,3. \\
	\end{aligned}\right. 
\end{align*}
Since the number of constraints and degrees of freedom are equal, these two interpolations are uniquely determined.
Non-conforming finite elements with bases of the same form as $\Pi^{(\alpha)}_\tau\varphi$
 and $\Pi^{(\beta)}_\tau\varphi$ are known as the enriched Crouzeix-Raviart element~\cite{Hu}
and the Morley element~\cite{Morley}, respectively.

In the rest of this section, we prepare some preliminary lemmas.
\begin{lemma} \label{C2-gradient}
If $\varphi\in V^2(\tau)$ satisfies
$$
	\int_{\gamma_k}\nabla\varphi\cdot n\,ds=0,\quad k=1,2,3,
$$
then
$$
	\varphi_x,\;\varphi_y\in V^{1,2}(\tau)
$$
holds.
\end{lemma}
\begin{proof}\quad
From $\varphi(p_1)=\varphi(p_2)=\varphi(p_3)=0$, we have
$$
	\int_{\gamma_k}\nabla\varphi\cdot\nu\,ds=0,\qquad k=1,2,3.
$$
Then, together with the assumption of this lemma,
$$
	\int_{\gamma_k}\nabla\varphi\cdot w\,ds=0,\qquad k=1,2,3,
$$
holds for any fixed vector $w$, which proves the lemma.
\end{proof}

On the interpolations $\Pi^{(\alpha)}_\tau$ and $\Pi^{(\beta)}_\tau$, the following orthogonal properties hold.
\begin{lemma} \label{projection1}
For $\varphi\in H^1(\tau)$, it holds that
$$
	\|\nabla \Pi^{(\alpha)}_\tau\varphi\|_{L^2(\tau)}^2+\|\nabla(\varphi-\Pi^{(\alpha)}_\tau\varphi)\|_{L^2(\tau)}^2=\|\nabla\varphi\|_{L^2(\tau)}^2.
$$
\end{lemma}
\begin{lemma} \label{projection2}
For $\varphi\in H^2(\tau)$, it holds that
$$
	|\Pi^{(\beta)}_\tau\varphi|_{H^2(\tau)}^2+|\varphi-\Pi^{(\beta)}_\tau\varphi|_{H^2(\tau)}^2=|\varphi|_{H^2(\tau)}^2.
$$
\end{lemma}

\begin{proof}[Proof of Lemma~\ref{projection1}]\quad
Let
$$
	\Pi^{(\alpha)}_\tau\varphi=a_1(x^2+y^2) + a_2x + a_3y + a_4.
$$
Then, the divergence theorem yields
\begin{align*}
	\|\nabla\varphi\|_{L^2(\tau)}^2&-\|\nabla \Pi^{(\alpha)}_\tau\varphi\|_{L^2(\tau)}^2-\|\nabla(\varphi-\Pi^{(\alpha)}_\tau\varphi)\|_{L^2(\tau)}^2 \\
	&=2\iint_\tau \nabla(\varphi-\Pi^{(\alpha)}_\tau\varphi)\cdot\nabla \Pi^{(\alpha)}_\tau\varphi\,dxdy \\
	&=2\iint_\tau \text{div}\Big( (\varphi-\Pi^{(\alpha)}_\tau\varphi)\nabla \Pi^{(\alpha)}_\tau\varphi\Big)dxdy
		- 2\iint_\tau (\varphi-\Pi^{(\alpha)}_\tau\varphi)\Delta \Pi^{(\alpha)}_\tau\varphi\,dxdy \\
	&=2\oint_{\partial \tau}(\varphi-\Pi^{(\alpha)}_\tau\varphi)\nabla \Pi^{(\alpha)}_\tau\varphi\cdot n\,ds
		- 8a_1\iint_\tau (\varphi-\Pi^{(\alpha)}_\tau\varphi)\,dxdy \\
	&=2\oint_{\partial \tau}(\varphi-\Pi^{(\alpha)}_\tau\varphi)\binom{2a_1x+a_2}{2a_1y+a_3}\cdot n\,ds \\
	&=4a_1\oint_{\partial \tau}(\varphi-\Pi^{(\alpha)}_\tau\varphi)\binom{x}{y}\cdot n\,ds \\
	&=4a_1\int_{\gamma_1}(\varphi-\Pi^{(\alpha)}_\tau\varphi)
	\left(\binom{x}{y} - p_2\right) \cdot n\,ds \\
	& \qquad \qquad	+4a_1\int_{\gamma_2}(\varphi-\Pi^{(\alpha)}_\tau\varphi)
		  \left(\binom{x}{y} - p_3\right)\cdot n\,ds \\
	&\qquad\qquad\qquad\qquad +4a_1\int_{\gamma_3}(\varphi-\Pi^{(\alpha)}_\tau\varphi)\left(\binom{x}{y} - p_1\right)\cdot n\,ds \\
	&	=0.
\end{align*}
\end{proof}

\begin{proof}[Proof of Lemma~\ref{projection2}]\quad
Let
$$
	\Pi^{(\beta)}_\tau\varphi=a_1x^2+a_2xy+a_3y^2 + a_4x + a_5y + a_6.
$$
Then, the divergence theorem yields
\begin{align*}
	|\varphi|_{H^2(\tau)}^2&-|\Pi^{(\beta)}_\tau\varphi|_{H^2(\tau)}^2-|\varphi-\Pi^{(\beta)}_\tau\varphi|_{H^2(\tau)}^2 \\
	&=2\iint_\tau\Big((\varphi-\Pi^{(\beta)}_\tau\varphi)_{xx}(\Pi^{(\beta)}_\tau\varphi)_{xx}
		+2(\varphi-\Pi^{(\beta)}_\tau\varphi)_{xy}(\Pi^{(\beta)}_\tau\varphi)_{xy} \\
	&\qquad\qquad\qquad\qquad\qquad\qquad\qquad
		+(\varphi-\Pi^{(\beta)}_\tau\varphi)_{yy}(\Pi^{(\beta)}_\tau\varphi)_{yy}\Big)dxdy \\
	&=2\iint_\tau \text{div}\binom{\nabla(\varphi-\Pi^{(\beta)}_\tau\varphi)\cdot \nabla(\Pi^{(\beta)}_\tau\varphi)_x}
		{\nabla(\varphi-\Pi^{(\beta)}_\tau\varphi)\cdot \nabla(\Pi^{(\beta)}_\tau\varphi)_y}dxdy \\
	&=2\oint_{\partial \tau}\binom{\nabla(\varphi-\Pi^{(\beta)}_\tau\varphi)\cdot \nabla(\Pi^{(\beta)}_\tau\varphi)_x}
		{\nabla(\varphi-\Pi^{(\beta)}_\tau\varphi)\cdot \nabla(\Pi^{(\beta)}_\tau\varphi)_y}\cdot n\,ds \\
	&=2\oint_{\partial \tau}\nabla(\varphi-\Pi^{(\beta)}_\tau\varphi)\cdot \nabla(\nabla \Pi^{(\beta)}_\tau\varphi\cdot n)\,ds \\
	&=2\oint_{\partial \tau}\nabla(\varphi-\Pi^{(\beta)}_\tau\varphi)\cdot \binom{2a_1 \;\; a_2}{a_2\;\;2a_3}n\,ds.
\end{align*}
Here, Lemma~\ref{C2-gradient} yields
$$
	\int_{\gamma_k}(\varphi-\Pi^{(\beta)}_\tau\varphi)_x\,ds
	=\int_{\gamma_k}(\varphi-\Pi^{(\beta)}_\tau\varphi)_y\,ds
	=0,\qquad k=1,2,3,
$$
which leads us to the conclusion.
\end{proof}

\section{Relation to finite-dimensional eigenvalue problem} \label{RelationFDEP}
In this section, we explain how to bound $C_j(T)$ by the solutions of the finite-dimensional problem.  Let $n \ge 2$ be an integer.
We divide triangle $T$ into $n^2$ congruent small triangles with a similarity ratio of $1/n$ (Fig.~\ref{Fig2}). 

\begin{figure}[t]
	\centering
\begin{tikzpicture}[scale=1.2]
  \coordinate (A) at (0.0,0.0);
  \coordinate (B) at (5.0,0.0);
  \coordinate (C) at (3.5,3.0);
  \draw [line width = 1.4pt](A) -- (B) -- (C) -- (A);
  \coordinate (A1) at ($(A)!0.25!(B)$);
  \coordinate (A2) at ($(A)!0.5!(B)$);
  \coordinate (A3) at ($(A)!0.75!(B)$);
  \coordinate (B1) at ($(B)!0.25!(C)$);
  \coordinate (B2) at ($(B)!0.5!(C)$);
  \coordinate (B3) at ($(B)!0.75!(C)$);
  \coordinate (C1) at ($(C)!0.25!(A)$);
  \coordinate (C2) at ($(C)!0.5!(A)$);
  \coordinate (C3) at ($(C)!0.75!(A)$);
  \draw [line width = 0.6pt] (A3) -- (B1);
  \draw [line width = 0.6pt] (A2) -- (B2);
  \draw [line width = 0.6pt] (A1) -- (B3);
  \draw [line width = 0.6pt] (A3) -- (C1);
  \draw [line width = 0.6pt] (A2) -- (C2);
  \draw [line width = 0.6pt] (A1) -- (C3);
  \draw [line width = 0.6pt] (C3) -- (B1);
  \draw [line width = 0.6pt] (C2) -- (B2);
  \draw [line width = 0.6pt] (C1) -- (B3);
  \draw (0.78,0.25) node{$\tau_1$};
  \draw (1.4,0.45)  node{$\tau_2$};
  \draw (2.0,0.25)  node{$\tau_3$};
  \draw (3.4,2.5)   node{$\tau_{n^2}$};
  \coordinate (D1) at (2.6,0.375);
  \coordinate (D2) at (4.45,0.375);
  \draw [loosely dotted, line width = 1.0pt] (D1) -- (D2);
  \coordinate (D3) at (1.7,1.125);
  \coordinate (D4) at (4.1,1.125);
  \draw [loosely dotted, line width = 1.0pt] (D3) -- (D4);
  \coordinate (D5) at (2.5,1.875);
  \coordinate (D6) at (3.7,1.875);
  \draw [loosely dotted, line width = 1.0pt] (D5) -- (D6);
  \coordinate [label=above:{\Large $T$}] (E) at (2.3,2.2);
\end{tikzpicture}
	\caption{Division of $T$ into $n^2$ congruent small triangles.}
	\label{Fig2}
\end{figure}
We assume that each $\tau_k$ is an open set, namely it does not contain its boundary. Furthermore, we define
$$
	T'=\bigcup_{k=1}^{n^2}\tau_k.
$$
Now, we define $\Pi^{(\alpha)}u$ for $u\in H^1(T)$ and $\Pi^{(\beta)}u$ for $u\in H^2(T)$ as follows:
$$
	\Pi^{(\alpha)}u|_{\tau_k} = \Pi^{(\alpha)}_{\tau_k}u, \qquad
	\Pi^{(\beta)}u|_{\tau_k} = \Pi^{(\beta)}_{\tau_k}u.
$$
Note that, in general, $\Pi^{(\alpha)}u$ and $\Pi^{(\beta)}u$ cannot be extended as continuous functions on $T$.

We assume that the following values exist.
\begin{align*}
	C_1^{(n)}(T)&=\sup_{u\in V^{1,1}(T)\setminus 0}\frac{\|\Pi^{(\alpha)}u\|_{L^2(T')}}{\|\nabla \Pi^{(\alpha)}u\|_{L^2(T')}}, &
	C_2^{(n)}(T)&=\sup_{u\in V^{1,2}(T)\setminus 0}\frac{\|\Pi^{(\alpha)}u\|_{L^2(T')}}{\|\nabla \Pi^{(\alpha)}u\|_{L^2(T')}}, \\
	C_3^{(n)}(T)&=\sup_{u\in V^2(T)\setminus 0}\frac{\|\Pi^{(\beta)}u\|_{L^2(T')}}{|\Pi^{(\beta)}u|_{H^2(T')}}, &
	C_4^{(n)}(T)&=\sup_{u\in V^2(T)\setminus 0}\frac{\|\nabla \Pi^{(\beta)}u\|_{L^2(T')}}{|\Pi^{(\beta)}u|_{H^2(T')}}.
\end{align*} 
Since $\Pi^{(\alpha)}u$ and $\Pi^{(\beta)}u$ are piecewise polynomials on $T'$,
we can obtain these values by solving finite-dimensional generalized eigenvalue problems,
namely the upper bound of the Rayleigh quotient.

Concerning these constants, we have the following theorem.
\begin{theorem} \label{BoundbyFiniteResult}
\begin{align*}
	C_1(T)&\le\sqrt{\frac{n^2}{n^2-1}}\;C_1^{(n)}(T), &
	C_2(T)&\le\sqrt{\frac{n^2}{n^2-1}}\;C_2^{(n)}(T), \\
	C_3(T)&\le\sqrt{\frac{n^4}{n^4-1}}\;C_3^{(n)}(T), &
	C_4(T)&\le\sqrt{C_4^{(n)}(T)^2+\frac{C_2(T)^2}{n^2}}.
\end{align*}
\end{theorem}

\begin{proof}\quad
We first note that the scaling properties $C_j(\tau_k)=C_j(T)/n,\;j= 1, 2,4$ and
$C_3(\tau_k)=C_3(T)/n^2$ hold. This property can be easily shown by a change of variables.

From Lemma~\ref{projection1}, for $u\in V^{1,j}(T),\;\; j=1,2$, we have
\begin{align*}
	\|u\|_{L^2(T)}&\le \|\Pi^{(\alpha)}u\|_{L^2(T')}+\|u-\Pi^{(\alpha)}u\|_{L^2(T')} \\
	&=\|\Pi^{(\alpha)}u\|_{L^2(T')}
		+\sqrt{\sum_{k=1}^{n^2}\|u-\Pi^{(\alpha)}_{\tau_k}u\|_{L^2(\tau_k)}^2} \\
	&\le C_j^{(n)}(T)\;\|\nabla \Pi^{(\alpha)}u\|_{L^2(T')}
		+\frac{C_j(T)}{n}\sqrt{\sum_{k=1}^{n^2}\|\nabla(u-\Pi^{(\alpha)}_{\tau_k}u)\|_{L^2(\tau_k)}^2} \\
	&\le \sqrt{C_j^{(n)}(T)^2+\frac{C_j(T)^2}{n^2}}
		\sqrt{\sum_{k=1}^{n^2}\left(\|\nabla \Pi^{(\alpha)}_{\tau_k}u\|_{L^2(\tau_k)}^2
		+\|\nabla(u-\Pi^{(\alpha)}_{\tau_k}u)\|_{L^2(\tau_k)}^2\right)} \\
	&= \sqrt{C_j^{(n)}(T)^2+\frac{C_j(T)^2}{n^2}}
		\sqrt{\sum_{k=1}^{n^2}\|\nabla u\|_{L^2(\tau_k)}^2} \\
	&= \sqrt{C_j^{(n)}(T)^2+\frac{C_j(T)^2}{n^2}}\;\|\nabla u\|_{L^2(T)}.
\end{align*}
Furthermore, from Lemma~\ref{projection2}, for $u\in V^2(T)$,
\begin{align*}
	\|u\|_{L^2(T)}&\le \|\Pi^{(\beta)}u\|_{L^2(T')}+\|u-\Pi^{(\beta)}u\|_{L^2(T')} \\
	&= \|\Pi^{(\beta)}u\|_{L^2(T')}
		+\sqrt{\sum_{k=1}^{n^2}\|u-\Pi^{(\beta)}_{\tau_k}u\|_{L^2(\tau_k)}^2} \\
	&\le C_3^{(n)}(T)\;|\Pi^{(\beta)}u|_{H^2(T')}
		+\frac{C_3(T)}{n^2}\sqrt{\sum_{k=1}^{n^2}|u-\Pi^{(\beta)}_{\tau_k}u|_{H^2(\tau_k)}^2} \\
	&\le \sqrt{C_3^{(n)}(T)^2+\frac{C_3(T)^2}{n^4}}
		\sqrt{\sum_{k=1}^{n^2}\left(|\Pi^{(\beta)}_{\tau_k}u|_{H^2(\tau_k)}^2
		+|u-\Pi^{(\beta)}_{\tau_k}u|_{H^2(\tau_k)}^2\right)} \\
	&= \sqrt{C_3^{(n)}(T)^2+\frac{C_3(T)^2}{n^4}}
		\sqrt{\sum_{k=1}^{n^2}|u|_{H^2(\tau_k)}^2} \\
	&= \sqrt{C_3^{(n)}(T)^2+\frac{C_3(T)^2}{n^4}}\;|u|_{H^2(T)}
\end{align*}
and from Lemma~\ref{C2-gradient} and Lemma~\ref{projection2},
\begin{align*}
	\|\nabla u\|_{L^2(T)}&\le \|\nabla \Pi^{(\beta)}u\|_{L^2(T')}+\|\nabla(u-\Pi^{(\beta)}u)\|_{L^2(T')} \\
	&= \|\nabla \Pi^{(\beta)}u\|_{L^2(T')}
		+\sqrt{\sum_{k=1}^{n^2}\left(\|(u-\Pi^{(\beta)}_{\tau_k}u)_x\|_{L^2(\tau_k)}^2 + \|(u-\Pi^{(\beta)}_{\tau_k}u)_y\|_{L^2(\tau_k)}^2\right)} \\
	&\le C_4^{(n)}(T)\;|\Pi^{(\beta)}u|_{H^2(T')}  \\
	&\qquad\qquad +\frac{C_2(T)}{n}\sqrt{\sum_{k=1}^{n^2}\left(\|\nabla(u-\Pi^{(\beta)}_{\tau_k}u)_x\|_{L^2(\tau_k)}^2 + \|\nabla(u-\Pi^{(\beta)}_{\tau_k}u)_y\|_{L^2(\tau_k)}^2\right)} \\
	&=C_4^{(n)}(T)\;|\Pi^{(\beta)}u|_{H^2(T')}+\frac{C_2(T)}{n}\sqrt{\sum_{k=1}^{n^2}|u-\Pi^{(\beta)}_{\tau_k}u|_{H^2(\tau_k)}^2} \\
	&\le \sqrt{C_4^{(n)}(T)^2+\frac{C_2(T)^2}{n^2}}
		\sqrt{\sum_{k=1}^{n^2}\left(|\Pi^{(\beta)}_{\tau_k}u|_{H^2(\tau_k)}^2
		+|u-\Pi^{(\beta)}_{\tau_k}u|_{H^2(\tau_k)}^2\right)} \\
	&= \sqrt{C_4^{(n)}(T)^2+\frac{C_2(T)^2}{n^2}}
		\sqrt{\sum_{k=1}^{n^2}|u|_{H^2(\tau_k)}^2} \\
	&= \sqrt{C_4^{(n)}(T)^2+\frac{C_2(T)^2}{n^2}}\;|u|_{H^2(T)}
\end{align*}
holds. 
From the above evaluations, we have the following.
\begin{align*}
	C_1(T)&\le \sqrt{C_1^{(n)}(T)^2+\frac{C_1(T)^2}{n^2}}, &
	C_2(T)&\le \sqrt{C_2^{(n)}(T)^2+\frac{C_2(T)^2}{n^2}}, \\
	C_3(T)&\le \sqrt{C_3^{(n)}(T)^2+\frac{C_3(T)^2}{n^4}}, &
	C_4(T)&\le \sqrt{C_4^{(n)}(T)^2+\frac{C_2(T)^2}{n^2}},
\end{align*}
which leads us to the conclusion.
\end{proof}

This result shows that we can bound the constants $C_j(T)$ using $C_j^{(n)}(T)$.
For specific $T$ and $n$, we can compute the upper bounds of $C_j^{(n)}(T)$ numerically
and obtain guaranteed results via a numerical verification method.

\section{Construction of finite-dimensional generalized eigenvalue problems} \label{Construction_FDGEP}
By the definitions of $C_j^{(n)}(T)$,  they can be evaluated
 by bounding the eigenvalues of the corresponding finite-dimensional generalized eigenvalue problem.
First, we describe how to construct the corresponding finite-dimensional generalized eigenvalue problem
 for a particular fixed triangle.

Lets $s$ be the area of the triangle $\tau$ and
$$
	c_1 = \frac{|\gamma_1|^2}{s},
	\qquad c_2 = \frac{|\gamma_2|^2}{s},
	\qquad c_3 = \frac{|\gamma_3|^2}{s}.
$$
Then, the following two lemmas hold.
\begin{lemma} \label{F1}
Let
\begin{align}
	w_k &= \frac{1}{|\gamma_k|}\int_{\gamma_k}\Pi^{(\alpha)}_\tau u\,ds,\qquad k=1,2,3, \label{F1-1} \\
	u_0 &= \frac{1}{s}\iint_\tau \Pi^{(\alpha)}_\tau u\,dxdy. \label{F1-2}
\end{align}
Then, it holds that
\begin{align}
	\|\Pi^{(\alpha)}_\tau u\|_{L^2(\tau)}^2&=F^{(\alpha)}_0(s,w_1,w_2,w_3,u_0), \label{F1-3} \\
	\|\nabla \Pi^{(\alpha)}_\tau u\|_{L^2(\tau)}^2&=F^{(\alpha)}_1(s,w_1,w_2,w_3,u_0), \label{F1-4}
\end{align}
where 
\begin{align*}
	F^{(\alpha)}_0(s,w_1,&w_2,w_3,u_0) \\
	&=\frac{s}{15}\left\{\frac{8(c_1^2+c_2^2+c_3^2)(w_1+w_2+w_3-3u_0)^2}{(c_1+c_2+c_3)^2}\right. \\
	&\qquad\quad
		\left.\phantom{\left(\frac{c_1^2}{c_1^2}\right)}
		-\frac{2(w_1+w_2+w_3-3u_0)v_0}{c_1+c_2+c_3}+ 5\big(w_1^2+w_2^2+w_3^2\big) \right\}, \\
	v_0&=\big(3c_2+3c_3-c_1\big)w_1+\big(3c_3+3c_1-c_2\big)w_2+\big(3c_1+3c_2-c_3\big)w_3, \\
	F^{(\alpha)}_1(s,w_1,&w_2,w_3,u_0) \\
	&=\frac{1}{2}\left\{\frac{32(w_1+w_2+w_3-3u_0)^2}{c_1+c_2+c_3}+\big(c_1+c_2-c_3\big)\big(w_1-w_2\big)^2\right. \\
	&\quad \left.\phantom{\frac{32(w_3)^2}{c_1}}
		+\big(c_2+c_3-c_1\big)\big(w_2-w_3\big)^2+\big(c_3+c_1-c_2\big)\big(w_3-w_1\big)^2\right\}.
\end{align*}
\end{lemma}
\begin{proof}\quad
Since $Q_\alpha$ is invariant under constant shifts and rotations,
we take $p_1=(0,0),\;p_2=(h,0),\;p_3=(ah,bh)$. Then, it holds that
$$
	s = \frac{bh^2}{2},
	\qquad c_1 = \frac{2((1-a)^2+b^2)}{b},
	\qquad c_2 = \frac{2(a^2+b^2)}{b},
	\qquad c_3 = \frac{2}{b}
$$
and $\Pi^{(\alpha)}_\tau u$ satisfying \eqref{F1-1} and \eqref{F1-2} is obtained as follows:
\begin{align*}
	\Pi^{(\alpha)}_\tau u= \frac{1}{bh}&\left\{
		\Big( b(2x-h)+2(1-a)y \Big)w_1 - \Big( b(2x-h)-2ay\Big)w_2 - (2y-bh)w_3\right\} \\
	& + \frac{2(3(x^2+y^2) - 2(1+a)hx - 2bhy + ah^2)}{(1-a+a^2+b^2)h^2}\Big(w_1+w_2+w_3-3u_0\Big).
\end{align*}
Then, the proof is shown from the uniqueness of $\Pi^{(\alpha)}_\tau u$ and the fact that $\Pi^{(\alpha)}_\tau u$
satisfies \eqref{F1-3} and \eqref{F1-4}.
\end{proof}

\begin{lemma} \label{F2}
Let
\begin{align}
	u_k&=\Pi^{(\beta)}_\tau u(p_k),\qquad k=1,2,3, \label{F2-1} \\
	w_k&=\int_{\gamma_k}\nabla \Pi^{(\beta)}_\tau u\cdot n\,ds,\qquad k=1,2,3. \label{F2-2}
\end{align}
Then it holds that
\begin{align}
	\|\Pi^{(\beta)}_\tau u\|_{L^2(\tau)}^2&=F^{(\beta)}_0(s,u_1,u_2,u_3,w_1,w_2,w_3), \label{F2-3} \\
	\|\nabla \Pi^{(\beta)}_\tau u\|_{L^2(\tau)}^2&=F^{(\beta)}_1(s,u_1,u_2,u_3,w_1,w_2,w_3), \label{F2-4} \\
	|\Pi^{(\beta)}_\tau u|_{H^2(\tau)}^2&=F^{(\beta)}_2(s,u_1,u_2,u_3,w_1,w_2,w_3), \label{F2-5}
\end{align}
where
\begin{align*}
	F^{(\beta)}_0(s,u_1,u_2,&u_3,w_1,w_2,w_3)
		=\frac{s}{720}\Big\{
		21\big(u_1^2+u_2^2+u_3^2\big)
		-6\big(u_1u_2+u_2u_3+u_3u_1\big) \\
	&\qquad\qquad
		+6\big({v_1}^2+{v_2}^2+{v_3}^2\big)
		+10\big(v_1v_2+v_2v_3+v_3v_1\big)\Big\}, \\
	v_1&=\frac{13u_1+u_2+u_3}{4}-\frac{4w_1-(c_2-c_3)(u_2-u_3)}{c_1}, \\
	v_2&=\frac{13u_2+u_3+u_1}{4}-\frac{4w_2-(c_3-c_1)(u_3-u_1)}{c_2}, \\
	v_3&=\frac{13u_3+u_1+u_2}{4}-\frac{4w_3-(c_1-c_2)(u_1-u_2)}{c_3}, \\
	F^{(\beta)}_1(s,u_1,u_2,&u_3,w_1,w_2,w_3) \\
		&=\frac{1}{3}\left(
		\frac{(u_2-u_3)^2+w_1^2}{c_1}
		+\frac{(u_3-u_1)^2+w_2^2}{c_2}
		+\frac{(u_1-u_2)^2+w_3^2}{c_3}
	\right), \\
	F^{(\beta)}_2(s,u_1,u_2,&u_3,w_1,w_2,w_3)
		=\frac{1}{16s}\Big\{\big(c_1v_4+c_2v_5+c_3v_6\big)^2
		-8\big(v_4v_5+v_5v_6+v_6v_4\big)\Big\}, \\
	v_4&=2u_1-u_2-u_3+\frac{4w_1-(c_2-c_3)(u_2-u_3)}{c_1}, \\
	v_5&=2u_2-u_3-u_1+\frac{4w_2-(c_3-c_1)(u_3-u_1)}{c_2}, \\
	v_6&=2u_3-u_1-u_2+\frac{4w_3-(c_1-c_2)(u_1-u_2)}{c_3}.
\end{align*}            
\end{lemma}
\begin{proof}\quad
Since $Q_\beta$ is invariant under constant shifts and rotations, as in the previous lemma,
we take $p_1=(0,0),\;p_2=(h,0),\;p_3=(ah,bh)$. Then, it holds that
$$
	s = \frac{bh^2}{2},
	\qquad c_1 = \frac{2((1-a)^2+b^2)}{b},
	\qquad c_2 = \frac{2(a^2+b^2)}{b},
	\qquad c_3 = \frac{2}{b}
$$
and $\Pi^{(\beta)}_\tau u$ satisfying \eqref{F2-1} and \eqref{F2-2} is obtained as follows:
\begin{align*}
	\Pi^{(\beta)}_\tau u= &\frac{1}{bh}\left\{-\Big(b(x-h)+(1-a)y\Big)u_1 + (bx-ay)u_2 + yu_3 \right\} \\
	&+ \frac{(b(x-h)+(1-a)y)(bx+(1-a)y)}{((1-a)^2+b^2)b^2h^2} \\
	&\qquad\qquad\qquad \times\Big( bw_1 + ((1-a)^2+b^2)u_1 + (a(1-a)-b^2)u_2 - (1-a)u_3 \Big) \\
	&+ \frac{(b(x-h)-ay)(bx-ay)}{(a^2+b^2)b^2h^2}\Big( bw_2 + (a(1-a)-b^2)u_1 + (a^2+b^2)u_2 - au_3 \Big) \\
	&+ \frac{(y-bh)y}{h^2b^2}\Big( bw_3 - (1-a)u_1 - au_2 + u_3 \Big).
\end{align*}
Then, the proof is shown from the uniqueness of $\Pi^{(\beta)}_\tau u$ and the fact that $\Pi^{(\beta)}_\tau u$
 satisfies \eqref{F2-3}, \eqref{F2-4}, and \eqref{F2-5}.
\end{proof}
The proofs of Lemma~\ref{F1} and Lemma~\ref{F2} can be checked by hand calculation in principle.
However, it is more realistic to verify them using a formula manipulation system.
Code~\ref{Lemma5-1} and Code~\ref{Lemma5-2} in Appendix~\ref{Appendix2} show the verification codes developed
using MATLAB with Symbolic Math Toolbox.

\begin{figure}[t]
	\centering
\begin{tikzpicture}[scale=1.2]
  \coordinate (A) at (0.0,0.0);
  \coordinate (B) at (5.0,0.0);
  \coordinate (C) at (3.5,3.0);
  \draw [line width = 0.8pt](A) -- (B) -- (C) -- cycle;
  \coordinate (A2) at ($(A)!0.5!(B)$);
  \coordinate (B2) at ($(B)!0.5!(C)$);
  \coordinate (C2) at ($(C)!0.5!(A)$);
  \draw [line width = 0.8pt] (A2) -- (B2) -- (C2) -- cycle;
  \coordinate (A1) at ($(A)!0.5!(A2)$);
  \draw (A1) circle [radius=0.07];
  \draw (A1) node[below=1mm]{$w_3$};
  \coordinate (A3) at ($(B)!0.5!(A2)$);
  \draw (A3) circle [radius=0.07];
  \draw (A3) node[below=1mm]{$w_6$};
  \coordinate (B1) at ($(B)!0.5!(B2)$);
  \draw (B1) circle [radius=0.07];
  \draw (B1) node[right=0.2mm]{$w_1$};
  \coordinate (B3) at ($(C)!0.5!(B2)$);
  \draw (B3) circle [radius=0.07];
  \draw (B3) node[right=0.2mm]{$w_4$};
  \coordinate (C1) at ($(C)!0.5!(C2)$);
  \draw (C1) circle [radius=0.07];
  \draw (C1) node[left=2.7mm, above=-0.1mm]{$w_2$};
  \coordinate (C3) at ($(A)!0.5!(C2)$);
  \draw (C3) circle [radius=0.07];
  \draw (C3) node[left=2.7mm, above=-0.1mm]{$w_5$};
  \coordinate (D1) at ($(A2)!0.5!(B2)$);
  \draw (D1) circle [radius=0.07];
  \draw (D1) node[right=2.2mm,below=0.1mm]{$w_8$};
  \coordinate [label=above:{$w_9$}] (D2) at ($(C2)!0.5!(B2)$);
  \draw (D2) circle [radius=0.07];
  \coordinate (D3) at ($(A2)!0.5!(C2)$);
  \draw (D3) circle [radius=0.07];
  \draw (D3) node[left=2.1mm,below=-0.1mm]{$w_7$};
  \coordinate (E1) at ($0.3333*($($(B2)+(A2)$)+(C2)$)$);
  \draw (E1) circle [radius=0.07];
  \draw (E1) node[below=0.25mm]{$u_4$};
  \coordinate (E2) at ($0.3333*($($(A)+(A2)$)+(C2)$)$);
  \draw (E2) circle [radius=0.07];
  \draw (E2) node[above=0.25mm]{$u_1$};
  \coordinate (E3) at ($0.3333*($($(A2)+(B)$)+(B2)$)$);
  \draw (E3) circle [radius=0.07];
  \draw (E3) node[below=2.0mm,right=-0.1mm]{$u_2$};
  \coordinate (E4) at ($0.3333*($($(C2)+(C)$)+(B2)$)$);
  \draw (E4) circle [radius=0.07];
  \draw (E4) node[above=1.0mm,right=-0.2mm]{$u_3$};
  \coordinate [label=above:{\Large $T$}] (E) at (1.8,1.7);
\end{tikzpicture}
	\caption{Variable specification for $C_1^{(n)}(T)$ and $C_2^{(n)}(T)$ when $n=2$.}
	\label{Fig3}
\end{figure}

We now explain the practical construction of the finite-dimensional eigenvalue problem.
Divide the given triangle $T$ into $n^2$ similar small triangles $\tau_1,\tau_2,\dots,\tau_{n^2}$,
as shown in Fig.~\ref{Fig2}, and let $S$ be the area of triangle $T$.
Under the constraints, for $u\in V^{1,j},\;j=1,2$, the problem of finding $C_1^{(n)}(T)$ is a generalized eigenvalue problem
in $(n+1)(5n-2)/2$ dimensions and the problem of finding $C_2^{(n)}(T)$ is a generalized eigenvalue problem in
$(5n^2+3n-6)/2$ dimensions.
We show an example for $n=2$.
Take $u_1,u_2,\dots,u_4$ and $w_1,w_2,\dots,w_9$, as in Fig.~\ref{Fig3}.
For $u\in V^{1,1}$, with the constraint
$$
	u_1+u_2+u_3+u_4=0
$$
taken into account, we obtain
\begin{align*}
	\|\Pi^{(\alpha)}u&\|_{L^2(T')}^2=F^{(\alpha)}_0(S/4, w_7, w_5, w_3, u_1)
		+ F^{(\alpha)}_0(S/4, w_1, w_8, w_6, u_2) \\
	&\qquad\qquad + F^{(\alpha)}_0(S/4, w_4, w_2, w_9, u_3)
     + F^{(\alpha)}_0(S/4, w_7, w_8, w_9, -u_1-u_2-u_3), \\
	\|\nabla\Pi^{(\alpha)}&u\|_{L^2(T')}^2=F^{(\alpha)}_1(S/4, w_7, w_5, w_3, u_1)
		+ F^{(\alpha)}_1(S/4, w_1, w_8, w_6, u_2) \\
	&\qquad\qquad + F^{(\alpha)}_1(S/4, w_4, w_2, w_9, u_3)
     + F^{(\alpha)}_1(S/4, w_7, w_8, w_9, -u_1-u_2-u_3).
\end{align*}

For $u\in V^{1,2}$, with the constraints
$$
	w_1+w_4=w_2+w_5=w_3+w_6=0
$$
taken into account, we obtain
\begin{align*}
	\|\Pi^{(\alpha)}u\|_{L^2(T')}^2&=F^{(\alpha)}_0(S/4, w_7, -w_2, w_3, u_1)
		+ F^{(\alpha)}_0(S/4, w_1, w_8, -w_3, u_2) \\
	&\qquad\qquad + F^{(\alpha)}_0(S/4, -w_1, w_2, w_9, u_3)
     + F^{(\alpha)}_0(S/4, w_7, w_8, w_9, u_4), \\
	\|\nabla\Pi^{(\alpha)}u\|_{L^2(T')}^2&=F^{(\alpha)}_1(S/4, w_7, -w_2, w_3, u_1)
		+ F^{(\alpha)}_1(S/4, w_1, w_8, -w_3, u_2) \\
	&\qquad\qquad + F^{(\alpha)}_1(S/4, -w_1, w_2, w_9, u_3)
     + F^{(\alpha)}_1(S/4, w_7, w_8, w_9, u_4).
\end{align*}

\begin{figure}[t]
	\centering
\begin{tikzpicture}[scale=1.2]
  \coordinate (A) at (0.0,0.0);
  \coordinate (B) at (5.0,0.0);
  \coordinate (C) at (3.5,3.0);
  \draw [line width = 0.8pt](A) -- (B) -- (C) -- cycle;
  \coordinate (A2) at ($(A)!0.5!(B)$);
  \coordinate (B2) at ($(B)!0.5!(C)$);
  \coordinate (C2) at ($(C)!0.5!(A)$);
  \draw [line width = 0.8pt] (A2) -- (B2) -- (C2) -- cycle;
  \draw [fill] (A) circle [radius=0.07];
  \draw [fill] (B) circle [radius=0.07];
  \draw [fill] (C) circle [radius=0.07];
  \draw [fill] (A2) circle [radius=0.07];
  \draw [fill] (B2) circle [radius=0.07];
  \draw [fill] (C2) circle [radius=0.07];
  \coordinate (A1) at ($(A)!0.25!(B)$);
  \coordinate (A3) at ($(A)!0.75!(B)$);
  \coordinate (B1) at ($(B)!0.25!(C)$);
  \coordinate (B3) at ($(B)!0.75!(C)$);
  \coordinate (C1) at ($(C)!0.25!(A)$);
  \coordinate (C3) at ($(C)!0.75!(A)$);
  \coordinate (D1) at ($(A2)!0.5!(B2)$);
  \coordinate (D2) at ($(B2)!0.5!(C2)$);
  \coordinate (D3) at ($(A2)!0.5!(C2)$);
  \draw (A) node[below=1.7mm,left=-0.2mm]{$u_4$};
  \draw (B) node[below=1.7mm,right=-0.2mm]{$u_5$};
  \draw (C) node[above=0.7mm]{$u_6$};
  \draw (A2) node[below=0.7mm]{$u_3$};
  \draw (B2) node[right=0.7mm]{$u_1$};
  \draw (C2) node[left=4mm,above=-1mm]{$u_2$};
 \node [draw, inner sep=0.5mm, single arrow, minimum height=5mm, 
         single arrow head extend=1mm, 
         rotate=-90] at (A1) {};
  \draw (A1) node[above=2mm,left=-0.2mm]{$w_3$};
 \node [draw, inner sep=0.5mm, single arrow, minimum height=5mm, 
         single arrow head extend=1mm, 
         rotate=-90] at (A3) {};
  \draw (A3) node[above=2mm,left=-0.2mm]{$w_6$};
 \node [draw, inner sep=0.5mm, single arrow, minimum height=5mm, 
         single arrow head extend=1mm, 
         rotate=-90] at (D2) {};
  \draw (D2) node[above=2mm,left=-0.2mm]{$w_9$};
 \node [draw, inner sep=0.5mm, single arrow, minimum height=5mm, 
         single arrow head extend=1mm, 
         rotate=27] at (B1) {};
  \draw ($(B1)-(0.1,0.3)$) node {$w_1$};
 \node [draw, inner sep=0.5mm, single arrow, minimum height=5mm, 
         single arrow head extend=1mm, 
         rotate=27] at (B3) {};
  \draw ($(B3)-(0.1,0.3)$) node {$w_4$};
 \node [draw, inner sep=0.5mm, single arrow, minimum height=5mm, 
         single arrow head extend=1mm, 
         rotate=27] at (D3) {};
  \draw ($(D3)-(0.1,0.3)$) node {$w_7$};
 \node [draw, inner sep=0.5mm, single arrow, minimum height=5mm, 
         single arrow head extend=1mm, 
         rotate=131] at (C1) {};
  \draw (C1) node[right=0.95mm]{$w_2$};
 \node [draw, inner sep=0.5mm, single arrow, minimum height=5mm, 
         single arrow head extend=1mm, 
         rotate=131] at (C3) {};
  \draw (C3) node[right=0.95mm]{$w_5$};
 \node [draw, inner sep=0.5mm, single arrow, minimum height=5mm, 
         single arrow head extend=1mm, 
         rotate=131] at (D1) {};
  \draw (D1) node[right=0.95mm]{$w_8$};
  \coordinate [label=above:{\Large $T$}] (E) at (1.4,1.9);
\end{tikzpicture}
	\caption{Variable specification for $C_3^{(n)}(T)$ and $C_4^{(n)}(T)$ when $n=2$.}
	\label{Fig4}
\end{figure}

Furthermore, for $u\in V^2$, the problem of finding $C_3^{(n)}(T)$ and $C_4^{(n)}(T)$ becomes a generalized eigenvalue problem
 in $(n+2)(2n-1)$ dimensions.
We show an example for $n=2$.
Take $u_1,u_2,\dots,u_6$ and $w_1,w_2,\dots,w_9$, as in Fig.~\ref{Fig4}.
With the constraint
$$
	u_4=u_5=u_6=0
$$
taken into account, we obtain
\begin{align*}
	\|\Pi^{(\beta)}u&\|_{L^2(T')}^2=F^{(\beta)}_0(S/4, 0, u_3, u_2, w_7, w_5, w_3)
		+ F^{(\beta)}_0(S/4, u_3, 0, u_1, w_1, w_8, w_6) \\
	&\qquad + F^{(\beta)}_0(S/4, u_2, u_1, 0, w_4, w_2, w_9)
		+ F^{(\beta)}_0(S/4, u_1, u_2, u_3, -w_7, -w_8, -w_9), \\
	\|\nabla \Pi^{(\beta)}&u\|_{L^2(T')}^2=F^{(\beta)}_1(S/4, 0, u_3, u_2, w_7, w_5, w_3)
		+ F^{(\beta)}_1(S/4, u_3, 0, u_1, w_1, w_8, w_6) \\
	&\qquad + F^{(\beta)}_1(S/4, u_2, u_1, 0, w_4, w_2, w_9)
		+ F^{(\beta)}_1(S/4, u_1, u_2, u_3, -w_7, -w_8, -w_9), \\
	|\Pi^{(\beta)}u&|_{H^2(T')}^2=F^{(\beta)}_2(S/4, 0, u_3, u_2, w_7, w_5, w_3)
		+ F^{(\beta)}_2(S/4, u_3, 0, u_1, w_1, w_8, w_6) \\
	&\qquad + F^{(\beta)}_2(S/4, u_2, u_1, 0, w_4, w_2, w_9)
		+ F^{(\beta)}_2(S/4, u_1, u_2, u_3, -w_7, -w_8, -w_9).
\end{align*}

For example, if we take
\begin{align*}
	G_1&=F^{(\beta)}_1(S/4, 0, u_3, u_2, w_7, w_5, w_3)
		+ F^{(\beta)}_1(S/4, u_3, 0, u_1, w_1, w_8, w_6) \\
	&\qquad\qquad + F^{(\beta)}_1(S/4, u_2, u_1, 0, w_4, w_2, w_9)
		+ F^{(\beta)}_1(S/4, u_1, u_2, u_3, -w_7, -w_8, -w_9), \\
	G_2&=F^{(\beta)}_2(S/4, 0, u_3, u_2, w_7, w_5, w_3)
		+ F^{(\beta)}_2(S/4, u_3, 0, u_1, w_1, w_8, w_6) \\
	&\qquad\qquad + F^{(\beta)}_2(S/4, u_2, u_1, 0, w_4, w_2, w_9)
		+ F^{(\beta)}_2(S/4, u_1, u_2, u_3, -w_7, -w_8, -w_9),
\end{align*}
we have
\begin{equation}
	C_4^{(2)}(T)^2=\sup_{\substack{x\in\mathbb{R}^{12} \\ x\not=0\;\;\;}}\frac{G_1}{G_2},
	\quad x=(u_1,u_2,u_3,w_1,w_2,w_3,w_4,w_5,w_6,w_7,w_8,w_9). \label{C_4uw}
\end{equation}
Similarly, every $C_j^{(n)}(T)$ can be expressed as a solution to a finite-dimensional eigenvalue problem.

From a mathematical point of view, we are interested in whether $C_j^{(n)}(T)$ converges to
 $C_j(T)$, respectively, as $n$ increases; however, we cannot provide a proof yet.
Nevertheless, from the numerical results in Section~\ref{Numerical_Results}, they seem to converge.

\section{Numerical verification using INTLAB} \label{Positive_Definiteness}
By using the results of Lemma~\ref{F1} and Lemma~\ref{F2} in the previous section,
 we were able to construct a concrete finite-dimensional generalized eigenvalue problem.
Now, we evaluate the upper bounds of these eigenvalues.
For this purpose, it is sufficient to know how to evaluate the upper bound of
$$
	\lambda=\sup_{x\not=0}\frac{x^T Ax}{x^T Bx}
$$
strictly for the real symmetric matrix $A$ and the real symmetric positive definite matrix $B$.
Now, since
\begin{equation}
	\lambda'>\lambda \qquad \Longleftrightarrow \qquad \lambda'B-A \quad \text{is positive definite}, \label{positive_definite}
\end{equation}
to show that a given real number $\lambda'$ is an upper bound of the eigenvalues of a finite-dimensional
 generalized eigenvalue problem, it is sufficient to show the positive definiteness of the real symmetric matrix $\lambda'B-A$.

There are various ways to perform numerical verification on a computer.
For portability and ease of programming, we utilized INTLAB, which is a toolbox for MATLAB.
For more information on INTLAB, see \cite{Rump2,MooreKearfottCloud}.
In the numerical verification method, variables are given as intervals.
In INTLAB, the function {\tt isspd()} is used to verify the positive definiteness of a symmetric matrix
whose components are given as intervals.

Suppose that the set of symmetric matrices is given using intervals as
$$
	A=\left(\left[\overline{a_{ij}},\;\underline{a_{ij}}\right]\right)
		=\left\{Y=(y_{ij})\;\Big|\;y_{ij}\in\left[\overline{a_{ij}},\;\underline{a_{ij}}\right],\; y_{ij}\in\mathbb{R}\right\}.
$$
Then, if the returned value of {\tt isspd(A)} is 1, all real symmetric matrices contained in $A$ are strictly verified to be positive definite.
If they cannot be verified to be positive definite, then {\tt isspd(A)} returns 0.
See {\cite{Rump1} for the algorithm used inside the function {\tt isspd()}.

Using the method described above, we can obtain mathematically rigorous proofs of the following two theorems.
\begin{theorem} \label{Verified_Result1}
Define $\{x_k\}$ and $\{y_k\}$ as follows:
$$
	y_1=1000, \qquad y_{k+1}=\left\lfloor \frac{51}{50}y_k \right\rfloor,
	\qquad x_{k}= \left\lceil \frac{y_k}{40} \right\rceil,
$$
where $\lfloor\cdot\rfloor$ is the floor function and $\lceil\cdot\rceil$ is the ceiling function.
Then, for
$$
	n=20, \qquad b_k = \frac{1000}{y_k},\qquad a_{kl} = \frac{l}{2x_k},
$$
the following inequalities hold for all $k=1,2,\dots,119,\quad l=0,1,\dots,x_k$.
\begin{align*}
	\frac{n^2}{n^2-1}C_j^{(n)}(T_{a_{kl},b_k})^2&\left(1+5\cdot\frac{1}{50^2}\right)\left(1+3\cdot\frac{1}{50^2}\right)< L_j(a_{kl},b_k), \qquad j=1,2, \\
	\frac{n^4}{n^4-1}C_3^{(n)}(T_{a_{kl},b_k})^2&\left(1+6\cdot\frac{1}{50^2}\right)\left(1+8\cdot\frac{1}{50^2}\right)< L_3(a_{kl},b_k), \\
	\left(C_4^{(n)}(T_{a_{kl},b_k})^2+\frac{L_2(a_{kl},b_k)}{n^2}\right)&\left(1+9\cdot\frac{1}{50^2}\right)\left(1+9\cdot\frac{1}{50^2}\right)< L_4(a_{kl},b_k).
\end{align*}
\end{theorem}
Figure~\ref{Points} shows the distribution of $(a_{kl}, b_k)$.

\begin{figure}[t]
	\centering
	\includegraphics[width=0.3\hsize]{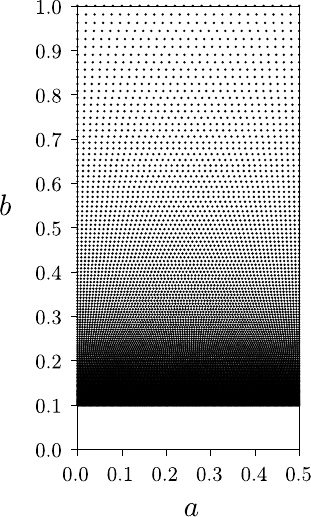}
	\caption{Distribution of $(a_{kl}, b_k)$.}
	\label{Points}
\end{figure}

\begin{theorem} \label{Verified_Result2}
For
$$
	n=20, \qquad b = \frac{1}{10},\qquad a_l = \frac{l}{500},
$$
the following inequalities hold for all $l=0,1,\dots, 250$.
\begin{align*}
	\frac{n^2}{n^2-1}C_j^{(n)}(T_{a_l,b})^2&\left(1+5\cdot\frac{1}{50^2}\right)< \lim_{y\downarrow0}L_j(a_l,y), \quad j=1,2, \\
	\frac{n^4}{n^4-1}C_3^{(n)}(T_{a_l,b})^2&\left(1+6\cdot\frac{1}{50^2}\right)< \lim_{y\downarrow0}L_3(a_l,y).
\end{align*}
Here, we remark that $\lim_{y\downarrow0}L_j(a,y), \; j=1,2,3$ coincide with the expression
obtained by formally substituting $b=0$ into $L_j(a,b)$.
\end{theorem}

Theorem~\ref{Verified_Result1} and Theorem~\ref{Verified_Result2} were confirmed on
both MATLAB R2019b with INTLAB version 12 and MATLAB R2024b with INTLAB version 12.
The codes used for verification are listed as Code~\ref{C1verify}, Code~\ref{C2verify}, Code~\ref{C3verify}, and Code~\ref{C4verify}
 in Appendix~\ref{Appendix2}.

Again, we want to emphasize that in the proofs of Theorem~\ref{Verified_Result1} and Theorem~\ref{Verified_Result2},
we do not compute the eigenvalues directly, but rather transform the inequality about $C_j^{(n)}$
into the form of equation \eqref{positive_definite} and show the positive definiteness of the matrix.

These results are only valid for a finite number of triangles. We extend them to results for arbitrary triangles
 in the next section.

\section{Extension to continuously distributed triangles} \label{Continuification}
Here, we prove the formulas in Theorem~\ref{formula} for an arbitrary triangle $T$.
Since the formulas have scale invariance, we can set $p_1=(0,0),\;p_2=(1,0),\;p_3=(a,b)$.
Under the assumption $\overline{p_1p_2}\ge\overline{p_2p_3}\ge\overline{p_3p_1}$,
as shown in Fig.~\ref{Range_ab}, it is sufficient to prove the formulas for
$T_{a,b}$ satisfying $0\le a\le \dfrac{1}{2},\; 0<b\le1$.

\begin{figure}[t]
	\centering
\begin{tikzpicture}[scale=1.0]
 \fill[fill=black!15!white](0,0)arc[start angle=180, end angle=120, x radius=4, y radius=4] -- (2,0);
 \draw[semithick,->,>=stealth](-0.5,0)--(4.5,0);
 \draw[semithick,->,>=stealth](0,-0.5)--(0,4.5);
 \draw[thick](0,0)arc[start angle=180, end angle=120, x radius=4, y radius=4];
 \draw[thick, dashed](4,4)arc[start angle=90, end angle=120, x radius=4, y radius=4];
 \draw[thick](2,0)--(2,3.4641);
 \draw[thick, dashed](2,3.4641)--(2,4);
 \draw[thick, dashed](4,0)--(4,4)--(0,4);
 \coordinate (A) at (0,0);
 \coordinate (B) at (4,0);
 \coordinate (C) at (1.4,1.8);
 \draw [line width = 0.8pt](A) -- (B) -- (C) -- cycle;
 \draw[fill] (A) circle [radius=0.07] node[right=3mm,below=0.7mm]{$p_1$} node[left=3mm,below=0.7mm]{$0$};
 \draw[fill] (B) circle [radius=0.07] node[below=0.7mm]{$p_2$};
 \draw[fill] (C) circle [radius=0.07] node[right=0.8mm,below=1.6mm]{$p_3$} node[above=0.2mm]{$(a,b)$};
 \draw(2,0) node[below]{$\frac{1}{2}$};
 \draw(1.4,0.6) node{$T_{a,b}$};
 \draw(0,4) node[left]{$1$};
\end{tikzpicture}
	\caption{Range of $(a,b)$ for $\overline{p_1p_2}\ge\overline{p_2p_3}\ge\overline{p_3p_1}$.}
	\label{Range_ab}
\end{figure}

In this section, we assume
\begin{align*}
	0\le a_1\le &a\le a_2\le\frac{1}{2}, & h_a &= \frac{a_2-a_1}{b},& 0&<h_a\le\frac{1}{50}, \\
	0< b_1\le &b\le b_2\le 1, & h_b &= \frac{b_2-b_1}{b_1},& 0&<h_b\le\frac{1}{50}.
\end{align*}
Note that $a_1,a_2,b_1,b_2$ are different from $a_k$ and $b_k$ used in Theorem~\ref{Verified_Result1}
and Theorem~\ref{Verified_Result2}.

In the following, we show that when evaluations of $C_j(T_{a,b})$
have been obtained for two triangles with slightly different shapes, we can also obtain a slightly milder evaluation
for triangles with intermediate shapes.

Specifically, this section proves the following two theorems.
\begin{theorem} \label{a-interpolation}
It holds that
\begin{align*}
	\frac{C_j(T_{a,b})^2}{L_j(a,b)}&\le (1+5h_a^2)\max_{k=1,2}\frac{C_j(T_{a_k,b})^2}{L_j(a_k,b)}, \qquad j=1,2, \\
	\frac{C_3(T_{a,b})^2}{L_3(a,b)}&\le (1+6h_a^2)\max_{k=1,2}\frac{C_3(T_{a_k,b})^2}{L_3(a_k,b)}, \\
	\frac{C_4(T_{a,b})^2}{L_4(a,b)}&\le (1+9h_a^2)\max_{k=1,2}\frac{C_4(T_{a_k,b})^2}{L_4(a_k,b)}.
\end{align*}
\end{theorem}

\begin{theorem} \label{b-interpolation}
It holds that
\begin{align*}
	\frac{C_j(T_{a,b})^2}{L_j(a,b)}&\le (1+3h_b^2)\max_{k=1,2}\frac{C_j(T_{a,b_k})^2}{L_j(a,b_k)}, \qquad j=1,2, \\
	\frac{C_3(T_{a,b})^2}{L_3(a,b)}&\le (1+8h_b^2)\max_{k=1,2}\frac{C_3(T_{a,b_k})^2}{L_3(a,b_k)}, \\
	\frac{C_4(T_{a,b})^2}{L_4(a,b)}&\le (1+9h_b^2)\max_{k=1,2}\frac{C_4(T_{a,b_k})^2}{L_4(a,b_k)}.
\end{align*}
\end{theorem}
These two theorems enable us to extend the results presented for a finite number of triangles to continuously distributed triangles.
The following lemma is used to prove Theorem~\ref{a-interpolation} and Theorem~\ref{b-interpolation}.
\begin{lemma} \label{BoundConstants}
For $j=1,2,3,4$, it holds that
\begin{align*}
	\left|\frac{\partial L_j}{\partial a}(a,b)\right|&< \frac{\alpha_{j1}}{b}L_j(a,b),
		& \frac{\partial^2 L_j}{\partial a^2}(\widetilde{a},b)&<\frac{\alpha_{j2}}{b^2}L_j(a,b), \\
	\left|\frac{\partial L_j}{\partial b}(a,b)\right|&< \frac{\beta_{j1}}{b}L_j(a,b),
		& \frac{\partial^2 L_j}{\partial b^2}(a,\widetilde{b})&<\frac{\beta_{j2}}{b^2}L_j(a,b),
\end{align*}
where
$$
	|\widetilde{a}-a|\le \frac{b}{50},\qquad |\widetilde{b}-b|\le \frac{b}{50}
$$
and
\begin{align*}
	\alpha_{j1}&=2, & \alpha_{j2}&=5,  & \beta_{j1}&=2, & \beta_{j2}&=4,  & j=1,2, \\
	\alpha_{31}&=2, & \alpha_{32}&=4, & \beta_{31}&=3, & \beta_{32}&=8, \\
	\alpha_{41}&=3, & \alpha_{42}&=9, & \beta_{41}&=3, & \beta_{42}&=9.
\end{align*}
\end{lemma}
The proof of this lemma comes down to showing the positivity of the multivariate polynomials;
since there is no intrinsic difficulty, it is presented in Appendix~\ref{Appendix1} separately for
 Lemma~\ref{L1-bound} to Lemma~\ref{L4-bound}.

\begin{proof}[Proof of Theorem~\ref{a-interpolation}]\quad
In this proof, for $k=1,2$, we take
$$
	u^{(k)}(x,y)=u\left(x+\frac{a-a_k}{b}y,y\right).
$$

For $u\in L^2(T_{a,b})$, from
$$
	\|u^{(k)}\|_{L^2(T_{a_k,b})}^2=\|u\|_{L^2(T_{a,b})}^2,
$$
we have
\begin{equation}
\frac{a_2-a}{a_2-a_1}\|u^{(1)}\|_{L^2(T_{a_1,b})}^2+\frac{a-a_1}{a_2-a_1}\|u^{(2)}\|_{L^2(T_{a_2,b})}^2
		=\|u\|_{L^2(T_{a,b})}^2 \label{a-L2-1}
\end{equation}
and for $u\in H^1(T_{a,b})$, from
\begin{align*}
	\|\nabla u^{(k)}\|_{L^2(T_{a_k,b})}^2&=\|u_x\|_{L^2(T_{a,b})}^2+\left\|\frac{a-a_k}{b}u_x+u_y\right\|_{L^2(T_{a,b})}^2 \\
	&=\|\nabla u\|_{L^2(T_{a,b})}^2+\frac{2(a-a_k)}{b}(u_x,u_y)_{L^2(T_{a,b})} + \frac{(a-a_k)^2}{b^2}\|u_x\|_{L^2(T_{a,b})}^2,
\end{align*}
we have
\begin{align}
	&\frac{a_2-a}{a_2-a_1}\|\nabla u^{(1)}\|_{L^2(T_{a_1,b})}^2+\frac{a-a_1}{a_2-a_1}\|\nabla u^{(2)}\|_{L^2(T_{a_2,b})}^2 \notag\\
	&\qquad =\|\nabla u\|_{L^2(T_{a,b})}^2+\frac{(a-a_1)(a_2-a)}{b^2}\|u_x\|_{L^2(T_{a,b})}^2 \notag\\
	&\qquad \ge\|\nabla u\|_{L^2(T_{a,b})}^2. \label{a-H1-1}
\end{align}

For $u\in H^1(T_{a,b})$,
\begin{align}
	\|\nabla u^{(k)}\|_{L^2(T_{a_k,b})}^2&=\|u_x\|_{L^2(T_{a,b})}^2+\left\|\frac{a-a_k}{b}u_x+u_y\right\|_{L^2(T_{a,b})}^2 \notag\\
		&=\|\nabla u\|_{L^2(T_{a,b})}^2+\frac{2(a-a_k)}{b}(u_x,u_y)_{L^2(T_{a,b})} + \frac{(a-a_k)^2}{b^2}\|u_x\|_{L^2(T_{a,b})}^2 \notag\\
		&\le\|\nabla u\|_{L^2(T_{a,b})}^2+\frac{2(a-a_k)}{b}(u_x,u_y)_{L^2(T_{a,b})} + \frac{(a-a_k)^2}{b^2}\|\nabla u\|_{L^2(T_{a,b})}^2 \notag\\
		&\le(1+h_a^2)\|\nabla u\|_{L^2(T_{a,b})}^2+\frac{2(a-a_k)}{b}(u_x,u_y)_{L^2(T_{a,b})}, \label{a-H1-2}
\end{align}
and for $u\in H^2(T_{a,b})$,
\begin{align}
	|u^{(k)}|_{H^2(T_{a_k,b})}^2&=\|u_{xx}\|_{L^2(T_{a,b})}^2+2\left\|\frac{a-a_k}{b}u_{xx}+u_{xy}\right\|_{L^2(T_{a,b})}^2 \notag\\
			&\qquad\qquad\qquad\qquad + \left\|\frac{(a-a_k)^2}{b^2}u_{xx}+\frac{2(a-a_k)}{b}u_{xy}+u_{yy}\right\|_{L^2(T_{a,b})}^2 \notag\\
		&=|u|_{H^2(T_{a,b})}^2+\frac{4(a-a_k)}{b}(\Delta u,u_{xy})_{L^2(T_{a,b})} \notag\\
			&\qquad + \frac{2(a-a_k)^2}{b^2}\Big(\|u_{xx}\|_{L^2(T_{a,b})}^2 + 2\|u_{xy}\|_{L^2(T_{a,b})}^2 + (u_{xx},u_{yy})_{L^2(T_{a,b})}\Big) \notag\\
			&\qquad + \frac{4(a-a_k)^3}{b^3}(u_{xx},u_{xy})_{L^2(T_{a,b})} + \frac{(a-a_k)^4}{b^4}\|u_{xx}\|_{L^2(T_{a,b})}^2 \notag\\
		&\le |u|_{H^2(T_{a,b})}^2+\frac{4(a-a_k)}{b}(\Delta u,u_{xy})_{L^2(T_{a,b})} \notag\\
			&\qquad + \frac{2(a-a_k)^2}{b^2}\left(\frac{6}{5}\|u_{xx}\|_{L^2(T_{a,b})}^2 + 2\|u_{xy}\|_{L^2(T_{a,b})}^2 + \frac{5}{4}\|u_{yy}\|_{L^2(T_{a,b})}^2\right) \notag\\
			&\qquad + \frac{4|a-a_k|^3}{b^3}\left(\frac{1}{4}\|u_{xx}\|_{L^2(T_{a,b})}^2 + \|u_{xy}\|_{L^2(T_{a,b})}^2\right) \notag\\
			&\qquad + \frac{(a-a_k)^4}{b^4}\|u_{xx}\|_{L^2(T_{a,b})}^2 \notag\\
		&\le |u|_{H^2(T_{a,b})}^2+\frac{4(a-a_k)}{b}(\Delta u,u_{xy})_{L^2(T_{a,b})} \notag\\
			&\qquad + \frac{2(a-a_k)^2}{b^2}\left(\frac{6}{5}\|u_{xx}\|_{L^2(T_{a,b})}^2 + 2\|u_{xy}\|_{L^2(T_{a,b})}^2 + \frac{5}{4}\|u_{yy}\|_{L^2(T_{a,b})}^2\right) \notag\\
			&\qquad + \frac{4(a-a_k)^2}{50b^2}\left(\frac{1}{4}\|u_{xx}\|_{L^2(T_{a,b})}^2 + \|u_{xy}\|_{L^2(T_{a,b})}^2\right) \notag\\
			&\qquad + \frac{(a-a_k)^2}{50^2b^2}\|u_{xx}\|_{L^2(T_{a,b})}^2 \notag\\
		&=|u|_{H^2(T_{a,b})}^2+\frac{4(a-a_k)}{b}(\Delta u,u_{xy})_{L^2(T_{a,b})} \notag\\
			&\qquad + \frac{(a-a_k)^2}{b^2}\left(\frac{6051}{2500}\|u_{xx}\|_{L^2(T_{a,b})}^2 + \frac{102}{25}\|u_{xy}\|_{L^2(T_{a,b})}^2 + \frac{5}{2}\|u_{yy}\|_{L^2(T_{a,b})}^2\right) \notag\\
		&\le |u|_{H^2(T_{a,b})}^2+\frac{4(a-a_k)}{b}(\Delta u,u_{xy})_{L^2(T_{a,b})} + \frac{5(a-a_k)^2}{2b^2}|u|_{H^2(T_{a,b})}^2 \notag\\
		&\le \left(1+\frac{5}{2}h_a^2\right)|u|_{H^2(T_{a,b})}^2+\frac{4(a-a_k)}{b}(\Delta u,u_{xy})_{L^2(T_{a,b})}. \label{a-H2-2}
\end{align}

Here, by Taylor's theorem, for each $j=1,2,3,4$ and $k=1,2$, there exist $a_1\le \widetilde{a}\le a_2$ such that
$$
	L_j(a_k,b)=L_j(a,b)+(a_k-a)\frac{\partial L_j}{\partial a}(a,b)+\frac{(a_k-a)^2}{2}\cdot\frac{\partial^2 L_j}{\partial a^2}(\widetilde{a},b)
$$
holds.
Then, using
$$
	|\widetilde{a}-a|\le|a_2-a_1|\le\frac{b}{50}
$$
and Lemma~\ref{BoundConstants}, we have
\begin{align}
	L_j(a_k,b)&\le L_j(a,b)+(a_k-a)\frac{\partial L_j}{\partial a}(a,b)+\frac{\alpha_{j2}(a_k-a)^2}{2b^2}L_j(a,b) \notag\\
	&\le \left(1+\frac{\alpha_{j2}}{2}h_a^2\right)L_j(a,b)-(a-a_k)\frac{\partial L_j}{\partial a}(a,b). \label{a-Lj}
\end{align}

For $u\in V^{1,j},\;j=1,2$, \eqref{a-H1-2}, \eqref{a-Lj}, and Lemma~\ref{BoundConstants} yield
\begin{align*}
	&\frac{a_2-a}{a_2-a_1}L_j(a_1,b)\|\nabla u^{(1)}\|_{L^2(T_{a_1,b})}^2
		+\frac{a-a_1}{a_2-a_1}L_j(a_2,b)\|\nabla u^{(2)}\|_{L^2(T_{a_2,b})}^2 \\
	&\le\frac{a_2-a}{a_2-a_1}\left\{\left(1+\frac{\alpha_{j2}}{2}h_a^2\right)L_j(a,b)-(a-a_1)\frac{\partial L_j}{\partial a}(a,b)\right\} \\
		&\qquad\qquad\qquad \times\left((1+h_a^2)\|\nabla u\|_{L^2(T_{a,b})}^2+\frac{2(a-a_1)}{b}(u_x,u_y)_{L^2(T_{a,b})}\right) \\
		&\qquad +\frac{a-a_1}{a_2-a_1}\left\{\left(1+\frac{\alpha_{j2}}{2}h_a^2\right)L_j(a,b)-(a-a_2)\frac{\partial L_j}{\partial a}(a,b)\right\}\ \\
		&\qquad\qquad\qquad \times\left((1+h_a^2)\|\nabla u\|_{L^2(T_{a,b})}^2+\frac{2(a-a_2)}{b}(u_x,u_y)_{L^2(T_{a,b})}\right) \\
	&=(1+h_a^2)\left(1+\frac{\alpha_{j2}}{2}h_a^2\right)L_j(a,b)\|\nabla u\|_{L^2(T_{a,b})}^2 \\
		&\qquad\qquad - \frac{2(a-a_1)(a_2-a)}{b}\cdot\frac{\partial L_j}{\partial a}(a,b)\cdot(u_x,u_y)_{L^2(T_{a,b})} \\
	&\le(1+h_a^2)\left(1+\frac{\alpha_{j2}}{2}h_a^2\right)L_j(a,b)\|\nabla u\|_{L^2(T_{a,b})}^2 \\
		&\qquad\qquad + \frac{(a_2-a_1)^2}{2b^2}\cdot\alpha_{j1}L_j(a,b)\cdot\frac{\|\nabla u\|_{L^2(T_{a,b})}^2}{2} \\
	&\le(1+h_a^2)\left(1+\frac{\alpha_{j2}}{2}h_a^2\right)L_j(a,b)\|\nabla u\|_{L^2(T_{a,b})}^2
		+ \frac{\alpha_{j1}h_a^2}{4}L_j(a,b)\|\nabla u\|_{L^2(T_{a,b})}^2 \\
	&=\left\{1+\frac{h_a^2}{4}\Big(4+\alpha_{j1}+2\alpha_{j2}(1+h_a^2)\Big)\right\}L_j(a,b)\|\nabla u\|_{L^2(T_{a,b})}^2 \\
	&\le\left\{1+\frac{h_a^2}{4}\left(4+2+2\cdot5\left(1+\frac{1}{50^2}\right)\right)\right\}L_j(a,b)\|\nabla u\|_{L^2(T_{a,b})}^2 \\
	&=\left(1+\frac{4001}{1000}h_a^2\right)L_j(a,b)\|\nabla u\|_{L^2(T_{a,b})}^2 \\
	&\le(1+5h_a^2)L_j(a,b)\|\nabla u\|_{L^2(T_{a,b})}^2.
\end{align*}
By this result and \eqref{a-L2-1}, we have
\begin{align*}
	\|u\|_{L^2(T_{a,b})}^2&=\frac{a_2-a}{a_2-a_1}\|u^{(1)}\|_{L^2(T_{a_1,b})}^2+\frac{a-a_1}{a_2-a_1}\|u^{(2)}\|_{L^2(T_{a_2,b})}^2 \\
	&\le\frac{a_2-a}{a_2-a_1}C_j(T_{a_1,b})^2\|\nabla u^{(1)}\|_{L^2(T_{a_1,b})}^2+\frac{a-a_1}{a_2-a_1}C_j(T_{a_2,b})^2\|\nabla u^{(2)}\|_{L^2(T_{a_2,b})}^2 \\
	&\le\max_{k=1,2}\frac{C_j(T_{a_k,b})^2}{L_j(a_k,b)}\left(\frac{a_2-a}{a_2-a_1}L_j(a_1,b)\|\nabla u^{(1)}\|_{L^2(T_{a_1,b})}^2+\frac{a-a_1}{a_2-a_1}L_j(a_2,b)\|\nabla u^{(2)}\|_{L^2(T_{a_2,b})}^2\right) \\
	&\le\max_{k=1,2}\frac{C_j(T_{a_k,b})^2}{L_j(a_k,b)}(1+5h_a^2)L_j(a,b)\|\nabla u\|_{L^2(T_{a,b})}^2.
\end{align*}
Then, considering the definitions of $C_j(T),\; j=1,2$, we have
$$
	\frac{C_j(T_{a,b})^2}{L_j(a,b)}\le(1+5h_a^2)\max_{k=1,2}\frac{C_j(T_{a_k,b})^2}{L_j(a_k,b)}.
$$

For $u\in V^2,\;j=3,4$, \eqref{a-H2-2}, \eqref{a-Lj}, and Lemma~\ref{BoundConstants} yield
\begin{align*}
	&\frac{a_2-a}{a_2-a_1}L_j(a_1,b)|u^{(1)}|_{H^2(T_{a_1,b})}^2
		+\frac{a-a_1}{a_2-a_1}L_j(a_2,b)|u^{(2)}|_{H^2(T_{a_2,b})}^2 \\
	&\le\frac{a_2-a}{a_2-a_1}\left\{\left(1+\frac{\alpha_{j2}}{2}h_a^2\right)L_j(a,b)-(a-a_1)\frac{\partial L_j}{\partial a}(a,b)\right\} \\
		&\qquad\qquad\qquad \times\left\{\left(1+\frac{5}{2}h_a^2\right)|u|_{H^2(T_{a,b})}^2+\frac{4(a-a_1)}{b}(\Delta u,u_{xy})_{L^2(T_{a,b})}\right\} \\
		&\qquad +\frac{a-a_1}{a_2-a_1}\left\{\left(1+\frac{\alpha_{j2}}{2}h_a^2\right)L_j(a,b)-(a-a_2)\frac{\partial L_j}{\partial a}(a,b)\right\}\ \\
		&\qquad\qquad\qquad \times\left\{\left(1+\frac{5}{2}h_a^2\right)|u|_{H^2(T_{a,b})}^2+\frac{4(a-a_2)}{b}(\Delta u,u_{xy})_{L^2(T_{a,b})}\right\} \\
	&=\left(1+\frac{5}{2}h_a^2\right)\left(1+\frac{\alpha_{j2}}{2}h_a^2\right)L_j(a,b)|u|_{H^2(T_{a,b})}^2 \\
		&\qquad\qquad - \frac{4(a-a_1)(a_2-a)}{b}\cdot\frac{\partial L_j}{\partial a}(a,b)\cdot(\Delta u,u_{xy})_{L^2(T_{a,b})} \\
	&\le\left(1+\frac{5}{2}h_a^2\right)\left(1+\frac{\alpha_{j2}}{2}h_a^2\right)L_j(a,b)|u|_{H^2(T_{a,b})}^2
		+ \frac{(a_2-a_1)^2}{b^2}\alpha_{j1}L_j(a,b)\cdot\frac{|u|_{H^2(T_{a,b})}^2}{2} \\
	&\le\left(1+\frac{5}{2}h_a^2\right)\left(1+\frac{\alpha_{j2}}{2}h_a^2\right)L_j(a,b)|u|_{H^2(T_{a,b})}^2
		+ \frac{\alpha_{j1}h_a^2}{2}L_j(a,b)|u|_{H^2(T_{a,b})}^2 \\
	&=\left\{1+\frac{h_a^2}{4}\Big(10+2\alpha_{j1}+\alpha_{j2}(2+5h_a^2)\Big)\right\}L_j(a,b)|u|_{H^2(T_{a,b})}^2.
\end{align*}
Then, for $j=3$, we have
\begin{align*}
	&\frac{a_2-a}{a_2-a_1}L_3(a_1,b)|u^{(1)}|_{H^2(T_{a_1,b})}^2
		+\frac{a-a_1}{a_2-a_1}L_3(a_2,b)|u^{(2)}|_{H^2(T_{a_2,b})}^2 \\
	&\le\left\{1+\frac{h_a^2}{4}\left(10+2\cdot2+4\left(2+\frac{5}{50^2}\right)\right)\right\}L_3(a,b)|u|_{H^2(T_{a,b})}^2 \\
	&= \left(1+\frac{2751}{500}h_a^2\right)L_3(a,b)|u|_{H^2(T_{a,b})}^2 \\
	&\le(1+6h_a^2)L_3(a,b)|u|_{H^2(T_{a,b})}^2,
\end{align*}
and consequently, by \eqref{a-L2-1},
\begin{align*}
	\|u\|_{L^2(T_{a,b})}^2&=\frac{a_2-a}{a_2-a_1}\|u^{(1)}\|_{L^2(T_{a_1,b})}^2+\frac{a-a_1}{a_2-a_1}\|u^{(2)}\|_{L^2(T_{a_2,b})}^2 \\
	&\le\frac{a_2-a}{a_2-a_1}C_3(T_{a_1,b})^2|u^{(1)}|_{H^2(T_{a_1,b})}^2+\frac{a-a_1}{a_2-a_1}C_3(T_{a_2,b})^2|u^{(2)}|_{H^2(T_{a_2,b})}^2 \\
	&\le\max_{k=1,2}\frac{C_3(T_{a_k,b})^2}{L_3(a_k,b)}\left(\frac{a_2-a}{a_2-a_1}L_3(a_1,b)|u^{(1)}|_{H^2(T_{a_1,b})}^2+\frac{a-a_1}{a_2-a_1}L_3(a_2,b)|u^{(2)}|_{H^2(T_{a_2,b})}^2\right) \\
	&\le\max_{k=1,2}\frac{C_3(T_{a_k,b})^2}{L_3(a_k,b)}(1+6h_a^2)L_3(a,b)|u|_{H^2(T_{a,b})}^2
\end{align*}
holds.
Then, considering the definitions of $C_3(T)$, we have
$$
	\frac{C_3(T_{a,b})^2}{L_3(a,b)}\le (1+8h_b^2)\max_{k=1,2}\frac{C_3(T_{a,b_k})^2}{L_3(a,b_k)}.
$$

For $j=4$, we have
\begin{align*}
	&\frac{a_2-a}{a_2-a_1}L_4(a_1,b)|u^{(1)}|_{H^2(T_{a_1,b})}^2
		+\frac{a-a_1}{a_2-a_1}L_4(a_2,b)|u^{(2)}|_{H^2(T_{a_2,b})}^2 \\
	&\le \left\{1+\frac{h_a^2}{4}\left(10+2\cdot3+9\left(2+\frac{5}{50^2}\right)\right)\right\}L_4(a,b)|u|_{H^2(T_{a,b})}^2 \\
	&\le \left(1+\frac{17009}{2000}h_a^2\right)L_4(a,b)|u|_{H^2(T_{a,b})}^2 \\
	&\le(1+9h_a^2)L_4(a,b)|u|_{H^2(T_{a,b})}^2,
\end{align*}
and consequently, by \eqref{a-H1-1},
\begin{align*}
	\|\nabla u\|_{L^2(T_{a,b})}^2&\le\frac{a_2-a}{a_2-a_1}\|\nabla u^{(1)}\|_{L^2(T_{a_1,b})}^2+\frac{a-a_1}{a_2-a_1}\|\nabla u^{(2)}\|_{L^2(T_{a_2,b})}^2 \\
	&\le\frac{a_2-a}{a_2-a_1}C_4(T_{a_1,b})^2|u^{(1)}|_{H^2(T_{a_1,b})}^2+\frac{a-a_1}{a_2-a_1}C_4(T_{a_2,b})^2|u^{(2)}|_{H^2(T_{a_2,b})}^2 \\
	&\le\max_{k=1,2}\frac{C_4(T_{a_k,b})^2}{L_4(a_k,b)}\left(\frac{a_2-a}{a_2-a_1}L_4(a_1,b)|u^{(1)}|_{H^2(T_{a_1,b})}^2+\frac{a-a_1}{a_2-a_1}L_4(a_2,b)|u^{(2)}|_{H^2(T_{a_2,b})}^2\right) \\
	&\le\max_{k=1,2}\frac{C_4(T_{a_k,b})^2}{L_4(a_k,b)}(1+9h_a^2)L_4(a,b)|u|_{H^2(T_{a,b})}^2
\end{align*}
holds.
Then, considering the definitions of $C_4(T)$, we have
$$
	\frac{C_4(T_{a,b})^2}{L_4(a,b)}\le (1+9h_b^2)\max_{k=1,2}\frac{C_4(T_{a,b_k})^2}{L_4(a,b_k)}.
$$
\end{proof}

\begin{proof}[Proof of theorem~\ref{b-interpolation}]\quad
In this proof, for $k=1,2$, we take
$$
	u^{(k)}(x,y)=u\left(x,\frac{b}{b_k}y\right).
$$

For $u\in L^2(T_{a,b})$, from
$$
	\|u^{(k)}\|_{L^2(T_{a,b_k})}^2=\frac{b_k}{b}\|u\|_{L^2(T_{a,b})}^2
$$
we have
\begin{equation}
	\frac{b_2-b}{b_2-b_1}\|u^{(1)}\|_{L^2(T_{a,b_1})}^2+\frac{b-b_1}{b_2-b_1}\|u^{(2)}\|_{L^2(T_{a,b_2})}^2
		=\|u\|_{L^2(T_{a,b})}^2 \label{b-L2-1}
\end{equation}
and for $u\in H^1(T_{a,b})$, from
\begin{equation}
	\|\nabla u^{(k)}\|_{L^2(T_{a,b_k})}^2=\frac{b_k}{b}\|u_x\|_{L^2(T_{a,b})}^2+\frac{b}{b_k}\|u_y\|_{L^2(T_{a,b})}^2 \label{b-H1-2}
\end{equation}
we have
\begin{align}
	&\frac{b_2-b}{b_2-b_1}\|\nabla u^{(1)}\|_{L^2(T_{a,b_1})}^2+\frac{b-b_1}{b_2-b_1}\|\nabla u^{(2)}\|_{L^2(T_{a,b_2})}^2 \notag\\
	&\qquad =\|\nabla u\|_{L^2(T_{a,b})}^2+\frac{(b-b_1)(b_2-b)}{b_1b_2}\|u_y\|_{L^2(T_{a,b})}^2 \notag\\
	&\qquad \ge\|\nabla u\|_{L^2(T_{a,b})}^2. \label{b-H1-1}
\end{align}

For $u\in H^2(T_{a,b})$, we have
\begin{equation}
	|u^{(k)}|_{H^2(T_{a,b_k})}^2=\frac{b_k}{b}\|u_{xx}\|_{L^2(T_{a,b})}^2
		+ \frac{2b}{b_k}\|u_{xy}\|_{L^2(T_{a,b})}^2 + \frac{b^3}{b_k^3}\|u_{yy}\|_{L^2(T_{a,b})}^2. \label{b-H2-2}
\end{equation}

Here, by Taylor's theorem, for each $j=1,2,3,4$ and $k=1,2$, there exist $b_1\le \widetilde{b}\le b_2$ such that
$$
	L_j(a,b_k)=L_j(a,b)+(b_k-b)\frac{\partial L_j}{\partial b}(a,b)+\frac{(b_k-b)^2}{2}\cdot\frac{\partial^2 L_j}{\partial b^2}(a,\widetilde{b})
$$
holds.
Then, using
$$
	|\widetilde{b}-b|\le|b_2-b_1|\le\frac{b_1}{50}\le\frac{b}{50}
$$
and Lemma~\ref{BoundConstants}, we have
\begin{align}
	L_j(a,b_k)&\le L_j(a,b)+(b_k-b)\frac{\partial L_j}{\partial b}(a,b)+\frac{\beta_{j2}(b_k-b)^2}{2b^2}L_j(a,b) \notag\\
	&\le \left(1+\frac{\beta_{j2}}{2}h_b^2\right)L_j(a,b)+(b_k-b)\frac{\partial L_j}{\partial b}(a,b). \label{b-Lj}
\end{align}

For $u\in V^{1,j},\;j=1,2$, \eqref{b-H1-2}, \eqref{b-Lj}, and Lemma~\ref{BoundConstants} yield
\begin{align*}
	&\frac{b_2-b}{b_2-b_1}L_j(a,b_1)\|\nabla u^{(1)}\|_{L^2(T_{a,b_1})}^2
		+\frac{b-b_1}{b_2-b_1}L_j(a,b_2)\|\nabla u^{(2)}\|_{L^2(T_{a,b_2})}^2 \\
	&\le\frac{b_2-b}{b_2-b_1}\left\{\left(1+\frac{\beta_{j2}}{2}h_b^2\right)L_j(a,b)+(b_1-b)\frac{\partial L_j}{\partial b}(a,b)\right\} \\
		&\qquad\qquad\qquad\qquad \times\left(\frac{b_1}{b}\|u_x\|_{L^2(T_{a,b})}^2+\frac{b}{b_1}\|u_y\|_{L^2(T_{a,b})}^2\right) \\
		&\qquad +\frac{b-b_1}{b_2-b_1}\left\{\left(1+\frac{\beta_{j2}}{2}h_b^2\right)L_j(a,b)+(b_2-b)\frac{\partial L_j}{\partial b}(a,b)\right\} \\
		&\qquad\qquad\qquad\qquad \times\left(\frac{b_2}{b}\|u_x\|_{L^2(T_{a,b})}^2+\frac{b}{b_2}\|u_y\|_{L^2(T_{a,b})}^2\right) \\
	&=\left(1+\frac{\beta_{j2}}{2}h_b^2\right)L_j(a,b)\|\nabla u\|_{L^2(T_{a,b})}^2
		+ (b-b_1)(b_2-b)\left\{\frac{2+\beta_{j2}h_b^2}{2b_1b_2}L_j(a,b)\|u_y\|_{L^2(T_{a,b})}^2\right. \\
		&\qquad\qquad\left. +b\cdot\frac{\partial L_j}{\partial b}(a,b)\left(	\frac{1}{b^2}\|u_x\|_{L^2(T_{a,b})}^2
		-\frac{1}{b_1b_2}\|u_y\|_{L^2(T_{a,b})}^2\right) \right\} \\
	&\le\left(1+\frac{\beta_{j2}}{2}h_b^2\right)L_j(a,b)\|\nabla u\|_{L^2(T_{a,b})}^2 \\
		&\qquad\qquad +\frac{(b_2-b_1)^2}{4}L_j(a,b)\left\{\frac{\beta_{j1}}{b_1^2}\|u_x\|_{L^2(T_{a,b})}^2
		+\frac{2+2\beta_{j1}+\beta_{j2}h_b^2}{2b_1^2}\|u_y\|_{L^2(T_{a,b})}^2 \right\} \\
	&\le\left(1+\frac{\beta_{j2}}{2}h_b^2\right)L_j(a,b)\|\nabla u\|_{L^2(T_{a,b})}^2
		+\frac{h_b^2}{8}\Big(2+2\beta_{j1}+\beta_{j2}h_b^2\Big)L_j(a,b)\|\nabla u\|_{L^2(T_{a,b})}^2 \\
	&\le\left\{1+\frac{h_b^2}{8}\Big(2+2\beta_{j1}+\beta_{j2}(4+h_b^2)\Big)\right\}L_j(a,b)\|\nabla u\|_{L^2(T_{a,b})}^2 \\
	&\le \left\{1+\frac{h_b^2}{8}\left(2+2\cdot2+4\left(4+\frac{1}{50^2}\right)\right)\right\}L_j(a,b)\|\nabla u\|_{L^2(T_{a,b})}^2 \\
	&=\left(1+\frac{13751}{5000}h_b^2\right)L_j(a,b)\|\nabla u\|_{L^2(T_{a,b})}^2 \\
	&\le(1+3h_b^2)L_j(a,b)\|\nabla u\|_{L^2(T_{a,b})}^2.
\end{align*}
By this result and \eqref{b-L2-1}, we have
\begin{align*}
	\|u\|_{L^2(T_{a,b})}^2&=\frac{b_2-b}{b_2-b_1}\|u^{(1)}\|_{L^2(T_{a,b_1})}^2+\frac{b-b_1}{b_2-b_1}\|u^{(2)}\|_{L^2(T_{a,b_2})}^2 \\
	&\le\frac{b_2-b}{b_2-b_1}C_j(T_{a,b_1})^2\|\nabla u^{(1)}\|_{L^2(T_{a,b_1})}^2+\frac{b-b_1}{b_2-b_1}C_j(T_{a,b_2})^2\|\nabla u^{(2)}\|_{L^2(T_{a,b_2})}^2 \\
	&\le\max_{k=1,2}\frac{C_j(T_{a,b_k})^2}{L_j(a,b_k)}\left(\frac{b_2-b}{b_2-b_1}L_j(a,b_1)\|\nabla u^{(1)}\|_{L^2(T_{a,b_1})}^2+\frac{b-b_1}{b_2-b_1}L_j(a,b_2)\|\nabla u^{(2)}\|_{L^2(T_{a,b_2})}^2\right) \\
	&\le\max_{k=1,2}\frac{C_j(T_{a,b_k})^2}{L_j(a,b_k)}(1+3h_b^2)L_j(a,b)\|\nabla u\|_{L^2(T_{a,b})}^2.
\end{align*}
Then, considering the definitions of $C_j(T),\; j=1,2$, we have
$$
	\frac{C_j(T_{a,b})^2}{L_j(a,b)}\le(1+3h_b^2)\max_{k=1,2}\frac{C_j(T_{a,b_k})^2}{L_j(a,b_k)}.
$$

For $u\in V^2,\;j=3,4$, \eqref{b-H2-2}, \eqref{b-Lj}, and Lemma~\ref{BoundConstants} yield
\begin{align*}
	&\frac{b_2-b}{b_2-b_1}L_j(a,b_1)|u^{(1)}|_{H^2(T_{a,b_1})}^2
		+\frac{b-b_1}{b_2-b_1}L_j(a,b_2)|u^{(2)}|_{H^2(T_{a,b_2})}^2 \\
	&\le\frac{b_2-b}{b_2-b_1}\left\{\left(1+\frac{\beta_{j2}}{2}h_b^2\right)L_j(a,b)+(b_1-b)\frac{\partial L_j}{\partial b}(a,b)\right\} \\
		&\qquad\qquad\qquad\qquad \times\left(\frac{b_1}{b}\|u_{xx}\|_{L^2(T_{a,b})}^2 + \frac{2b}{b_1}\|u_{xy}\|_{L^2(T_{a,b})}^2 + \frac{b^3}{b_1^3}\|u_{yy}\|_{L^2(T_{a,b})}^2\right) \\
		&\qquad +\frac{b-b_1}{b_2-b_1}\left\{\left(1+\frac{\beta_{j2}}{2}h_b^2\right)L_j(a,b)+(b_2-b)\frac{\partial L_j}{\partial b}(a,b)\right\} \\
		&\qquad\qquad\qquad\qquad \times\left(\frac{b_2}{b}\|u_{xx}\|_{L^2(T_{a,b})}^2 + \frac{2b}{b_2}\|u_{xy}\|_{L^2(T_{a,b})}^2 + \frac{b^3}{b_2^3}\|u_{yy}\|_{L^2(T_{a,b})}^2\right) \\
	&=\left(1+\frac{\beta_{j2}}{2}h_b^2\right)L_j(a,b)|u|_{H^2(T_{a,b})}^2 + (b-b_1)(b_2-b)\left\{\phantom{\left(\frac{1}{b^2}\right)}\right. \\
		&\qquad (2+\beta_{j2}h_b^2)L_j(a,b)\left(\frac{1}{b_1b_2}\|u_{xy}\|_{L^2(T_{a,b})}^2 \right. \\
		&\qquad\qquad\qquad\left. +\frac{b_1^2b_2^2 + b_1b_2(b_1+b_2)b + (b_1^2+b_1b_2+b_2^2)b^2}{2b_1^3b_2^3}\|u_{yy}\|_{L^2(T_{a,b})}^2\right) \\
		&\qquad\left. +b\cdot\frac{\partial L_j}{\partial b}(a,b)\left(\frac{1}{b^2}\|u_{xx}\|_{L^2(T_{a,b})}^2 - \frac{2}{b_1b_2}\|u_{xy}\|_{L^2(T_{a,b})}^2 - \frac{(b_1^2+b_1b_2+b_2^2)b^2}{b_1^3b_2^3}\|u_{yy}\|_{L^2(T_{a,b})}^2 \right)\right\} \\
	&\le \left(1+\frac{\beta_{j2}}{2}h_b^2\right)L_j(a,b)|u|_{H^2(T_{a,b})}^2 + \frac{(b_2-b_1)^2}{4}\left\{\phantom{\left(\frac{1}{b^2}\right)}\right. \\
		&\qquad (2+\beta_{j2}h_b^2)L_j(a,b)\left(\frac{1}{b_1^2}\|u_{xy}\|_{L^2(T_{a,b})}^2
		+\frac{3b_2}{b_1^3}\|u_{yy}\|_{L^2(T_{a,b})}^2\right) \\
		&\qquad\left. +\beta_{j1}L_j(a,b)\left(\frac{1}{b_1^2}\|u_{xx}\|_{L^2(T_{a,b})}^2 + \frac{2}{b_1^2}\|u_{xy}\|_{L^2(T_{a,b})}^2
		+ \frac{3b_2}{b_1^3}\|u_{yy}\|_{L^2(T_{a,b})}^2 \right)\right\} \\
	&\le \left(1+\frac{\beta_{j2}}{2}h_b^2\right)L_j(a,b)|u|_{H^2(T_{a,b})}^2 \\
	&\qquad\qquad + \frac{h_b^2}{4}\left\{
		(2+\beta_{j2}h_b^2)L_j(a,b)\left(\|u_{xy}\|_{L^2(T_{a,b})}^2+3(1+h_b)\|u_{yy}\|_{L^2(T_{a,b})}^2\right)\right. \\
		&\qquad\left. +\beta_{j1}L_j(a,b)\left(\|u_{xx}\|_{L^2(T_{a,b})}^2 + 2\|u_{xy}\|_{L^2(T_{a,b})}^2 + 3(1+h_b)\|u_{yy}\|_{L^2(T_{a,b})}^2 \right)\right\} \\
	&\le \left(1+\frac{\beta_{j2}}{2}h_b^2\right)L_j(a,b)|u|_{H^2(T_{a,b})}^2 + \frac{3h_b^2}{4} 
		(1+h_b)\Big(2+\beta_{j1}+\beta_{j2}h_b^2\Big)L_j(a,b)|u|_{H^2(T_{a,b})}^2 \\
	&=\left\{1+\frac{h_b^2}{4}\Big(3(2+\beta_{j1})(1+h_b)+\beta_{j2}(2+3h_b^2+3h_b^3)\Big)\right\}
		L_j(a,b)|u|_{H^2(T_{a,b})}^2.
\end{align*}
Then, for $j=3$, we have
\begin{align*}
	&\frac{b_2-b}{b_2-b_1}L_3(a,b_1)|u^{(1)}|_{H^2(T_{a,b_1})}^2
		+\frac{b-b_1}{b_2-b_1}L_3(a,b_2)|u^{(2)}|_{H^2(T_{a,b_2})}^2 \\
	&\le \left\{1+\frac{h_b^2}{4}\left(3(2+3)\left(1+\frac{1}{50}\right)+8\left(2+\frac{3}{50^2}+\frac{3}{50^3}\right)\right)\right\}L_3(a,b)|u|_{H^2(T_{a,b})}^2 \\
	&=\left(1+\frac{978431}{125000}h_b^2\right)L_3(a,b)|u|_{H^2(T_{a,b})}^2 \\
	&\le(1+8h_b^2)L_3(a,b)|u|_{H^2(T_{a,b})}^2,
\end{align*}
and consequently, by \eqref{b-L2-1},
\begin{align*}
	\|u\|_{L^2(T_{a,b})}^2&=\frac{b_2-b}{b_2-b_1}\|u^{(1)}\|_{L^2(T_{a,b_1})}^2+\frac{b-b_1}{b_2-b_1}\|u^{(2)}\|_{L^2(T_{a,b_2})}^2 \\
	&\le\frac{b_2-b}{b_2-b_1}C_3(T_{a,b_1})^2|u^{(1)}|_{H^2(T_{a,b_1})}^2+\frac{b-b_1}{b_2-b_1}C_3(T_{a,b_2})^2|u^{(2)}|_{H^2(T_{a,b_2})}^2 \\
	&\le\max_{k=1,2}\frac{C_3(T_{a,b_k})^2}{L_3(a,b_k)}\left(\frac{b_2-b}{b_2-b_1}L_3(a,b_1)|u^{(1)}|_{H^2(T_{a,b_1})}^2+\frac{b-b_1}{b_2-b_1}L_3(a,b_2)|u^{(2)}|_{H^2(T_{a,b_2})}^2\right) \\
	&\le\max_{k=1,2}\frac{C_3(T_{a,b_k})^2}{L_3(a,b_k)}(1+8h_b^2)L_3(a,b)|u|_{H^2(T_{a,b})}^2
\end{align*}
holds.
Then, considering the definitions of $C_3(T)$, we have
$$
	\frac{C_3(T_{a,b})^2}{L_3(a,b)}\le (1+8h_b^2)\max_{k=1,2}\frac{C_3(T_{a,b_k})^2}{L_3(a,b_k)}.
$$

For $j=4$, we have
\begin{align*}
	&\frac{b_2-b}{b_2-b_1}L_4(a,b_1)|u^{(1)}|_{H^2(T_{a,b_1})}^2
		+\frac{b-b_1}{b_2-b_1}L_4(a,b_2)|u^{(2)}|_{H^2(T_{a,b_2})}^2 \\
	&\le \left\{1+\frac{h_b^2}{4}\left(3(2+3)\left(1+\frac{1}{50}\right)+9\left(2+\frac{3}{50^2}+\frac{3}{50^3}\right)\right)\right\}L_3(a,b)|u|_{H^2(T_{a,b})}^2 \\
	&=\left(1+\frac{4163877}{500000}h_b^2\right)L_4(a,b)|u|_{H^2(T_{a,b})}^2 \\
	&\le(1+9h_b^2)L_4(a,b)|u|_{H^2(T_{a,b})}^2
\end{align*}
and consequently, by \eqref{b-H1-1},
\begin{align*}
	\|\nabla u\|_{L^2(T_{a,b})}^2&\le\frac{b_2-b}{b_2-b_1}\|\nabla u^{(1)}\|_{L^2(T_{a,b_1})}^2+\frac{b-b_1}{b_2-b_1}\|\nabla u^{(2)}\|_{L^2(T_{a,b_2})}^2 \\
	&\le\frac{b_2-b}{b_2-b_1}C_4(T_{a,b_1})^2|u^{(1)}|_{H^2(T_{a,b_1})}^2+\frac{b-b_1}{b_2-b_1}C_4(T_{a,b_2})^2|u^{(2)}|_{H^2(T_{a,b_2})}^2 \\
	&\le\max_{k=1,2}\frac{C_4(T_{a,b_k})^2}{L_4(a,b_k)}\left(\frac{b_2-b}{b_2-b_1}L_4(a,b_1)|u^{(1)}|_{H^2(T_{a,b_1})}^2+\frac{b-b_1}{b_2-b_1}L_4(a,b_2)|u^{(2)}|_{H^2(T_{a,b_2})}^2\right) \\
	&\le\max_{k=1,2}\frac{C_4(T_{a,b_k})^2}{L_4(a,b_k)}(1+9h_b^2)L_4(a,b)|u|_{H^2(T_{a,b})}^2
\end{align*}
holds.
Then, considering the definitions of $C_4(T)$, we have
$$
	\frac{C_4(T_{a,b})^2}{L_4(a,b)}\le (1+9h_b^2)\max_{k=1,2}\frac{C_4(T_{a,b_k})^2}{L_4(a,b_k)}.
$$
\end{proof}

\section{Proof for triangles that are not significantly degenerate} \label{Non_Degenerate}
In this section, we discuss the proof of Theorem~\ref{formula} when the triangle is not significantly degenerate;
namely, the following theorem holds.

\begin{theorem}
For $0\le a\le\dfrac{1}{2},\;\dfrac{1}{10}\le b\le 1$, it holds that
$$
	C_j(T_{a,b}) < K_j(T_{a,b}), \qquad j=1,2,3,4.
$$
\end{theorem}
\begin{proof}\quad
From the fact that
$$
	0\le\frac{a_{k\;l+1}-a_{kl}}{b_k}=\frac{1}{2x_kb_k}\le\frac{20}{y_kb_k}=\frac{1}{50}
$$
holds for $\{x_k\},\{y_k\},\{b_k\},\{a_{kl}\}$ that appear in Theorem~\ref{Verified_Result1}
and from Theorem~\ref{BoundbyFiniteResult} and Theorem~\ref{Verified_Result1},
\begin{align*}
	C_j(T_{a_{kl},b_k})^2&\left(1+5\cdot\frac{1}{50^2}\right)\left(1+3\cdot\frac{1}{50^2}\right)< L_j(a_{kl},b_k), \qquad j=1,2, \\
	C_3(T_{a_{kl},b_k})^2&\left(1+6\cdot\frac{1}{50^2}\right)\left(1+8\cdot\frac{1}{50^2}\right)< L_3(a_{kl},b_k)
\end{align*}
holds for $k=1,2,\dots,119,\; l=0,1,\dots x_k$.
Using this result and Theorem~\ref{a-interpolation},
\begin{align*}
	C_j(T_{a,b_k})^2&\left(1+3\cdot\frac{1}{50^2}\right)< L_j(a,b_k), \qquad j=1,2, \\
	C_3(T_{a,b_k})^2&\left(1+8\cdot\frac{1}{50^2}\right)< L_3(a,b_k)
\end{align*}
holds for $k=1,2,\dots,119,\; 0\le a\le\dfrac{1}{2}$.
Furthermore, from
$$
	0\le\frac{b_k-b_{k+1}}{b_{k+1}}=\frac{y_{k+1}}{y_k}-1\le\frac{51}{50}-1=\frac{1}{50},\qquad b_{119}=\frac{1000}{10133}<\frac{1}{10}
$$
and Theorem~\ref{b-interpolation},
$$
	C_j(T_{a,b})^2 < L_j(a,b), \qquad j=1,2,3
$$
holds for $0\le a\le\dfrac{1}{2},\; \dfrac{1}{10}\le b\le 1$.

From the results above, especially for $C_2(T_{a,b})$, and Theorem~\ref{BoundbyFiniteResult} and Theorem~\ref{Verified_Result1},
 it follows that 
$$
	C_4(T_{a_{kl},b_k})^2\left(1+9\cdot\frac{1}{50^2}\right)\left(1+9\cdot\frac{1}{50^2}\right)\le L_4(a_{kl},b_k)
$$
holds for $k=1,2,\dots,119,\; l=0,1,\dots x_k$, and similarly using Theorem~\ref{a-interpolation} and Theorem~\ref{b-interpolation},
$$
	C_4(T_{a,b})^2 < L_4(a,b)
$$
holds for $0\le a\le\dfrac{1}{2},\; \dfrac{1}{10}\le b\le 1$.
\end{proof}

Figure~\ref{PointsLines} shows the flow of the proof.
The inequalities, which are shown at discrete points on the $(a,b)$ plane (left), are extended to lines
 by Theorem~\ref{a-interpolation} (middle) and then to a rectangular domain by Theorem~\ref{b-interpolation} (right).

\begin{figure}[!t]
	\centering
	\includegraphics[width=0.95\hsize]{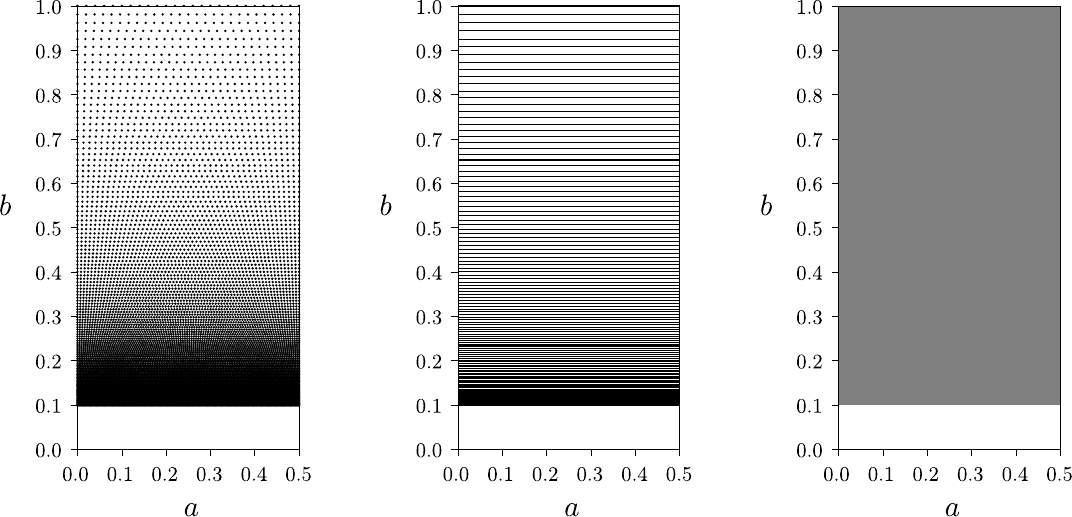}
	\caption{The region where the inequalities hold; (from left) discrete points, extended to lines,
and then to a rectangular domain.
}
	\label{PointsLines}
\end{figure}

\section{Proof for $C_j(T),\ j=1,2,3$ when triangle is degenerate} \label{Degenerate}
This section proves Theorem~\ref{formula} for $C_j(T),\; j=1,2,3$ when the triangle is degenerate.
That is, we show that the following theorem holds.

\begin{theorem} \label{small_b}
For $0\le a\le\dfrac{1}{2},\; 0<b\le\dfrac{1}{10}$,
$$
	C_j(T_{a,b}) < K_j(T_{a,b}), \qquad j=1,2,3
$$
holds.
\end{theorem}

To prove this theorem, we use the following two lemmas concerning the monotonicity of $C_j(T)$ and $L_j(a,b)$.
\begin{lemma} \label{Cj-monotonicity}
Assume that $0<\eta\le 1$. Then,
$$
	C_j(T_{a,\eta b})\le C_j(T_{a,b}), \qquad j=1,2,3
$$
holds.
\end{lemma}
\begin{proof}\quad
For $u(x,y)\in V^{1,j}(T_{a,\eta b}),\; j=1,2$, let $v(x,y)=u(x,\eta y)\in V^{1,j}(T_{a,b})$.
Then, it holds that
\begin{align*}
	\|v\|_{L^2(T_{a,b})}^2&=\frac{1}{\eta}\|u\|_{L^2(T_{a,\eta b})}^2 \\
	\|v_x\|_{L^2(T_{a,b})}^2&=\frac{1}{\eta}\|u_x\|_{L^2(T_{a,\eta b})}^2 \\
	\|v_y\|_{L^2(T_{a,b})}^2&=\eta\|u_y\|_{L^2(T_{a,\eta h})}^2
\end{align*}
and
\begin{align*}
	\|u\|_{L^2(T_{a,\eta b})}^2&=\eta\|v\|_{L^2(T_{a,b})}^2 \\
	&\le \eta C_j(T_{a,b})^2\|\nabla v\|_{L^2(T_{a,b})}^2=C_j(T_{a,b})^2\left(\|u_x\|_{L^2(T_{a,\eta b})}^2+\eta^2\|u_y\|_{L^2(T_{a,\eta b})}^2\right) \\
	&\le C_j(T_{a,b})^2\|\nabla u\|_{L^2(T_{a,\eta b})}^2.
\end{align*}
From this, we have
$$
	C_j(T_{a,\eta b}) \le C_j(T_{a,b}).
$$
Next, we examine the relation between $C_3(T_{a,\eta b})$ and $C_3(T_{a,b})$.
For $u(x,y)\in V^2(T_{a,\eta b}),\; j=1,2$, let $v(x,y)=u(x,\eta y)\in V^2(T_{a,b})$.
Then, it holds that
\begin{align*}
	\|v\|_{L^2(T_{a,b})}^2&=\frac{1}{\eta}\|u\|_{L^2(T_{a,\eta b})}^2 \\
	\|v_{xx}\|_{L^2(T_{a,b})}^2&=\frac{1}{\eta}\|u_{xx}\|_{L^2(T_{a,\eta b})}^2 \\
	\|v_{xy}\|_{L^2(T_{a,b})}^2&=\eta\|u_{xy}\|_{L^2(T_{a,\eta b})}^2 \\
	\|v_{yy}\|_{L^2(T_{a,b})}^2&=\eta^3\|u_{yy}\|_{L^2(T_{a,\eta b})}^2
\end{align*}
and
\begin{align*}
	\|u\|_{L^2(T_{a,\eta b})}^2&=\eta\|v\|_{L^2(T_{a,b})}^2 \\
	&\le \eta C_3(T_{a,b})^2|v|_{H^2(T_{a,b})}^2 \\
	&=C_3(T_{a,b})^2
		\left(\|u_{xx}\|_{L^2(T_{a,\eta b})}^2+2\eta^2\|u_{xy}\|_{L^2(T_{a,\eta b})}^2+\eta^4\|u_{yy}\|_{L^2(T_{a,\eta b})}^2\right) \\
	&\le C_3(T_{a,b})^2|u|_{H^2(T_{a,\eta b})}^2.
\end{align*}
From this, we have
$$
	C_3(T_{a,\eta b}) \le C_3(T_{a,b}).
$$
\end{proof}

\begin{lemma} \label{Lj-monotonicity}
For $0\le a\le1/2,\;0<b$, it holds that
$$
	\lim_{y\downarrow0}L_j(a,y) < L_j(a,b), \qquad j=1,2,3.
$$
\end{lemma}
The proof of this lemma comes down to showing the positivity of the multivariate polynomials;
since there is no intrinsic difficulty, it is presented in Lemma~\ref{Lj-limit} in Appendix~\ref{Appendix1}.

\begin{proof}[Proof of Theorem~\ref{small_b}]\quad
From Theorem~\ref{BoundbyFiniteResult} and Theorem~\ref{Verified_Result2},
\begin{align*}
	C_j(T_{a_l,\frac{1}{10}})^2&\left(1+5\cdot\frac{1}{50^2}\right) < \lim_{y\downarrow0}L_j(a_l,y), \qquad j=1,2, \\
	C_3(T_{a_l,\frac{1}{10}})^2&\left(1+6\cdot\frac{1}{50^2}\right) < \lim_{y\downarrow0}L_3(a_l,y)
\end{align*}
hold for $a_l = \dfrac{l}{500},\; l=0,1,\dots, 250$.
Therefore, Lemma~\ref{Cj-monotonicity} and Lemma~\ref{Lj-monotonicity} imply
\begin{align*}
	C_j(T_{a_l,b})^2&\left(1+5\cdot\frac{1}{50^2}\right) <  L_j(a_l,b), \qquad j=1,2, \\
	C_3(T_{a_l,b})^2&\left(1+6\cdot\frac{1}{50^2}\right) < L_3(a_l,b)
\end{align*}
for $0<b\le\dfrac{1}{10}$ and $l=0,1,\dots, 250$.

Here, using the result of Theorem~\ref{a-interpolation}, we have
$$
	C_j(T_{a,b})^2 < L_j(a,b), \qquad j=1,2,3
$$
for $0\le a\le\dfrac{1}{2},\;0<b\le\dfrac{1}{10}$.
\end{proof}

\section{Asymptotic analysis for $C_4(T)$} \label{Asymptotic_Analysis}
In this section, we deal with $C_4(T_{a,b})$ for $0\le a\le\dfrac{1}{2},\; 0<b\le\dfrac{1}{10}$.
Namely, we prove the following theorem.

\begin{theorem} \label{C4-asymptotic}
For $0\le a\le\dfrac{1}{2},\; 0<b\le\dfrac{1}{10}$,
$$
	C_4(T_{a,b}) < K_4(T_{a,b})
$$
holds.
\end{theorem}
\begin{proof}\quad
Considering Theorem~\ref{BoundbyFiniteResult}, in particular the case $n=2$,
$$
	C_4(T_{a,b})\le\sqrt{C_4^{(2)}(T_{a,b})^2+\frac{C_2(T_{a,b})^2}{4}}
$$
holds.
Therefore, from Theorem~\ref{small_b},
$$
	C_4(T_{a,b})^2<C_4^{(2)}(T_{a,b})^2+\frac{L_2(a,b)}{4}
$$
holds. Consequently, it holds that
\begin{equation}
	C_4^{(2)}(T_{a,b})^2\le L_4(a,b)-\frac{L_2(a,b)}{4}\quad \Longrightarrow\quad C_4(T_{a,b})<K_4(T_{a,b}).\label{C4K4}
\end{equation}

Now, let
$$
	c_1 = \frac{2((1-a)^2+b^2)}{b},
	\qquad c_2 = \frac{2(a^2+b^2)}{b},
	\qquad c_3 = \frac{2}{b},
	\qquad S=\frac{b}{2}.
$$
Then, we have
\begin{align*}
	L_2(a,b)&=S\left(\frac{c_1+c_2+c_3}{54}-\frac{1}{2c_1c_2c_3}\right),\\
	L_4(a,b)&=S\left(\frac{c_1c_2c_3}{16}-\frac{c_1+c_2+c_3}{30}-\frac{1}{5}\left(\frac{1}{c_1}+\frac{1}{c_2}+\frac{1}{c_3}\right)\right).
\end{align*}
Therefore, it holds that
\begin{align}
	L_4(a,b)-\frac{L_2(a,b)}{4}&=S\left(\frac{c_1c_2c_3}{16}-\frac{41(c_1+c_2+c_3)}{1080}-\frac{1}{5}\left(\frac{1}{c_1}+\frac{1}{c_2}+\frac{1}{c_3}\right)+\frac{1}{2c_1c_2c_3}\right) \notag \\
	&>S\left(\frac{c_1c_2c_3}{16}-\frac{41(c_1+c_2+c_3)}{1080}-\frac{1}{5}\left(\frac{1}{c_1}+\frac{1}{c_2}+\frac{1}{c_3}\right)\right). \label{L4L2}
\end{align}
Here, let
\begin{align*}
	G_1&=F^{(\beta)}_1(S/4, 0, u_3, u_2, w_7, w_5, w_3)
		+ F^{(\beta)}_1(S/4, u_3, 0, u_1, w_1, w_8, w_6) \\
	&\qquad\qquad + F^{(\beta)}_1(S/4, u_2, u_1, 0, w_4, w_2, w_9)
		+ F^{(\beta)}_1(S/4, u_1, u_2, u_3, -w_7, -w_8, -w_9), \\
	G_2&=F^{(\beta)}_2(S/4, 0, u_3, u_2, w_7, w_5, w_3)
		+ F^{(\beta)}_2(S/4, u_3, 0, u_1, w_1, w_8, w_6) \\
	&\qquad\qquad + F^{(\beta)}_2(S/4, u_2, u_1, 0, w_4, w_2, w_9)
		+ F^{(\beta)}_2(S/4, u_1, u_2, u_3, -w_7, -w_8, -w_9),
\end{align*}
(see Lemma~\ref{F2} for the definition of $F^{(\beta)}_1$ and $F^{(\beta)}_2$).
Then, from \eqref{C_4uw},\;\eqref{positive_definite},\;\eqref{C4K4} and \eqref{L4L2}, if we can prove that the following
inequality holds,
$$
	S\left(\frac{c_1c_2c_3}{16}-\frac{41(c_1+c_2+c_3)}{1080}-\frac{1}{5}\left(\frac{1}{c_1}+\frac{1}{c_2}+\frac{1}{c_3}\right)\right)G_2 - G_1
		\ge 0
$$
then $C_4(T_{a,b})<K_4(T_{a,b})$ will be shown.

The problem comes down to that of showing the non-negativeness of a multivariate polynomial in 14 variables,
which is not intrinsically difficult. However, since the formula transformation is very complicated, it is presented
in Lemma~\ref{quodoratic-form} in Appendix~\ref{Appendix1}.
This allows us to prove the theorem.
\end{proof}

\section{Numerical results} \label{Numerical_Results}
In this section, we present numerical results to show the validity of the formulas $K_j(T)$.
We note that the values obtained in this section are the results of approximate calculations and have not been explicitly verified,
as the goal here is to demonstrate that our formulas fit well.
We show the values of the upper bounds for $C_j(T)$ obtained using Theorem~\ref{BoundbyFiniteResult},
those of $K_j(T)$ in Theorem~\ref{formula}, and those of $C_j(T)$ themselves.
The upper bounds for $C_j(T)$ are obtained as follows: 
\begin{align*}
	\overline{C}_1^{(n)}(T)&=\sqrt{\frac{n^2}{n^2-1}}\;C_1^{(n)}(T), &
	\overline{C}_2^{(n)}(T)&=\sqrt{\frac{n^2}{n^2-1}}\;C_2^{(n)}(T), \\
	\overline{C}_3^{(n)}(T)&=\sqrt{\frac{n^4}{n^4-1}}\;C_3^{(n)}(T), &
	\overline{C}_4^{(n)}(T)&=\sqrt{C_4^{(n)}(T)^2+\frac{\overline{C}_2^{(n)}(T)^2}{n^2}}.
\end{align*}
We compute $C^{(n)}_j(T),\;j=1,2,3,4,\; n=10, 20$ via MATLAB R2024b.
For $C_j(T)$ themselves, we cannot determine their values analytically.
Therefore, we compute the following values.
\begin{align*}
	\widetilde{C}_1(T)&=\sup_{u\in V^{1,1}(T)\cap\mathcal{P}_{10}\setminus 0}\frac{\|u\|_{L^2(T)}}{\|\nabla u\|_{L^2(T)}}, &
	\widetilde{C}_2(T)&=\sup_{u\in V^{1,2}(T)\cap\mathcal{P}_{10}\setminus 0}\frac{\|u\|_{L^2(T)}}{\|\nabla u\|_{L^2(T)}}, \\
	\widetilde{C}_3(T)&=\sup_{u\in V^2(T)\cap\mathcal{P}_{10}\setminus 0}\frac{\|u\|_{L^2(T)}}{|u|_{H^2(T)}}, &
	\widetilde{C}_4(T)&=\sup_{u\in V^2(T)\cap\mathcal{P}_{10}\setminus 0}\frac{\|\nabla u\|_{L^2(T)}}{|u|_{H^2(T)}},
\end{align*}
where $\mathcal{P}_{10}$ denotes the space of polynomials in two variables of degree at most $10$.

We present the numerical results in Table~\ref{Table:C_1} to Table~\ref{Table:C_4};
all numerical results are rounded up to seven decimal places.
Note that $T_{a,b},\;0\le a\le\dfrac{1}{2},\;0<b\le1$ provides all shapes of triangles and, due to the scaling property,
the relative error between the upper bound
and the optimal value depends only on the shape of the triangle.
We show the graphs of $\widetilde{C}_j(T_{a,b})$ and $K_j(T_{a,b})-\widetilde{C}_j(T_{a,b})$
in Fig.~\ref{FigC1} to Fig.~\ref{FigK4}.

These numerical results demonstrate that sharp and explicit upper bounds are obtained
using the formulas introduced in Theorem~\ref{formula}.

\begin{table}[!t]
	\centering
\begin{tabular}{p{40pt}p{60pt}p{50pt}p{50pt}p{50pt}p{50pt}}
\hline
 $T$ & Shape & $K_1(T)$ & $\overline{C}_1^{(10)}(T)$ & $\overline{C}_1^{(20)}(T)$ & $\widetilde{C}_1(T)$ \\
\hline
 $T_{0,\,1}$ & \begin{tikzpicture}[scale=0.36]\draw (0,0)--(1,0)--(0,1)--cycle;\end{tikzpicture}\rule{0pt}{5mm}\;{\tiny \shortstack[l]{Isosceles right \\[-2pt] triangle}}
	& $0.3340766$ & $0.3212289$ & $0.3190362$ & $0.3183099$ \\
 $T_{0,\,1/2}$ & \begin{tikzpicture}[scale=0.36]\draw (0,0)--(1.6,0)--(0,0.8)--cycle;\end{tikzpicture}\rule{0pt}{4mm}
	& $0.2771024$ & $0.2740806$ & $0.2723722$ & $0.2718063$ \\
 $T_{0,\,1/5}$ & \begin{tikzpicture}[scale=0.36]\draw (0,0)--(3,0)--(0,0.6)--cycle;\end{tikzpicture}\rule{0pt}{4mm}
	& $0.2681079$ & $0.2648395$ & $0.2632357$ & $0.2627046$ \\
 $T_{0,\,1/10}$ & \begin{tikzpicture}[scale=0.36]\draw (0,0)--(4,0)--(0,0.4)--cycle;\end{tikzpicture}\rule{0pt}{4mm}
	& $0.2674398$ & $0.2635352$ & $0.2619417$ & $0.2614141$ \\
 $T_{1/4,\,1}$ & \begin{tikzpicture}[scale=0.36]\draw (0,0)--(1,0)--(0.25,1)--cycle;\end{tikzpicture}\rule{0pt}{4mm}
	& $0.3030136$ & $0.2911751$ & $0.2892957$ & $0.2886729$ \\
 $T_{1/4,\,1/2}$ & \begin{tikzpicture}[scale=0.36]\draw (0,0)--(1.6,0)--(0.4,0.8)--cycle;\end{tikzpicture}\rule{0pt}{4mm}
	& $0.2459842$ & $0.2436089$ & $0.2420929$ & $0.2415907$ \\
 $T_{1/4,\,1/5}$ & \begin{tikzpicture}[scale=0.36]\draw (0,0)--(3,0)--(0.75,0.6)--cycle;\end{tikzpicture}\rule{0pt}{4mm}
	& $0.2434617$ & $0.2329771$ & $0.2312898$ & $0.2307191$ \\
 $T_{1/4,\,1/10}$ & \begin{tikzpicture}[scale=0.36]\draw (0,0)--(4,0)--(1,0.4)--cycle;\end{tikzpicture}\rule{0pt}{4mm}
	& $0.2420732$ & $0.2310302$ & $0.2292243$ & $0.2285776$ \\
 $T_{1/2,\sqrt{3}/2}$ & \begin{tikzpicture}[scale=0.36]\draw (0,0)--(1.154,0)--(0.577,1)--cycle;\end{tikzpicture}\rule{0pt}{4mm}\;{\tiny \shortstack[l]{Equilateral \\[-2pt] triangle}}
	& $0.2683032$ & $0.2408093$ & $0.2392497$ & $0.2387324$ \\
 $T_{1/2,\,1/2}$ & \begin{tikzpicture}[scale=0.36]\draw (0,0)--(1.6,0)--(0.8,0.8)--cycle;\end{tikzpicture}\rule{0pt}{4mm}\;{\tiny \shortstack[l]{Isosceles right \\[-2pt] triangle}}\hspace*{-10pt}
	& $0.2362278$ & $0.2271431$ & $0.2255926$ & $0.2250791$ \\
 $T_{1/2,\,1/5}$ & \begin{tikzpicture}[scale=0.36]\draw (0,0)--(3,0)--(1.5,0.6)--cycle;\end{tikzpicture}\rule{0pt}{4mm}
	& $0.2350309$ & $0.2150884$ & $0.2129925$ & $0.2122504$ \\
 $T_{1/2,\,1/10}$ & \begin{tikzpicture}[scale=0.36]\draw (0,0)--(4,0)--(2,0.4)--cycle;\end{tikzpicture}\rule[-2mm]{0pt}{6mm}
	& $0.2327945$ & $0.2124694$ & $0.2100806$ & $0.2091369$ \\
\hline
\end{tabular}
\caption{Numerical results for $C_1(T)$}
\label{Table:C_1}
\end{table}

\begin{table}[!t]
	\centering
\begin{tabular}{p{40pt}p{60pt}p{50pt}p{50pt}p{50pt}p{50pt}}
\hline
 $T$ & Shape & $K_2(T)$ & $\overline{C}_2^{(10)}(T)$ & $\overline{C}_2^{(20)}(T)$ & $\widetilde{C}_2(T)$ \\
\hline
 $T_{0,\,1}$ & \begin{tikzpicture}[scale=0.36]\draw (0,0)--(1,0)--(0,1)--cycle;\end{tikzpicture}\rule{0pt}{5mm}\;{\tiny \shortstack[l]{Isosceles right \\[-2pt] triangle}}
	& $0.2417624$ & $0.2396038$ & $0.2381772$ & $0.2377024$ \\
 $T_{0,\,1/2}$ & \begin{tikzpicture}[scale=0.36]\draw (0,0)--(1.6,0)--(0,0.8)--cycle;\end{tikzpicture}\rule{0pt}{4mm}
	& $0.2001157$ & $0.1998408$ & $0.1985657$ & $0.1981417$ \\
 $T_{0,\,1/5}$ & \begin{tikzpicture}[scale=0.36]\draw (0,0)--(3,0)--(0,0.6)--cycle;\end{tikzpicture}\rule{0pt}{4mm}
	& $0.1931750$ & $0.1916920$ & $0.1904436$ & $0.1900287$ \\
 $T_{0,\,1/10}$ & \begin{tikzpicture}[scale=0.36]\draw (0,0)--(4,0)--(0,0.4)--cycle;\end{tikzpicture}\rule{0pt}{4mm}
	& $0.1926084$ & $0.1906411$ & $0.1893971$ & $0.1889838$ \\
 $T_{1/4,\,1}$ & \begin{tikzpicture}[scale=0.36]\draw (0,0)--(1,0)--(0.25,1)--cycle;\end{tikzpicture}\rule{0pt}{4mm}
	& $0.2197865$ & $0.2177021$ & $0.2164123$ & $0.2159829$ \\
 $T_{1/4,\,1/2}$ & \begin{tikzpicture}[scale=0.36]\draw (0,0)--(1.6,0)--(0.4,0.8)--cycle;\end{tikzpicture}\rule{0pt}{4mm}
	& $0.1779313$ & $0.1782024$ & $0.1770818$ & $0.1767091$ \\
 $T_{1/4,\,1/5}$ & \begin{tikzpicture}[scale=0.36]\draw (0,0)--(3,0)--(0.75,0.6)--cycle;\end{tikzpicture}\rule{0pt}{4mm}
	& $0.1753979$ & $0.1720157$ & $0.1709010$ & $0.1705285$ \\
 $T_{1/4,\,1/10}$ & \begin{tikzpicture}[scale=0.36]\draw (0,0)--(4,0)--(1,0.4)--cycle;\end{tikzpicture}\rule{0pt}{4mm}
	& $0.1743206$ & $0.1711857$ & $0.1700506$ & $0.1696650$ \\
 $T_{1/2,\sqrt{3}/2}$ & \begin{tikzpicture}[scale=0.36]\draw (0,0)--(1.154,0)--(0.577,1)--cycle;\end{tikzpicture}\rule{0pt}{4mm}\;{\tiny \shortstack[l]{Equilateral \\[-2pt] triangle}}
	& $0.1948780$ & $0.1906371$ & $0.1895418$ & $0.1891769$ \\
 $T_{1/2,\,1/2}$ & \begin{tikzpicture}[scale=0.36]\draw (0,0)--(1.6,0)--(0.8,0.8)--cycle;\end{tikzpicture}\rule{0pt}{4mm}\;{\tiny \shortstack[l]{Isosceles right \\[-2pt] triangle}}\hspace*{-10pt}
	& $0.1709519$ & $0.1694255$ & $0.1684167$ & $0.1680810$ \\
 $T_{1/2,\,1/5}$ & \begin{tikzpicture}[scale=0.36]\draw (0,0)--(3,0)--(1.5,0.6)--cycle;\end{tikzpicture}\rule{0pt}{4mm}
	& $0.1693066$ & $0.1645692$ & $0.1635627$ & $0.1632275$ \\
 $T_{1/2,\,1/10}$ & \begin{tikzpicture}[scale=0.36]\draw (0,0)--(4,0)--(2,0.4)--cycle;\end{tikzpicture}\rule[-2mm]{0pt}{6mm}
	& $0.1676363$ & $0.1638829$ & $0.1628606$ & $0.1625185$ \\
\hline
\end{tabular}
\caption{Numerical results for $C_2(T)$}
\label{Table:C_2}
\end{table}

\begin{table}[!t]
	\centering
\begin{tabular}{p{40pt}p{60pt}p{50pt}p{50pt}p{50pt}p{50pt}}
\hline
 $T$ & Shape & $K_3(T)$ & $\overline{C}_3^{(10)}(T)$ & $\overline{C}_3^{(20)}(T)$ & $\widetilde{C}_3(T)$ \\
\hline
 $T_{0,\,1}$ & \begin{tikzpicture}[scale=0.36]\draw (0,0)--(1,0)--(0,1)--cycle;\end{tikzpicture}\rule{0pt}{5mm}\;{\tiny \shortstack[l]{Isosceles right \\[-2pt] triangle}}
	& $0.1702673$ & $0.1684445$ & $0.1675538$ & $0.1672535$ \\
 $T_{0,\,1/2}$ & \begin{tikzpicture}[scale=0.36]\draw (0,0)--(1.6,0)--(0,0.8)--cycle;\end{tikzpicture}\rule{0pt}{4mm}
	& $0.1184266$ & $0.1180689$ & $0.1175454$ & $0.1173697$ \\
 $T_{0,\,1/5}$ & \begin{tikzpicture}[scale=0.36]\draw (0,0)--(3,0)--(0,0.6)--cycle;\end{tikzpicture}\rule{0pt}{4mm}
	& $0.1107396$ & $0.1096648$ & $0.1092457$ & $0.1091055$ \\
 $T_{0,\,1/10}$ & \begin{tikzpicture}[scale=0.36]\draw (0,0)--(4,0)--(0,0.4)--cycle;\end{tikzpicture}\rule{0pt}{4mm}
	& $0.1099925$ & $0.1087203$ & $0.1083185$ & $0.1081842$ \\
 $T_{1/4,\,1}$ & \begin{tikzpicture}[scale=0.36]\draw (0,0)--(1,0)--(0.25,1)--cycle;\end{tikzpicture}\rule{0pt}{4mm}
	& $0.1487598$ & $0.1464850$ & $0.1458511$ & $0.1456392$ \\
 $T_{1/4,\,1/2}$ & \begin{tikzpicture}[scale=0.36]\draw (0,0)--(1.6,0)--(0.4,0.8)--cycle;\end{tikzpicture}\rule{0pt}{4mm}
	& $0.0950295$ & $0.0946780$ & $0.0942615$ & $0.0941222$ \\
 $T_{1/4,\,1/5}$ & \begin{tikzpicture}[scale=0.36]\draw (0,0)--(3,0)--(0.75,0.6)--cycle;\end{tikzpicture}\rule{0pt}{4mm}
	& $0.0855112$ & $0.0849795$ & $0.0844706$ & $0.0842822$ \\
 $T_{1/4,\,1/10}$ & \begin{tikzpicture}[scale=0.36]\draw (0,0)--(4,0)--(1,0.4)--cycle;\end{tikzpicture}\rule{0pt}{4mm}
	& $0.0843544$ & $0.0837110$ & $0.0831604$ & $0.0829349$ \\
 $T_{1/2,\sqrt{3}/2}$ & \begin{tikzpicture}[scale=0.36]\draw (0,0)--(1.154,0)--(0.577,1)--cycle;\end{tikzpicture}\rule{0pt}{4mm}\;{\tiny \shortstack[l]{Equilateral \\[-2pt] triangle}}
	& $0.1201798$ & $0.1177043$ & $0.1172419$ & $0.1170871$ \\
 $T_{1/2,\,1/2}$ & \begin{tikzpicture}[scale=0.36]\draw (0,0)--(1.6,0)--(0.8,0.8)--cycle;\end{tikzpicture}\rule{0pt}{4mm}\;{\tiny \shortstack[l]{Isosceles right \\[-2pt] triangle}}\hspace*{-10pt}
	& $0.0851337$ & $0.0842223$ & $0.0837769$ & $0.0836268$ \\
 $T_{1/2,\,1/5}$ & \begin{tikzpicture}[scale=0.36]\draw (0,0)--(3,0)--(1.5,0.6)--cycle;\end{tikzpicture}\rule{0pt}{4mm}
	& $0.0732578$ & $0.0727067$ & $0.0719785$ & $0.0716838$ \\
 $T_{1/2,\,1/10}$ & \begin{tikzpicture}[scale=0.36]\draw (0,0)--(4,0)--(2,0.4)--cycle;\end{tikzpicture}\rule[-2mm]{0pt}{6mm}
	& $0.0715701$ & $0.0710650$ & $0.0702397$ & $0.0698653$ \\
\hline
\end{tabular}
\caption{Numerical results for $C_3(T)$}
\label{Table:C_3}
\end{table}

\begin{table}[t!]
	\centering
\begin{tabular}{p{40pt}p{60pt}p{50pt}p{50pt}p{50pt}p{50pt}p{50pt}}
\hline
 $T$ & Shape & $K_4(T)$ & $\overline{C}_4^{(10)}(T)$ & $\overline{C}_4^{(20)}(T)$ & $\widetilde{C}_4(T)$ \\
\hline
 $T_{0,\,1}$ & \begin{tikzpicture}[scale=0.36]\draw (0,0)--(1,0)--(0,1)--cycle;\end{tikzpicture}\rule{0pt}{5mm}\;{\tiny \shortstack[l]{Isosceles right \\[-2pt] triangle}}
	& $0.4915960$ & $0.4894003$ & $0.4888905$ & $0.4887224$ \\
 $T_{0,\,1/2}$ & \begin{tikzpicture}[scale=0.36]\draw (0,0)--(1.6,0)--(0,0.8)--cycle;\end{tikzpicture}\rule{0pt}{4mm}
	& $0.3958114$ & $0.3813624$ & $0.3809003$ & $0.3807481$ \\
 $T_{0,\,1/5}$ & \begin{tikzpicture}[scale=0.36]\draw (0,0)--(3,0)--(0,0.6)--cycle;\end{tikzpicture}\rule{0pt}{4mm}
	& $0.3697886$ & $0.3372741$ & $0.3367581$ & $0.3365882$ \\
 $T_{0,\,1/10}$ & \begin{tikzpicture}[scale=0.36]\draw (0,0)--(4,0)--(0,0.4)--cycle;\end{tikzpicture}\rule{0pt}{4mm}
	& $0.3662944$ & $0.3286113$ & $0.3280651$ & $0.3278853$ \\
 $T_{1/4,\,1}$ & \begin{tikzpicture}[scale=0.36]\draw (0,0)--(1,0)--(0.25,1)--cycle;\end{tikzpicture}\rule{0pt}{4mm}
	& $0.4063827$ & $0.3969773$ & $0.3964682$ & $0.3963006$ \\
 $T_{1/4,\,1/2}$ & \begin{tikzpicture}[scale=0.36]\draw (0,0)--(1.6,0)--(0.4,0.8)--cycle;\end{tikzpicture}\rule{0pt}{4mm}
	& $0.3393940$ & $0.3262145$ & $0.3257825$ & $0.3256403$ \\
 $T_{1/4,\,1/5}$ & \begin{tikzpicture}[scale=0.36]\draw (0,0)--(3,0)--(0.75,0.6)--cycle;\end{tikzpicture}\rule{0pt}{4mm}
	& $0.5516444$ & $0.5391173$ & $0.5389130$ & $0.5388449$ \\
 $T_{1/4,\,1/10}$ & \begin{tikzpicture}[scale=0.36]\draw (0,0)--(4,0)--(1,0.4)--cycle;\end{tikzpicture}\rule{0pt}{4mm}
	& $0.9871945$ & $0.9749195$ & $0.9748213$ & $0.9747887$ \\
 $T_{1/2,\sqrt{3}/2}$ & \begin{tikzpicture}[scale=0.36]\draw (0,0)--(1.154,0)--(0.577,1)--cycle;\end{tikzpicture}\rule{0pt}{4mm}\;{\tiny \shortstack[l]{Equilateral \\[-2pt] triangle}}
	& $0.3476109$ & $0.3189929$ & $0.3185476$ & $0.3184012$ \\
 $T_{1/2,\,1/2}$ & \begin{tikzpicture}[scale=0.36]\draw (0,0)--(1.6,0)--(0.8,0.8)--cycle;\end{tikzpicture}\rule{0pt}{4mm}\;{\tiny \shortstack[l]{Isosceles right \\[-2pt] triangle}}\hspace*{-10pt}
	& $0.3476109$ & $0.3460583$ & $0.3456978$ & $0.3455789$ \\
 $T_{1/2,\,1/5}$ & \begin{tikzpicture}[scale=0.36]\draw (0,0)--(3,0)--(1.5,0.6)--cycle;\end{tikzpicture}\rule{0pt}{4mm}
	& $0.6761399$ & $0.6631990$ & $0.6630530$ & $0.6630039$ \\
 $T_{1/2,\,1/10}$ & \begin{tikzpicture}[scale=0.36]\draw (0,0)--(4,0)--(2,0.4)--cycle;\end{tikzpicture}\rule[-2mm]{0pt}{6mm}
	& $1.2786662$ & $1.2689186$ & $1.2688509$ & $1.2688285$ \\
\hline
\end{tabular}
\caption{Numerical results for $C_4(T)$}
\label{Table:C_4}
\end{table}

\begin{figure}[t!]
\begin{minipage}[b]{0.49\hsize}
	\centering
	\includegraphics[width=0.8\hsize]{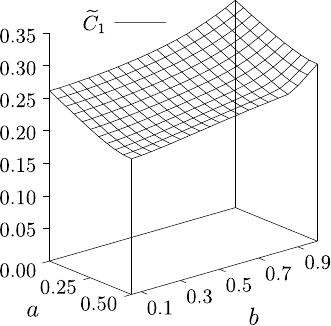}
	\caption{Graph of $\widetilde{C}_1$.}
	\label{FigC1}
\end{minipage}
\begin{minipage}[b]{0.49\hsize}
	\centering
	\includegraphics[width=0.8\hsize]{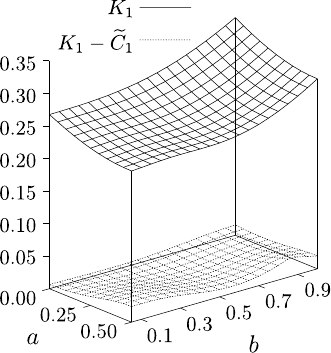}
	\caption{Graph of $K_1$ and $K_1-\widetilde{C}_1$.}
	\label{FigK1}
\end{minipage}
\bigskip
\medskip
\end{figure}

\begin{figure}[t!]
\begin{minipage}[b]{0.49\hsize}	
	\centering
	\includegraphics[width=0.8\hsize]{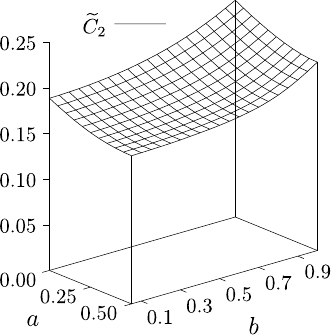}
	\caption{Graph of $\widetilde{C}_2$.}
	\label{FigC2}
\end{minipage}
\begin{minipage}[b]{0.49\hsize}
	\centering
	\includegraphics[width=0.8\hsize]{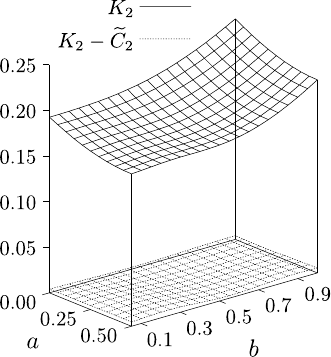}
	\caption{Graph of $K_2$ and $K_2-\widetilde{C}_2$.}
	\label{FigK2}
\end{minipage}
\bigskip
\medskip
\end{figure}

\begin{figure}[t!]
\begin{minipage}[b]{0.49\hsize}
	\centering
	\includegraphics[width=0.8\hsize]{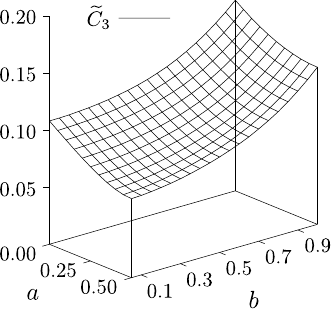}
	\caption{Graph of $\widetilde{C}_3$.}
	\label{FigC3}
\end{minipage}
\begin{minipage}[b]{0.49\hsize}
	\centering
	\includegraphics[width=0.8\hsize]{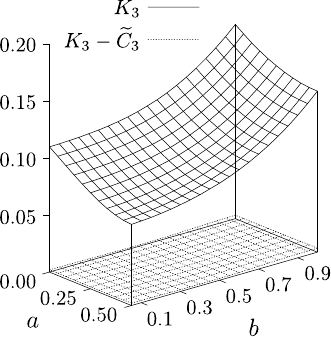}
	\caption{Graph of $K_3$ and $K_3-\widetilde{C}_3$.}
	\label{FigK3}
\end{minipage}
\bigskip
\medskip
\end{figure}

\begin{figure}[t!]
\begin{minipage}[b]{0.49\hsize}
	\centering
	\includegraphics[width=0.8\hsize]{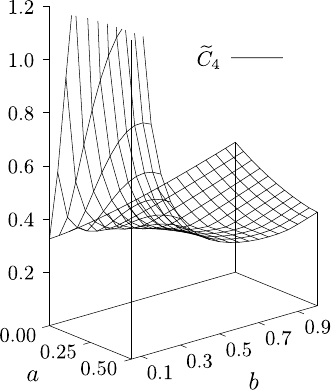}
	\caption{Graph of $\widetilde{C}_4$.}
	\label{FigC4}
\end{minipage}
\begin{minipage}[b]{0.49\hsize}
	\centering
	\includegraphics[width=0.8\hsize]{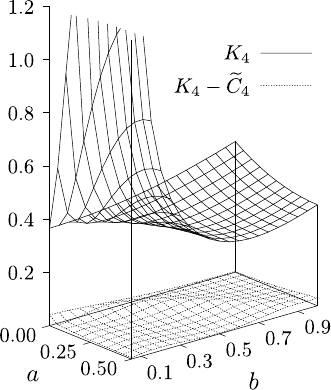}
	\caption{Graph of $K_4$ and $K_4-\widetilde{C}_4$.}
	\label{FigK4}
\end{minipage}
\bigskip
\medskip
\end{figure}

\newpage

\quad \\

\newpage

\quad \\

\newpage

\quad \\

\newpage

\section{Circumradius and $C_4(T)$} \label{Circumradius_Condition}
Theorem~\ref{formula} claims that the following estimate holds
for the interpolation constant $C_4(T)$.
$$
	C_4(T)<K_4(T)=\sqrt{\frac{A^2B^2C^2}{16S^2}-\frac{A^2+B^2+C^2}{30}-\frac{S^2}{5}\left(\frac{1}{A^2}+\frac{1}{B^2}+\frac{1}{C^2}\right)},
$$
where $A,B,C$ are the edge lengths of triangle $T$ and $S$ is the area of $T$.
Since the circumradius of $T$ is given by
$$
	R(T)=\frac{ABC}{4S},
$$
we have the estimation
$$
	C_4(T)<R(T).
$$
This fact has many interesting implications for error analysis in the finite element method.
 See~\cite{KobayashiTsuchiya,KobayashiTsuchiya2} for details.

\section{Conclusion} \label{Conclusion}
Interpolation error constants are essential for analyzing interpolation error, particularly
in the context of error analysis in the finite element method.
We derived remarkable formulas that give sharp upper bounds
for the interpolation error constants over the given triangular elements.
These formulas were obtained via a numerical verification method and asymptotic analysis.
This study is significant in that it presents a series of procedures for obtaining mathematically rigorous formulas
through a combination of a numerical verification method and mathematical analysis.
Several issues remain to be addressed.
In this study, the upper bounds of $C_j(T)$ were obtained using $C_j^{(n)}(T)$;
however, we could not prove the convergence of $C_j^{(n)}(T)$ themselves. We hope to provide this proof in the future.
Obtaining similar sharp formulas for tetrahedra is another topic for future work.
However, unlike a triangle, a tetrahedron cannot be divided into tetrahedra that are similar to its original shape,
so the method used in this study cannot be applied.

\section*{Acknowledgements}
This work was supported by a Japan Society for the Promotion of Science Grant-in-Aid for Scientific Research (B)
(grant number 24K00538).
The author would like to express his deepest gratitude to Prof.~Takuya Tuchiya of Osaka University,
without whose encouragement and careful checking of the manuscript, this paper would not have been completed.

\appendix

\section{Lemmas} \label{Appendix1}
In this appendix, we provide proofs of the inequalities that are too lengthy to include in the main text.
Note once again that $d_k=1-ka$ and particularly that $d_1\ge0$ when $0\le a\le1$ and $d_2\ge0$
when $0\le a\le\dfrac{1}{2}$.

For subsequent lemmas, except for Lemma~\ref{positive-definiteness}, we provide
codes developed in MATLAB with Symbolic Math Toolbox for checking the correctness of the formula transformations.
See Code~\ref{LemmaA-1} to Code~\ref{LemmaA-11} in Appendix~\ref{Appendix2}.

Hereafter, let
$$
	\varphi(a,b)=\frac{b^2}{a^2+b^2},\qquad
	\psi(a,b)=\frac{b^2(1+2a)}{1+4b^2}\cdot\varphi(a,b),\qquad
	\omega(a,b)=\frac{a(1-a)}{1-a+a^2+b^2}.
$$

We first prove a number of lemmas.
\begin{lemma} \label{appendix-phi}
For $0\le a,\; 0<b$, it holds that
\begin{align*}
	0&\le\varphi(a,b)\le 1, \\
	-\frac{2}{3b}&\le\varphi_a(a,b)\le 0, & -\frac{2}{b^2}&\le\varphi_{aa}(a,b), \\
	0&\le\varphi_b(a,b)\le\frac{1}{2b}, &	-\frac{7}{9b^2}&\le\varphi_{bb}(a,b).
\end{align*}
\end{lemma}

\begin{proof}\quad
$0\le\varphi(a,b)\le1$ is obvious.
\begin{align*}
	\varphi_a(a,b)\; &= \frac{-2ab^2}{(a^2+b^2)^2}
		\le 0, \\
	\varphi_a(a,b)\; &= \frac{-2ab^2}{(a^2+b^2)^2}= -\frac{2}{3b}+\frac{2(a-b)^4+2ab(2a-b)^2}{3b(a^2+b^2)^2}
		\ge -\frac{2}{3b}, \\
	\varphi_{aa}(a,b) &= \frac{2b^2(3a^2-b^2)}{(a^2 + b^2)^3}
		\ge \frac{-2b^4}{b^6}=-\frac{2}{b^2}, \\
	\varphi_b(a,b)\; &= \frac{2a^2b}{(a^2+b^2)^2}
		\le \frac{2a^2b}{4a^2b^2}=\frac{1}{2b}, \\
	\varphi_b(a,b)\; &= \frac{2a^2b}{(a^2+b^2)^2}
		\ge 0, \\
	\varphi_{bb}(a,b) &= \frac{2a^2(a^2-3b^2)}{(a^2 + b^2)^3} \\
		&= -\frac{7}{9b^2}+\frac{3a^2b^2(a^2+b^2)+(7a^2+3b^2)(2a^2-b^2)^2+b^2(13a^2-5b^2)^2}{36b^2(a^2 + b^2)^3} \\
		&\ge -\frac{7}{9b^2}.
\end{align*}
\end{proof}

\begin{lemma} \label{appendix-psi}
For $0\le a,\; 0<b$, it holds that
\begin{align*}
	0&\le\psi(a,b)\le\frac{b^2(1+2a)}{1+4b^2}, \\
	-\frac{2b(1+2a)}{3(1+4b^2)}&\le\psi_a(a,b)\le\frac{2b^2}{1+4b^2},
		& -\frac{2(3+6a+4b)}{3(1+4b^2)}&\le\psi_{aa}(a,b), \\
	0&\le\psi_b(a,b)\le\frac{(1+2a)(2+16b-9b^2)}{10(1+4b^2)},
		& -\frac{21(1+2a)}{11(1+4b^2)}&\le\psi_{bb}(a,b).
\end{align*}
\end{lemma}

\begin{proof}\quad
From Lemma~\ref{appendix-phi}, we can immediately obtain
$$
	0\le\psi(a,b)\le\frac{b^2(1+2a)}{1+4b^2}.
$$
Furthermore, we have
\begin{align*}
	\psi_a(a,b)\; &= \frac{2b^2}{1+4b^2}\cdot\varphi(a,b)+\frac{b^2(1+2a)}{1+4b^2}\cdot\varphi_a(a,b)
		\le \frac{2b^2}{1+4b^2}, \\
	\psi_a(a,b)\; &= \frac{2b^2}{1+4b^2}\cdot\varphi(a,b)+\frac{b^2(1+2a)}{1+4b^2}\cdot\varphi_a(a,b)
		\ge \frac{b^2(1+2a)}{1+4b^2}\cdot\frac{-2}{3b}
		= -\frac{2b(1+2a)}{3(1+4b^2)}, \\
	\psi_{aa}(a,b) &= \frac{4b^2}{1+4b^2}\cdot\varphi_a(a,b)+\frac{b^2(1+2a)}{1+4b^2}\cdot\varphi_{aa}(a,b)\\
		&\ge \frac{4b^2}{1+4b^2}\cdot\frac{-2}{3b}+\frac{b^2(1+2a)}{1+4b^2}\cdot\frac{-2}{b^2}
		=-\frac{2(3+6a+4b)}{3(1+4b^2)}, \\
	\psi_b(a,b)\;&=\frac{2b(1+2a)}{(1+4b^2)^2}\cdot\varphi(a,b)+\frac{b^2(1+2a)}{1+4b^2}\cdot\varphi_b(a,b)\le \frac{2b(1+2a)}{(1+4b^2)^2}+\frac{b^2(1+2a)}{1+4b^2}\cdot\frac{1}{2b} \\
		&= \frac{(1+2a)(2+16b-9b^2)}{10(1+4b^2)} \\
			&\qquad\qquad - \frac{(1+2a)(1-b)}{50(1+4b^2)^2}\Big(11b^3+10(1-3b)^2+b(5-13b)^2\Big) \\
		&\le \frac{(1+2a)(2+16b-9b^2)}{10(1+4b^2)}, \\
	\psi_b(a,b)\;&=\frac{2b(1+2a)}{(1+4b^2)^2}\cdot\varphi(a,b)+\frac{b^2(1+2a)}{1+4b^2}\cdot\varphi_b(a,b)
		\ge 0, \\
	\psi_{bb}(a,b)&=\frac{2(1+2a)(1-12b^2)}{(1+4b^2)^3}\cdot\varphi(a,b)
			+\frac{4b(1+2a)}{(1+4b^2)^2}\cdot\varphi_b(a,b) \\
			&\qquad\qquad +\frac{b^2(1+2a)}{1+4b^2}\cdot\varphi_{bb}(a,b) \\
		&\ge\frac{2(1+2a)(1-12b^2-(5/4-3b^2)^2)}{(1+4b^2)^3}\cdot\varphi(a,b)
			-\frac{b^2(1+2a)}{1+4b^2}\cdot\frac{7}{9b^2} \\
		&=-\frac{9(1+2a)}{8(1+4b^2)}\cdot\varphi(a,b)-\frac{7(1+2a)}{9(1+4b^2)} \\
		&\ge -\frac{9(1+2a)}{8(1+4b^2)}-\frac{7(1+2a)}{9(1+4b^2)}
			=-\frac{137(1+2a)}{72(1+4b^2)} \\
		&\ge -\frac{21(1+2a)}{11(1+4b^2)}.
\end{align*}
\end{proof}

\begin{lemma} \label{appendix-omega}
For $0\le a\le\dfrac{1}{2},\; 0<b\le 1$, it holds that
\begin{align*}
	&\omega(a,b)\le\frac{4a(1-a)(7-4b^2)}{21}, \\
	\frac{1-2a}{2}\le\;&\omega_a(a,b)\le\frac{4(1-2a)(2-b^2+5a(1-a))}{7}, \\
	-\frac{2(11a(1-a)+2b^2)}{7b^2}\le\;&\omega_{aa}(a,b), \\
	-\frac{1}{6b}\le\;&\omega_b(a,b)\le-\frac{a(1-a)b}{2}, \\
	-\frac{24-64a(1-a)-17b^2}{44b^2}\le\;&\omega_{bb}(a,b).
\end{align*}
\end{lemma}

\begin{proof}\quad
These inequalities are obtained as follows:
\begin{align*}
	\omega(a,b)\;\;&=\frac{4a(1-a)(7-4b^2)}{21}-\frac{ad_1}{21(1-ad_1+b^2)}\Big(3d_2^2+4(1-b^2)(d_2^2+4b^2)\Big) \\
		&\le\frac{4a(1-a)(7-4b^2)}{21}, \\
	\omega_a(a,b)\;&= \frac{4(1-2a)(2-b^2+5a(1-a))}{7} \\
			&\qquad\qquad -\frac{d_2}{7(1-ad_1+b^2)^2}\Big(d_2^2+ ad_1(8d_2^2(1+4b^2)+28b^4)  \\
			&\qquad\qquad  + 4a^2d_1^2(5ad_1+21b^2) + b^2(5-4b^4)\Big) \\
		&\le \frac{4(1-2a)(2-b^2+5a(1-a))}{7}, \\
	\omega_a(a,b)\;&= \frac{(1-2a)(1+b^2)}{(1-ad_1+b^2)^2}
		\ge\frac{(1-2a)(1+b^2)}{(1+b^2)^2}
		=\frac{1-2a}{1+b^2}
		\ge\frac{1-2a}{2}, \\
	\omega_{aa}(a,b) &=-\frac{2(11a(1-a)+2b^2)}{7b^2}+\frac{2(1+b^2)}{7b^2(1-ad_1+b^2)^3}\Big\{ \\
			&\qquad\qquad \frac{a^3d_1^3(ad_1 + 3d_2^2 + b^2)}{1+b^2}
				+ a^2d_1^2(10ad_1 + 19d_2^2 + b^2) \\
			&\qquad\qquad + ad_1d_2^2(1 + 6d_2^2 + 7b^2) + ad_1(2-3b^2)^2 + d_2^2b^4 + 2b^2(1-b^2)^2\Big\} \\
		&\ge-\frac{2(11a(1-a)+2b^2)}{7b^2}, \\
	\omega_b(a,b)\;&=-\frac{2a(1-a)b}{(1-ad_1+b^2)^2}
		\le-\frac{2a(1-a)b}{(1+1)^2}
		=-\frac{a(1-a)b}{2}, \\
	\omega_b(a,b)\;&=-\frac{1}{6b}\left(1-\frac{(1-7ad_1+b^2)^2+12ad_1d_2^2}{(1-ad_1+b^2)^2}\right)
		\ge-\frac{1}{6b}, \\
	\omega_{bb}(a,b) &=-\frac{24-64a(1-a)-17b^2}{44b^2}
			+\frac{1}{44b^2(1-ad_1+b^2)^3}\Big\{ \\
			&\qquad\qquad 13d_2^2b^2 + a^3d_1^3(8+17b^2) + d_2^4\Big(1+6b^2 + (3-2ad_1+3b^2)^2\Big) \\
			&\qquad\qquad + a^2d_1^2\Big(4+10b^2 + 3b^2(1-b^2)\Big) + 2(6ad_1+7d_2^2)(1-b^2)(1-3b^2)^2 \\
			&\qquad\qquad + b^2(1-b^2)\Big(34d_2^2+3ad_1(13+27b^2)+82d_2^2(1-b^2)+17b^4\Big)\Big\} \\
		&\ge-\frac{24-64a(1-a)-17b^2}{44b^2}.
\end{align*}
\end{proof}

In the following Lemma~\ref{L1-bound} to Lemma~\ref{L4-bound},
\begin{align*}
	0\le &a\le\frac{1}{2},&\quad 0\le&\widetilde{a}\le\frac{1}{2},&\quad |\widetilde{a}-a|&\le \frac{b}{50}, \\
	0<&b\le 1, &\quad 0<&\widetilde{b}\le 1, &\quad |\widetilde{b}-b|&\le \frac{b}{50},
\end{align*}
are assumed. Note that in this case,
$$
	\frac{49}{50}b\le\widetilde{b}\le\frac{51}{50}b
$$
holds.

\begin{lemma} \label{L1-bound}
It holds that
\begin{align*}
	\left|\frac{\partial L_1}{\partial a}(a,b)\right|&< \frac{2}{b}L_1(a,b),
		& \frac{\partial^2 L_1}{\partial a^2}(\widetilde{a},b)&< \frac{5}{b^2}L_1(a,b), \\
	\left|\frac{\partial L_1}{\partial b}(a,b)\right|&< \frac{2}{b}L_1(a,b),
		& \frac{\partial^2 L_1}{\partial b^2}(a,\widetilde{b})&< \frac{4}{b^2}L_1(a,b).
\end{align*}
\end{lemma}

\begin{proof}\quad
From the definition of $L_1(a,b)$ and Lemma~\ref{appendix-psi}, we have
\begin{align*}
	L_1(a,b)\;\; &= \frac{1+a^2+(1-a)^2+2b^2}{28}
			- \frac{b^4}{16(a^2+b^2)((1-a)^2+b^2)} \\
		&= \frac{1-a+a^2+b^2}{14} - \frac{b^4(1+2a)}{16(a^2+b^2)(1+4b^2)} - \frac{b^4(3-2a)}{16((1-a)^2+b^2)(1+4b^2)} \\
		&= \frac{1-a+a^2+b^2}{14} - \frac{\psi(a,b)+\psi(1-a,b)}{16} \\
		&\ge \frac{1-a+a^2+b^2}{14} - \frac{1}{16}\left(\frac{b^2(1+2a)}{1+4b^2}+\frac{b^2(1+2(1-a))}{1+4b^2}\right) \\
		&=\frac{1-a+a^2+b^2}{14} - \frac{b^2}{4(1+4b^2)}, \\
	\frac{\partial L_1}{\partial a}(a,b)\; &= -\frac{1-2a}{14} - \frac{\psi_a(a,b)-\psi_a(1-a,b)}{16} \\
		&\ge -\frac{1-2a}{14} - \frac{1}{16}\left(\frac{2b^2}{1+4b^2}+\frac{2b(1+2(1-a))}{3(1+4b^2)}\right) \\
		&= -\frac{1-2a}{14} - \frac{b(3-2a+3b)}{24(1+4b^2)}, \\
	\frac{\partial L_1}{\partial a}(a,b)\; &= -\frac{1-2a}{14} - \frac{\psi_a(a,b)-\psi_a(1-a,b)}{16} \\
		&\le -\frac{1-2a}{14} - \frac{1}{16}\left(-\frac{2b(1+2a)}{3(1+4b^2)}-\frac{2b^2}{1+4b^2}\right) \\
		&= -\frac{1-2a}{14} + \frac{b(1+2a+3b)}{24(1+4b^2)}, \\
	\frac{\partial^2 L_1}{\partial a^2}(\widetilde{a},b) &= \frac{1}{7} - \frac{\psi_{aa}(\widetilde{a},b)+\psi_{aa}(1-\widetilde{a},b)}{16} \\
		&\le \frac{1}{7} - \frac{1}{16}\left(-\frac{2(3+6\widetilde{a}+4b)}{3(1+4b^2)}-\frac{2(3+6(1-\widetilde{a})+4b)}{3(1+4b^2)}\right) \\
		&= \frac{1}{7} + \frac{3+2b}{6(1+4b^2)}, \\
	\frac{\partial L_1}{\partial b}(a,b)\; &= \frac{b}{7} - \frac{\psi_b(a,b)+\psi_b(1-a,b)}{16} \\
		&\ge\frac{b}{7} - \frac{1}{16}\left(\frac{(1+2a)(2+16b-9b^2)}{10(1+4b^2)}+\frac{(1+2(1-a))(2+16b-9b^2)}{10(1+4b^2)}\right) \\
		&= \frac{b}{7}-\frac{2+16b-9b^2}{40(1+4b^2)}, \\
	\frac{\partial L_1}{\partial b}(a,b)\; &= \frac{b}{7} - \frac{\psi_b(a,b)+\psi_b(1-a,b)}{16}
		\le\frac{b}{7}, \\
	\frac{\partial^2 L_1}{\partial b^2}(a,\widetilde{b}) &= \frac{1}{7} - \frac{\psi_{bb}(a,\widetilde{b})+\psi_{bb}(1-a,\widetilde{b})}{16} \\
		&\le\frac{1}{7} - \frac{1}{16}\left(-\frac{21(1+2a)}{11(1+4\widetilde{b}^2)}-\frac{21(1+2(1-a))}{11(1+4\widetilde{b}^2)}\right) \\
		&= \frac{1}{7}+\frac{21}{44(1+4\widetilde{b}^2)}
		\le\frac{1}{7}+\frac{21}{44(49/50)^2(1+4b^2)} \\
		&\le\frac{1}{7}+\frac{1}{2(1+4b^2)}.
\end{align*}
Then,
\begin{align*}
	168(1+4b^2)&\left(2L_1(a,b)+b\cdot\frac{\partial L_1}{\partial a}(a,b)\right) \\
		&\ge 168(1+4b^2)\left\{2\left(\frac{1-a+a^2+b^2}{14} - \frac{b^2}{4(1+4b^2)}\right)\right. \\
			&\qquad\qquad\qquad\qquad\qquad\qquad\qquad
			\left. +b\left(-\frac{1-2a}{14} - \frac{b(3-2a+3b)}{24(1+4b^2)}\right)\right\} \\
		&= 6(d_2^2+4ab)(1+4b^2) + b^2\Big(3+14a+3b+46(1-b)^2\Big) + 2(3-b-5b^2)^2 \\
		&> 0, \\
	168(1+4b^2)&\left(2L_1(a,b)-b\cdot\frac{\partial L_1}{\partial a}(a,b)\right) \\
		&\ge 168(1+4b^2)\left\{2\left(\frac{1-a+a^2+b^2}{14} - \frac{b^2}{4(1+4b^2)}\right)\right. \\
			&\qquad\qquad\qquad\qquad\qquad\qquad\qquad
			\left. -b\left(-\frac{1-2a}{14} + \frac{b(1+2a+3b)}{24(1+4b^2)}\right)\right\} \\
		&= 6(d_2^2+4d_1b)(1+4b^2) + b^2\Big(3+14d_1+3b+46(1-b)^2\Big) + 2(3-b-5b^2)^2 \\
		&> 0, \\
	84(1+4b^2)&\left(5L_1(a,b)-b^2\cdot\frac{\partial^2 L_1}{\partial a^2}(\widetilde{a},b)\right) \\
		&\ge 84(1+4b^2)\left\{5\left(\frac{1-a+a^2+b^2}{14} - \frac{b^2}{4(1+4b^2)}\right)
			- b^2\left(\frac{1}{7} + \frac{3+2b}{6(1+4b^2)}\right)\right\} \\
		&= 2ad_1 + 2d_2^2(4+15b^2) + 22(1+2b)(1-b)^2 + 9b^2\Big(1+2(1-2b)^2\Big) \\
		&> 0, \\
	280(1+4b^2)&\left(2L_1(a,b)+b\cdot\frac{\partial L_1}{\partial b}(a,b)\right) \\
		&\ge 280(1+4b^2)\left\{2\left(\frac{1-a+a^2+b^2}{14} - \frac{b^2}{4(1+4b^2)}\right)\right. \\
			&\qquad\qquad\qquad\qquad\qquad\qquad\qquad
			\left. +b\left(\frac{b}{7}-\frac{2+16b-9b^2}{40(1+4b^2)}\right)\right\} \\
		&= 64ad_1b(1-2b)^2 + 2\Big(2ad_1(30+b^2) + d_2^2(20+13b)\Big)(1-b) \\
			&\qquad\qquad + b^2(11 + 3d_2^2 + 320b^2 + 63d_2^2b) \\
		&> 0, \\
	28(1+4b^2)&\left(2L_1(a,b)-b\cdot\frac{\partial L_1}{\partial b}(a,b)\right) \\
		&\ge 28(1+4b^2)\left\{2\left(\frac{1-a+a^2+b^2}{14} - \frac{b^2}{4(1+4b^2)}\right)-b\cdot\frac{b}{7}\right\} \\
		&=1+2(1-b^2) + d_2^2(1+4b^2) \\
		&> 0, \\
	14(1+4b^2)&\left(4L_1(a,b)-b^2\cdot\frac{\partial^2 L_1}{\partial b^2}(a,\widetilde{b})\right) \\
		&\ge 14(1+4b^2)\left\{4\left(\frac{1-a+a^2+b^2}{14} - \frac{b^2}{4(1+4b^2)}\right)
			-b^2\left(\frac{1}{7}+\frac{1}{2(1+4b^2)}\right)\right\} \\
		&=1 + b^2 + d_2^2(1+4b^2) + 2(1-2b^2)^2 \\
		&> 0,
\end{align*}
hold and the desired inequalities are shown.
\end{proof}

\begin{lemma} \label{L2-bound}
It holds that
\begin{align*}
	\left|\frac{\partial L_2}{\partial a}(a,b)\right|&< \frac{2}{b}L_2(a,b),
		& \frac{\partial^2 L_2}{\partial a^2}(\widetilde{a},b)&<\frac{5}{b^2}L_2(a,b), \\
	\left|\frac{\partial L_2}{\partial b}(a,b)\right|&< \frac{2}{b}L_2(a,b),
		& \frac{\partial^2 L_2}{\partial b^2}(a,\widetilde{b})&<\frac{4}{b^2}L_2(a,b).
\end{align*}
\end{lemma}

\begin{proof}\quad
From the definition of $L_2(a,b)$ and Lemma~\ref{appendix-psi}, we have
\begin{align*}
	L_2(a,b)\;\; &= \frac{1+a^2+(1-a)^2+2b^2}{54}
			- \frac{b^4}{32(a^2+b^2)((1-a)^2+b^2)} \\
		&= \frac{1-a+a^2+b^2}{27} - \frac{b^4(1+2a)}{32(a^2+b^2)(1+4b^2)} - \frac{b^4(3-2a)}{32((1-a)^2+b^2)(1+4b^2)} \\
		&= \frac{1-a+a^2+b^2}{27} - \frac{\psi(a,b)+\psi(1-a,b)}{32} \\
		&\ge \frac{1-a+a^2+b^2}{27} - \frac{1}{32}\left(\frac{b^2(1+2a)}{1+4b^2}+\frac{b^2(1+2(1-a))}{1+4b^2}\right) \\
		&=\frac{1-a+a^2+b^2}{27} - \frac{b^2}{8(1+4b^2)}, \\
	\frac{\partial L_2}{\partial a}(a,b)\; &= -\frac{1-2a}{27} - \frac{\psi_a(a,b)-\psi_a(1-a,b)}{32} \\
		&\ge -\frac{1-2a}{27} - \frac{1}{32}\left(\frac{2b^2}{1+4b^2}+\frac{2b(1+2(1-a))}{3(1+4b^2)}\right) \\
		&= -\frac{1-2a}{27} - \frac{b(3-2a+3b)}{48(1+4b^2)}, \\
	\frac{\partial L_2}{\partial a}(a,b)\; &= -\frac{1-2a}{27} - \frac{\psi_a(a,b)-\psi_a(1-a,b)}{32} \\
		&\le -\frac{1-2a}{27} - \frac{1}{32}\left(-\frac{2b(1+2a)}{3(1+4b^2)}-\frac{2b^2}{1+4b^2}\right) \\
		&= -\frac{1-2a}{27} + \frac{b(1+2a+3b)}{48(1+4b^2)}, \\
	\frac{\partial^2 L_2}{\partial a^2}(\widetilde{a},b) &= \frac{2}{27} - \frac{\psi_{aa}(\widetilde{a},b)+\psi_{aa}(1-\widetilde{a},b)}{32} \\
		&\le \frac{2}{27} - \frac{1}{32}\left(-\frac{2(3+6\widetilde{a}+4b)}{3(1+4b^2)}-\frac{2(3+6(1-\widetilde{a})+4b)}{3(1+4b^2)}\right) \\
		&= \frac{2}{27} + \frac{3+2b}{12(1+4b^2)}, \\
	\frac{\partial L_2}{\partial b}(a,b)\; &= \frac{2b}{27} - \frac{\psi_b(a,b)+\psi_b(1-a,b)}{32} \\
		&\ge\frac{2b}{27} - \frac{1}{32}\left(\frac{(1+2a)(2+16b-9b^2)}{10(1+4b^2)}+\frac{(1+2(1-a))(2+16b-9b^2)}{10(1+4b^2)}\right) \\
		&= \frac{2b}{27}-\frac{2+16b-9b^2}{80(1+4b^2)}, \\
	\frac{\partial L_2}{\partial b}(a,b)\; &= \frac{2b}{27} - \frac{\psi_b(a,b)+\psi_b(1-a,b)}{32}
		\le\frac{2b}{27}, \\
	\frac{\partial^2 L_2}{\partial b^2}(a,\widetilde{b}) &= \frac{2}{27} - \frac{\psi_{bb}(a,\widetilde{b})+\psi_{bb}(1-a,\widetilde{b})}{32} \\
		&\le\frac{2}{27} - \frac{1}{32}\left(-\frac{21(1+2a)}{11(1+4\widetilde{b}^2)}-\frac{21(1+2(1-a))}{11(1+4\widetilde{b}^2)}\right) \\
		&= \frac{2}{27}+\frac{21}{88(1+4\widetilde{b}^2)} \le\frac{2}{27}+\frac{21}{88(49/50)^2(1+4b^2)} \\
		&\le\frac{2}{27}+\frac{1}{4(1+4b^2)}.
\end{align*}
Then,
\begin{align*}
	432(1+4b^2)&\left(2L_2(a,b)+b\cdot\frac{\partial L_2}{\partial a}(a,b)\right) \\
		&\ge 432(1+4b^2)\left\{2\left(\frac{1-a+a^2+b^2}{27} - \frac{b^2}{8(1+4b^2)}\right)\right. \\
			&\qquad\qquad\qquad\qquad\qquad\qquad\qquad
			\left. +b\left(-\frac{1-2a}{27} - \frac{b(3-2a+3b)}{48(1+4b^2)}\right)\right\} \\
		&=8 + 8d_2^2 + 16(1-b) + b^2(18a+128a^2+9b+28b^2) \\
			&\qquad\qquad + b(32a+25b)(1-2b)^2 \\
		&> 0, \\
	432(1+4b^2)&\left(2L_2(a,b)-b\cdot\frac{\partial L_2}{\partial a}(a,b)\right) \\
		&\ge 432(1+4b^2)\left\{2\left(\frac{1-a+a^2+b^2}{27} - \frac{b^2}{8(1+4b^2)}\right)\right. \\
			&\qquad\qquad\qquad\qquad\qquad\qquad\qquad
			\left. -b\left(-\frac{1-2a}{27} + \frac{b(1+2a+3b)}{48(1+4b^2)}\right)\right\} \\
		&=8 + 8d_2^2 + 16(1-b) + b^2(18d_1+128d_1^2+9b+28b^2) \\
			&\qquad\qquad + b(32d_1+25b)(1-2b)^2 \\
		&> 0, \\
	216(1+4b^2)&\left(5L_2(a,b)-b^2\cdot\frac{\partial^2 L_2}{\partial a^2}(\widetilde{a},b)\right) \\
		&\ge 216(1+4b^2)\left\{5\left(\frac{1-a+a^2+b^2}{27} - \frac{b^2}{8(1+4b^2)}\right)
			- b^2\left(\frac{2}{27} + \frac{3+2b}{12(1+4b^2)}\right)\right\} \\
		&= 9b(1-b) + 15(2+b)(1-2b)^2 + 10d_2^2(1+4b^2) + 96b(1+b)(1-b)^2 \\
		&> 0, \\
	2160(1+4b^2)&\left(2L_2(a,b)+b\cdot\frac{\partial L_2}{\partial b}(a,b)\right) \\
		&\ge 2160(1+4b^2)\left\{2\left(\frac{1-a+a^2+b^2}{27} - \frac{b^2}{8(1+4b^2)}\right)\right. \\
			&\qquad\qquad\qquad\qquad\qquad\qquad\qquad
			\left. +b\left(\frac{2b}{27}-\frac{2+16b-9b^2}{80(1+4b^2)}\right)\right\} \\
		&= 337b^3 + 40d_2^2(1+4b^2) + 30(2-b-7b^2)^2 + 2b(1-b)(33+352b+95b^2) \\
		&> 0, \\
	108(1+4b^2)&\left(2L_2(a,b)-b\cdot\frac{\partial L_2}{\partial b}(a,b)\right) \\
		&\ge 108(1+4b^2)\left\{2\left(\frac{1-a+a^2+b^2}{27} - \frac{b^2}{8(1+4b^2)}\right)-b\cdot\frac{2b}{27}\right\} \\
		&= 3 + 3(1-b^2) + 2d_2^2(1+4b^2) \\
		&> 0, \\
	108(1+4b^2)&\left(4L_2(a,b)-b^2\cdot\frac{\partial^2 L_2}{\partial b^2}(a,\widetilde{b})\right) \\
		&\ge 108(1+4b^2)\left\{4\left(\frac{1-a+a^2+b^2}{27} - \frac{b^2}{8(1+4b^2)}\right)
			-b^2\left(\frac{2}{27}+\frac{1}{4(1+4b^2)}\right)\right\} \\
		&= 4 + 7b^2 + 4d_2^2(1+4b^2) + 8(1-2b^2)^2 \\
		&> 0,
\end{align*}
hold and the desired inequalities are shown.
\end{proof}

\begin{lemma} \label{L3-bound}
It holds that
\begin{align*}
	\left|\frac{\partial L_3}{\partial a}(a,b)\right|&< \frac{2}{b}L_3(a,b),
		& \frac{\partial^2 L_3}{\partial a^2}(\widetilde{a},b)&<\frac{4}{b^2}L_3(a,b), \\
	\left|\frac{\partial L_3}{\partial b}(a,b)\right|&< \frac{3}{b}L_3(a,b),
		& \frac{\partial^2 L_3}{\partial b^2}(a,\widetilde{b})&<\frac{8}{b^2}L_3(a,b).
\end{align*}
\end{lemma}

\begin{proof}\quad
From the definition of $L_3(a,b)$ and Lemma~\ref{appendix-omega}, we have
\begin{align*}
	L_3(a,b)\;\;&=\frac{b^2+(1-a+a^2+b^2)^2}{83} + \frac{2a(1-a)-3b^2}{96}-\frac{\omega(a,b)}{48} \\
		&\ge\frac{b^2+(1-a+a^2+b^2)^2}{84} + \frac{2a(1-a)-3b^2}{96}-\frac{a(1-a)(7-4b^2)}{252} \\
		&= \frac{24(1-a+a^2+b^2)^2 - 39b^2 - 2a(1-a)(7-16b^2)}{2016}, \\
	\frac{\partial L_3}{\partial a}(a,b)\;&=-\frac{2(1-2a)(1-a+a^2+b^2)}{83}+\frac{1-2a}{48}-\frac{\omega_a(a,b)}{48} \\
		&\ge-\frac{25(1-2a)(1-a+a^2+b^2)}{1008}+\frac{1-2a}{48} \\
			&\qquad\qquad -\frac{(1-2a)(2-b^2+5a(1-a))}{84}  \\
		&=-\frac{(1-2a)(28+13b^2+35a(1-a))}{1008}, \\
	\frac{\partial L_3}{\partial a}(a,b)\;&=-\frac{2(1-2a)(1-a+a^2+b^2)}{83}+\frac{1-2a}{48}-\frac{\omega_a(a,b)}{48} \\
		&\le-\frac{2(1-2a)(1-a+a^2+b^2)}{144}+\frac{1-2a}{48}-\frac{1-2a}{96} \\
		&=-\frac{d_2(d_2^2+4b^2)}{288}
		\le0, \\
	\frac{\partial^2L_3}{\partial a^2}(\widetilde{a},b)
		&=\frac{3+3(1-2\widetilde{a})^2+4b^2}{83}-\frac{1}{24}-\frac{\omega_{aa}(\widetilde{a},b)}{48} \\
		&\le\frac{25(3+3(1-2\widetilde{a})^2+4b^2)}{2016}-\frac{1}{24}+\frac{11\widetilde{a}(1-\widetilde{a})+2b^2}{168b^2} \\
		&=\frac{(11-25b^2)(\widetilde{a}-a)(1-\widetilde{a}-a)}{168b^2} \\
			&\qquad\qquad +\frac{2a(1-a)(33-75b^2) + 5b^2(9+10b^2)}{1008b^2} \\
		&\le\frac{14(b/50)}{168b^2}+\frac{2a(1-a)(33-75b^2) + 5b^2(9+10b^2)}{1008b^2} \\
		&\le\frac{1}{504b^2}+\frac{2a(1-a)(33-75b^2) + 5b^2(9+10b^2)}{1008b^2} \\
		&=\frac{2 + 2a(1-a)(33-75b^2) + 5b^2(9+10b^2)}{1008b^2}, \\
	\frac{\partial L_3}{\partial b}(a,b)\;&=\frac{b(5+(1-2a)^2+4b^2)}{83}-\frac{b}{16}-\frac{\omega_b(a,b)}{48} \\
		&\ge\frac{23b(5+(1-2a)^2+4b^2)}{1920}-\frac{b}{16}+\frac{a(1-a)b}{96} \\
		&=\frac{b(9d_2^2+46b^2)}{960}
		\ge0, \\
	\frac{\partial L_3}{\partial b}(a,b)\;&=\frac{b(5+(1-2a)^2+4b^2)}{83}-\frac{b}{16}-\frac{\omega_b(a,b)}{48} \\
		&\le\frac{b(5+(1-2a)^2+4b^2)}{83}-\frac{b}{16}+\frac{1}{288b} \\
		&\le\frac{17b(5+(1-2a)^2+4b^2)}{1344}-\frac{b}{16}+\frac{1}{168b} \\
		&=\frac{4+9b^2-34a(1-a)b^2+34b^4}{672b}, \\
	\frac{\partial^2L_3}{\partial b^2}(a,\widetilde{b})
		&=\frac{5+(1-2a)^2+12\widetilde{b}^2}{83}-\frac{1}{16}-\frac{\omega_{bb}(a,\widetilde{b})}{48} \\
		&\le\frac{5+(1-2a)^2+12\widetilde{b}^2}{83}-\frac{1}{16}+\frac{24-64a(1-a)-17\widetilde{b}^2}{2112\widetilde{b}^2} \\
		&=\frac{(1-2a)^2+12\widetilde{b}^2}{83}+\frac{4a(1-a)+3(1-2a)^2}{264\widetilde{b}^2} -\frac{1807}{175296} \\
		&\le\frac{(1-2a)^2+12(51/50)^2b^2}{83}+\frac{4a(1-a)+3(1-2a)^2}{264(49/50)^2b^2}-\frac{1807}{175296} \\
		&\le\frac{2(1-2a)^2+19b^2}{126}+\frac{4a(1-a)+3(1-2a)^2}{252b^2}-\frac{1}{126} \\
		&=\frac{3+2b^2+38b^4 -8a(1-a)(1+2b^2)}{252b^2}.
\end{align*}
Then,
\begin{align*}
	1008&\left(2L_3(a,b)+b\cdot\frac{\partial L_3}{\partial a}(a,b)\right) \\
		&\ge 1008\left\{2\left(\frac{24(1-a+a^2+b^2)^2 - 39b^2 - 2a(1-a)(7-16b^2)}{2016}\right)\right. \\
			&\qquad\qquad\qquad\qquad
			\left. +b\left(-\frac{(1-2a)(28+13b^2+35a(1-a))}{1008}\right)\right\} \\
		&= 6d_2^4 + 6ad_1(4-b^2) + d_2b(ad_1+b^2) + 4(1-2b^2)^2 + b^2(4+b^2) \\
			&\qquad\qquad  + (d_2-b)^2(14+18ad_1+7b^2) \\
		&> 0, \\
	1008&\left(2L_3(a,b)-b\cdot\frac{\partial L_3}{\partial a}(a,b)\right) \\
		&\ge 1008\left\{2\left(\frac{24(1-a+a^2+b^2)^2 - 39b^2 - 2a(1-a)(7-16b^2)}{2016}\right) - b\cdot 0\right\} \\
		&= 34ad_1 + 5b^2 + 24a^2d_1^2 + 4d_2^2(6+b^2) + 24b^4 \\
		&> 0, \\
	1008&\left(4L_3(a,b)-b^2\cdot\frac{\partial^2 L_3}{\partial a^2}(\widetilde{a},b)\right) \\
		&\ge 1008\left\{4\left(\frac{24(1-a+a^2+b^2)^2 - 39b^2 - 2a(1-a)(7-16b^2)}{2016}\right)\right. \\
			&\qquad\qquad\qquad\qquad\qquad
			\left. -b^2\left(\frac{2 + 2a(1-a)(33-75b^2) + 5b^2(9+10b^2)}{1008b^2}\right)\right\} \\
		&=17d_2^4 + 2ad_1(12ad_1+31d_2^2+b^2) + (1-b^2)(29d_2^2+2b^2) \\
		&> 0, \\
	672&\left(3L_3(a,b)+b\cdot\frac{\partial L_3}{\partial b}(a,b)\right) \\
		&\ge 672\left\{3\left(\frac{24(1-a+a^2+b^2)^2 - 39b^2 - 2a(1-a)(7-16b^2)}{2016}\right) + b\cdot0\right\} \\
		&= 34ad_1 + 5b^2 + 24a^2d_1^2 + 4d_2^2(6+b^2) + 24b^4 \\
		&> 0, \\
	336&\left(3L_3(a,b)-b\cdot\frac{\partial L_3}{\partial b}(a,b)\right) \\
		&\ge 336\left\{3\left(\frac{24(1-a+a^2+b^2)^2 - 39b^2 - 2a(1-a)(7-16b^2)}{2016}\right)\right. \\
			&\qquad\qquad\qquad\qquad
			\left. -b\left(\frac{4+9b^2-34a(1-a)b^2+34b^4}{672b}\right)\right\} \\
		&=ad_1(1+17d_2^2) + 5d_2^4 + (1-b^2)(11ad_1+5d_2^2+5b^2) \\
		&> 0, \\
	252&\left(8L_3(a,b)-b^2\cdot\frac{\partial^2 L_3}{\partial b^2}(a,\widetilde{b})\right) \\
		&\ge 252\left\{8\left(\frac{24(1-a+a^2+b^2)^2 - 39b^2 - 2a(1-a)(7-16b^2)}{2016}\right)\right. \\
			&\qquad\qquad\qquad\qquad\qquad\qquad
			\left. -b^2\left(\frac{3+2b^2+38b^4 -8a(1-a)(1+2b^2)}{252b^2}\right)\right\} \\
		&= 14d_2^2 + 2ad_1(1+12ad_1) + 7(1-b^2)(1+2b^2) \\
		&> 0,
\end{align*}
hold and the desired inequalities are shown.
\end{proof}

\begin{lemma} \label{L4-bound}
It holds that
\begin{align*}
	\left|\frac{\partial L_4}{\partial a}(a,b)\right|&< \frac{3}{b}L_4(a,b),
		& \frac{\partial^2 L_4}{\partial a^2}(\widetilde{a},b)&<\frac{9}{b^2}L_4(a,b), \\
	\left|\frac{\partial L_4}{\partial b}(a,b)\right|&< \frac{3}{b}L_4(a,b),
		& \frac{\partial^2 L_4}{\partial b^2}(a,\widetilde{b})&<\frac{9}{b^2}L_4(a,b).
\end{align*}
\end{lemma}

\begin{proof}\quad
From the definition of $L_4(a,b)$ and Lemma~\ref{appendix-phi}, we have
\begin{align*}
	L_4(a,b)\;\;&= \frac{(a^2+b^2)((1-a)^2+b^2)}{4b^2}-\frac{1+a^2+(1-a)^2+2b^2}{30} \\
			&\qquad\qquad -\frac{b^2}{20}\left(1+\frac{1}{a^2+b^2}+\frac{1}{(1-a)^2+b^2}\right) \\
		&= \frac{a^2(1-a)^2}{4b^2} + \frac{11-26a(1-a)}{60} + \frac{2b^2}{15}
			-\frac{\varphi(a,b)+\varphi(1-a,b)}{20} \\
		&\ge \frac{a^2(1-a)^2}{4b^2} + \frac{11-26a(1-a)}{60} + \frac{2b^2}{15}-\frac{1+1}{20} \\
		&=\frac{a^2(1-a)^2}{4b^2} + \frac{5 - 26a(1-a)+8b^2}{60}, \\
	\frac{\partial L_4}{\partial a}(a,b)\; &= (1-2a)\left(\frac{a(1-a)}{2b^2}-\frac{13}{30}\right)
			-\frac{\varphi_a(a,b)-\varphi_a(1-a,b)}{20} \\
		&\ge (1-2a)\left(\frac{a(1-a)}{2b^2}-\frac{13}{30}\right)
			-\frac{1}{20}\cdot\frac{2}{3b} \\
		&=(1-2a)\left(\frac{a(1-a)}{2b^2}-\frac{13}{30}\right)-\frac{1}{30b}, \\
	\frac{\partial L_4}{\partial a}(a,b)\; &= (1-2a)\left(\frac{a(1-a)}{2b^2}-\frac{13}{30}\right)
			-\frac{\varphi_a(a,b)-\varphi_a(1-a,b)}{20} \\
		&\le (1-2a)\left(\frac{a(1-a)}{2b^2}-\frac{13}{30}\right)+\frac{1}{20}\cdot\frac{2}{3b} \\
		&=(1-2a)\left(\frac{a(1-a)}{2b^2}-\frac{13}{30}\right)+\frac{1}{30b}, \\
	\frac{\partial^2 L_4}{\partial a^2}(\widetilde{a},b) &= \frac{1-6\widetilde{a}(1-\widetilde{a})}{2b^2} + \frac{13}{15}
			-\frac{\varphi_{aa}(\widetilde{a},b)+\varphi_{aa}(1-\widetilde{a},b)}{20} \\
		&\le \frac{1-6\widetilde{a}(1-\widetilde{a})}{2b^2} + \frac{13}{15}
			+ \frac{1}{20}\left( \frac{2}{b^2} + \frac{2}{b^2} \right) \\
		&=\frac{7-30a(1-a)-30(\widetilde{a}-a)(1-\widetilde{a}-a)}{10b^2}+\frac{13}{15} \\
		&\le\frac{7-30a(1-a)+30(b/45)}{10b^2}+\frac{13}{15} \\
		&=\frac{21-90a(1-a)+2b}{30b^2}+\frac{13}{15}, \\
	\frac{\partial L_4}{\partial b}(a,b)\; &= - \frac{a^2(1-a)^2}{2b^3} + \frac{4b}{15} 
			-\frac{\varphi_b(a,b)+\varphi_b(1-a,b)}{20} \\
		&\ge - \frac{a^2(1-a)^2}{2b^3} + \frac{4b}{15} 
			- \frac{1}{20}\left(\frac{1}{2b}+\frac{1}{2b}\right) \\
		&= - \frac{a^2(1-a)^2}{2b^3} - \frac{1}{20b} + \frac{4b}{15}, \\
	\frac{\partial L_4}{\partial b}(a,b)\; &= - \frac{a^2(1-a)^2}{2b^3} + \frac{4b}{15} 
			-\frac{\varphi_b(a,b)+\varphi_b(1-a,b)}{20} \\
		&\le - \frac{a^2(1-a)^2}{2b^3} + \frac{4b}{15}, \\
	\frac{\partial^2 L_4}{\partial b^2}(a,\widetilde{b})&= \frac{3a^2(1-a)^2}{2\widetilde{b}^4} + \frac{4}{15}
			-\frac{\varphi_{bb}(a,\widetilde{b})+\varphi_{bb}(1-a,\widetilde{b})}{20} \\
		&\le \frac{3a^2(1-a)^2}{2\widetilde{b}^4} + \frac{4}{15}
			+ \frac{1}{20}\left(\frac{7}{9\widetilde{b}^2}+\frac{7}{9\widetilde{b}^2} \right) \\
		&= \frac{3a^2(1-a)^2}{2\widetilde{b}^4} + \frac{7}{90\widetilde{b}^2} + \frac{4}{15} \\
		&\le \frac{3a^2(1-a)^2}{2(49/50)^4b^4} + \frac{7}{90(49/50)^2b^2} + \frac{4}{15} \\
		&\le \frac{49a^2(1-a)^2}{30b^4} + \frac{1}{12b^2} + \frac{4}{15}.
\end{align*}
Then,
\begin{align*}
	60b^2&\left(3L_4(a,b)+b\cdot\frac{\partial L_4}{\partial a}(a,b)\right) \ge 60b^2\left\{3\left(\frac{a^2(1-a)^2}{4b^2} + \frac{5 - 26a(1-a)}{60} + \frac{2b^2}{15}\right)\right. \\
			&\qquad\qquad\qquad\qquad\qquad\qquad\qquad
			\left. +b\left((1-2a)\left(\frac{a(1-a)}{2b^2}-\frac{13}{30}\right)-\frac{1}{30b}\right)\right\} \\
		&= 11ad_1^2b + 5(3ad_1-b)^2 + ab(7d_1-5b)^2 \\
			&\qquad\qquad + b^2\Big(2(1-b) + a(4+2a+3b) + 6(d_1-2b)^2\Big) \\
		&> 0, \\
	60b^2&\left(3L_4(a,b)-b\cdot\frac{\partial L_4}{\partial a}(a,b)\right) \ge 60b^2\left\{3\left(\frac{a^2(1-a)^2}{4b^2} + \frac{5 - 26a(1-a)}{60} + \frac{2b^2}{15}\right)\right. \\
			&\qquad\qquad\qquad\qquad\qquad\qquad\qquad
			\left. -b\left((1-2a)\left(\frac{a(1-a)}{2b^2}-\frac{13}{30}\right)+\frac{1}{30b}\right)\right\} \\
		&= 11a^2d_1b + 5(3ad_1-b)^2 + d_1b(7a-5b)^2 \\
			&\qquad\qquad + b^2\Big(2(1-b) + d_1(4+2d_1+3b) + 6(a-2b)^2\Big) \\
		&> 0, \\
	60b^2&\left(9L_4(a,b)-b^2\cdot\frac{\partial^2 L_4}{\partial a^2}(\widetilde{a},b)\right) \\
		&\ge 60b^2\left\{9\left(\frac{a^2(1-a)^2}{4b^2} + \frac{5 - 26a(1-a)}{60} + \frac{2b^2}{15}\right)\right. \\
			&\qquad\qquad\qquad\qquad\qquad\qquad\qquad
			\left. -b^2\left(\frac{21-90a(1-a)+2b}{30b^2}+\frac{13}{15}\right)\right\} \\
 		&= 86a^2d_1^2 + (7ad_1-4b^2)^2 + b^2\Big(2 + 2ad_1 + (1-2b)^2\Big) \\
		&> 0, \\
	60b^2&\left(3L_4(a,b)+b\cdot\frac{\partial L_4}{\partial b}(a,b)\right) \ge 60b^2\left\{3\left(\frac{a^2(1-a)^2}{4b^2} + \frac{5 - 26a(1-a)}{60} + \frac{2b^2}{15}\right)\right. \\
			&\qquad\qquad\qquad\qquad\qquad\qquad\qquad
			\left. +b\left(- \frac{a^2(1-a)^2}{2b^3} - \frac{1}{20b} + \frac{4b}{15}\right)\right\} \\
		&=b^2(12d_2^2 + 25b^2) + 15(ad_1-b^2)^2 \\
		&> 0, \\
	60b^2&\left(3L_4(a,b)-b\cdot\frac{\partial L_4}{\partial b}(a,b)\right)
		\ge 60b^2\left\{3\left(\frac{a^2(1-a)^2}{4b^2} + \frac{5 - 26a(1-a)}{60} + \frac{2b^2}{15}\right)\right. \\
			&\qquad\qquad\qquad\qquad\qquad\qquad\qquad
			\left. -b\left(- \frac{a^2(1-a)^2}{2b^3} + \frac{4b}{15}\right)\right\} \\
		&=b^2(3+12d_2^2+5b^2) + 3(5ad_1-b^2)^2  \\
		&> 0, \\
	60b^2&\left(9L_4(a,b)-b^2\cdot\frac{\partial^2}{\partial b^2}L_4(a,\widetilde{b})\right) \\
		&\ge 60b^2\left\{9\left(\frac{a^2(1-a)^2}{4b^2} + \frac{5 - 26a(1-a)}{60} + \frac{2b^2}{15}\right)\right. \\
			&\qquad\qquad\qquad\qquad\qquad\qquad
			\left. -b^2\left(\frac{49a^2(1-a)^2}{30b^4} + \frac{1}{12b^2} + \frac{4}{15}\right)\right\} \\
		&=b^2(40d_2^2+19b^2) + 37(ad_1-b^2)^2 \\
		&> 0,
\end{align*}
hold and the desired inequalities are shown.
\end{proof}

\begin{lemma} \label{Lj-limit}
For $0\le a\le1,\; 0<b$, it holds that
$$
	\lim_{y\downarrow0} L_j(a,y) < L_j(a,b), \qquad j=1,2,3.
$$
\end{lemma}
\begin{proof}\quad
Although we cannot substitute $a=0$ and $b=0$ into $L_j(a,b),\; j=1,2$ simultaneously,
we can easily confirm that $\lim_{y\downarrow0}L_j(a,y), \; j=1,2,3$ coincide with the expression
obtained by formally substituting $b=0$ into $L_j(a,b)$.
Then, we have
\begin{align*}
	&\frac{112(a^2 + b^2)(d_1^2+b^2)}{b^2}\left(L_1(a,b)-\lim_{y\downarrow0} L_1(a,y)\right)
		= 8(ad_1-b^2)^2 + b^2 > 0, \\
	&\frac{864(a^2 + b^2)(d_1^2+b^2)}{b^2}\left(L_2(a,b)-\lim_{y\downarrow0} L_2(a,y)\right)
		= 32(ad_1-b^2)^2 + 5b^2 > 0, \\
	&\frac{7968(d_1+a^2+b^2)(d_1+a^2)}{b^2}\left(L_3(a,b)-\lim_{y\downarrow0} L_3(a,y)\right) \\
	&\qquad = 52ad_1 + 39d_2^2 + 3a^2d_1^2(125 + 16d_2^2) + 3b^2(d_1 + a^2)(21 + 32b^2 + 24d_2^2) \\
	&\qquad > 0.
\end{align*}
\end{proof}

\begin{lemma} \label{quodoratic-form}
Assume that $0\le a\le \dfrac{1}{2},\; 0<b\le \dfrac{1}{10}$ and let
$c_1, c_2, c_3, S, G_1, G_2$ be the notation that appears in Theorem~\ref{C4-asymptotic}. Then,
$$
	S\left(\frac{c_1c_2c_3}{16}-\frac{41(c_1+c_2+c_3)}{1080}-\frac{1}{5}\left(\frac{1}{c_1}+\frac{1}{c_2}+\frac{1}{c_3}\right)\right)G_2 - G_1
		\ge 0
$$
holds true.
\end{lemma}

\begin{proof}\quad
Let
$$
	D = \frac{c_1c_2c_3}{16}-\frac{41(c_1+c_2+c_3)}{1080}-\frac{1}{5}\left(\frac{1}{c_1}+\frac{1}{c_2}+\frac{1}{c_3}\right)
$$
and
$$
	H=48c_1^2c_2^2c_3^2(SDG_2 - G_1).
$$
Moreover, let
\begin{align*}
	E &= 48D-c_1-c_2-c_3, \\
	q_1&=2E + 2c_1 + c_2 + c_3, \\
	q_2&=2E + 2c_2 + c_3 + c_1, \\
	q_3&=2E + 2c_3 + c_1 + c_2, \\
	r_1&=\frac{c_1(12D(E+c_1) + 1)}{(2E + c_1 + 2c_2 + 2c_3)^2}, \\
	r_2&=\frac{c_2(12D(E+c_2) + 1)}{(2E + c_2 + 2c_3 + 2c_1)^2}, \\
	r_3&=\frac{c_3(12D(E+c_3) + 1)}{(2E + c_3 + 2c_1 + 2c_2)^2}, \\
	v_1&=c_2c_3(w_1+w_4), \\
	v_2&=c_3c_1(w_2+w_5), \\
	v_3&=c_1c_2(w_3+w_6), \\
	v_4&=w_1+w_2-w_3+w_4-w_5+w_6+2(w_8+w_9), \\
	v_5&=w_2+w_3-w_1+w_5-w_6+w_4+2(w_9+w_7), \\
	v_6&=w_3+w_1-w_2+w_6-w_4+w_5+2(w_7+w_8), \\
	v_7&=c_1c_2c_3(4q_1(w_4-w_8+w_9)-q_1(v_5-v_6)-(c_2-c_3)v_4), \\
	v_8&=c_1c_2c_3(4q_2(w_5-w_9+w_7)-q_2(v_6-v_4)-(c_3-c_1)v_5), \\
	v_9&=c_1c_2c_3(4q_3(w_6-w_7+w_8)-q_3(v_4-v_5)-(c_1-c_2)v_6).
\end{align*}
Then, from
\begin{align*}
	\frac{45b^3E}{4}&=\frac{108(ad_1-b^2)^2(1 + d_2^2 + 4b^2)}{c_1c_2} \\
	&\qquad + 270a^2d_1^2 + b^2\Big(47ad_1 + 88d_2^2 + (1-100b^2)\Big) + 81b^4 \\
	&>0,
\end{align*}
$D,E,q_1,q_2,q_3,r_1,r_2$, and $r_3$ are all positive and the following equality holds.
\begin{align*}
	H=& 16\left\{
		\frac{1}{c_2c_3q_1}\Big( c_1c_2c_3\Big( q_1(u_1-u_2-u_3) - 6D(c_1-c_2-c_3)v_4 \Big) - 96Dv_1 \Big)^2 \right. \\
	&\qquad\qquad + \frac{1}{c_3c_1q_2}\Big( c_1c_2c_3\Big( q_2(u_2-u_3-u_1) - 6D(c_2-c_3-c_1)v_5 \Big) - 96Dv_2 \Big)^2 \\
	&\qquad\qquad\qquad\left. + \frac{1}{c_1c_2q_3}\Big( c_1c_2c_3\Big( q_3(u_3-u_1-u_2) - 6D(c_3-c_1-c_2)v_6 \Big) - 96Dv_3 \Big)^2 \right\} \\
	& + \left\{
		\frac{1}{c_2c_3q_1}\Big( q_1(2c_1v_1 - c_2v_2 + c_3v_3) + (c_2-c_3)(96D-c_1)v_1 - v_7 \Big)^2 \right. \\
	&\qquad\qquad + \frac{1}{c_3c_1q_2}\Big( q_2(2c_2v_2 - c_3v_3 + c_1v_1) + (c_3-c_1)(96D-c_2)v_2 - v_8 \Big)^2 \\
	&\qquad\qquad\qquad\left. + \frac{1}{c_1c_2q_3}\Big( q_3(2c_3v_3 - c_1v_1 + c_2v_2) + (c_1-c_2)(96D-c_3)v_3 - v_9 \Big)^2 \right\} \\
	& + 4\left\{
		\frac{c_1}{q_1r_1}\Big( v_1 + r_1c_2c_3(96D-c_1)v_4 \Big)^2
		+ \frac{c_2}{q_2r_2}\Big( v_2 + r_2c_3c_1(96D-c_2)v_5 \Big)^2 \right. \\
	&\qquad\qquad\qquad\qquad\qquad\qquad\qquad\qquad\qquad
		\left. + \frac{c_3}{q_3r_3}\Big( v_3 + r_3c_1c_2(96D-c_3)v_6 \Big)^2 \right\} \\
	&+ 8\left\{
		\frac{q_1c_1 + 3D(c_2+c_3)(c_2-c_3)^2}{q_1(12D(E+c_1)+1)}v_1^2
			+ \frac{q_2c_2 + 3D(c_3+c_1)(c_3-c_1)^2}{q_2(12D(E+c_2)+1)}v_2^2 \right. \\
	&\qquad\qquad\qquad\left. + \frac{q_3c_3 + 3D(c_1+c_2)(c_1-c_2)^2}{q_3(12D(E+c_3)+1)}v_3^2 \right\} \\
	&+ \frac{2}{E}\left\{\frac{(c_2-c_3)^2 + 32}{E+c_1}\left(c_1 + \frac{E}{12D(E+c_1)+1}\right)v_1^2 \right. \\
	&\qquad+ \frac{(c_3-c_1)^2 + 32}{E+c_2}\left(c_2 + \frac{E}{12D(E+c_2)+1}\right)v_2^2 \\
	&\qquad\qquad \left. + \frac{(c_1-c_2)^2 + 32}{E+c_3}\left(c_3 + \frac{E}{12D(E+c_3)+1}\right)v_3^2\right\} \\
	&+ \frac{2}{E}\left\{\Big( 8E(3Dc_1-1)c_1 - (c_2-c_3)^2 - 32 \Big)v_1^2 \right. \\
	&\qquad\qquad + \Big( 8E(3Dc_2-1)c_2 - (c_3-c_1)^2 - 32 \Big)v_2^2 \\
	&\qquad\qquad\qquad + \Big( 8E(3Dc_3-1)c_3 - (c_1-c_2)^2 - 32 \Big)v_3^2 \\
	&\qquad\qquad \left. + 48DE\Big( (c_1c_2-4)v_1v_2 + (c_2c_3-4)v_2v_3 + (c_3c_1-4)v_3v_1 \Big)\right\}.
\end{align*}
Therefore,
\begin{align*}
	\frac{EH}{2}&\ge\Big( 8E(3Dc_1-1)c_1 - (c_2-c_3)^2 - 32 \Big)v_1^2 \\
	&\qquad + \Big( 8E(3Dc_2-1)c_2 - (c_3-c_1)^2 - 32 \Big)v_2^2 \\
	&\qquad\qquad + \Big( 8E(3Dc_3-1)c_3 - (c_1-c_2)^2 - 32 \Big)v_3^2 \\
	&\qquad\qquad\qquad+ 48DE\Big( (c_1c_2-4)v_1v_2 + (c_2c_3-4)v_2v_3 + (c_3c_1-4)v_3v_1 \Big)
\end{align*}
holds, and by Lemma~\ref{positive-definiteness2}, which will be discussed later,
we can confirm that this quadratic form regarding $v_1,v_2,v_3$ is positive definite, and $H\ge0$ is shown.
\end{proof}

Here, as a preliminary step, we present a lemma on the positive definiteness of quadratic forms.
\begin{lemma} \label{positive-definiteness}
Assume that $A_1,A_2,A_3,B_1,B_2,B_3\in\mathbb{R}$ satisfy
$$
	\left\{\begin{aligned}
		&B_1B_2B_3 \ge 0, \\
		&A_3 > 0, \\
		&A_1B_1^2+A_2B_2^2-2B_1B_2B_3 > 0, \\
		&A_1A_2A_3 - A_1B_1^2 - A_2B_2^2 - A_3B_3^2 + 2B_1B_2B_3 > 0.
	\end{aligned}\right.
$$
Then, for arbitrary $v_1,v_2,v_3\in\mathbb{R}$,
$$
	A_1v_1^2 + A_2v_2^2 + A_3v_3^2 + 2B_1v_2v_3 + 2B_2v_3v_1 + 2B_3v_1v_2\ge0
$$
holds.
\end{lemma}
\begin{proof}\quad
Under the assumptions, it holds that
\begin{align*}
	A_1B_1^2+A_2B_2^2&\ge A_1B_1^2+A_2B_2^2-2B_1B_2B_3 \\
	&> 0, \\
	A_1A_2-B_3^2&=\frac{1}{A_3}\Big((A_1B_1^2+A_2B_2^2-2B_1B_2B_3) \\
	&\qquad\qquad + (A_1A_2A_3 - A_1B_1^2 - A_2B_2^2 - A_3B_3^2 + 2B_1B_2B_3)\Big) \\
	&> \frac{1}{A_3}(A_1B_1^2+A_2B_2^2-2B_1B_2B_3) \\
	&> 0.
\end{align*}
When $B_2=0$, we have $A_1>0$ by
$$
	A_1B_1^2 = A_1B_1^2+A_2B_2^2 > 0
$$
and when $B_2\not=0$, we also have $A_1>0$ by
\begin{align*}
	A_1&=\frac{(A_1A_2 - B_3^2)B_2^2 + A_1^2B_1^2 + B_2^2B_3^2}{A_1B_1^2 + A_2B_2^2} \\
	&\ge\frac{ (A_1A_2 - B_3^2)B_2^2}{A_1B_1^2 + A_2B_2^2} \\
	&> 0.
\end{align*}
Therefore,
\begin{align*}
	&A_1>0,\qquad \begin{vmatrix} A_1 & B_3 \\ B_3 & A_2 \end{vmatrix}=A_1A_2-B_3^2>0, \\
	&\begin{vmatrix} A_1 & B_3 & B_2 \\ B_3 & A_2 & B_1 \\ B_2 & B_1 & A_3\\\end{vmatrix}
		=A_1A_2A_3 + 2B_1B_2B_3 - A_1B_1^2 - A_2B_2^2 - A_3B_3^2 > 0
\end{align*}
hold. Then, the general theory of quadratic forms and its positive definiteness (see \cite{HornJohnson})
lead us to the conclusion.
\end{proof}

\begin{lemma} \label{positive-definiteness2}
Assume that $0\le a\le \dfrac{1}{2},\; 0<b\le \dfrac{1}{10}$ and let
$c_1, c_2, c_3, D$, and $E$ be the notation that appears in Theorem~\ref{C4-asymptotic} and Lemma~\ref{quodoratic-form}.
Furthermore, let
\begin{align*}
	A_1&=8E(3Dc_1-1)c_1 - (c_2-c_3)^2 - 32, \\
	A_2&=8E(3Dc_2-1)c_2 - (c_3-c_1)^2 - 32, \\
	A_3&=8E(3Dc_3-1)c_3 - (c_1-c_2)^2 - 32, \\
	B_1&=24DE(c_2c_3-4), \\
	B_2&=24DE(c_3c_1-4), \\
	B_3&=24DE(c_1c_2-4).
\end{align*}
Then, for arbitrary $v_1,v_2,v_3\in\mathbb{R}$,
$$
	A_1v_1^2 + A_2v_2^2 + A_3v_3^2 + 2B_1v_2v_3 + 2B_2v_3v_1 + 2B_3v_1v_2\ge0
$$
holds true.
\end{lemma}
\begin{proof}\quad
It is sufficient to confirm that $A_1,A_2,A_3,B_1,B_2,B_3$ satisfy the assumptions of Lemma~\ref{positive-definiteness}.
Since $D>0, E>0$ is shown in Lemma~\ref{quodoratic-form}, $B_1B_2B_3\ge 0$ is shown by
$$
	b^6B_1B_2B_3=884736D^3E^3a^2d_1^2(ad_1-b^2)^2\ge 0.
$$
Other assumptions of Lemma~\ref{positive-definiteness} can be shown by the following transformations.
\begin{align*}
	\qquad&\!\!\!\!\!\!\!\!\!\!\!\!
		\frac{2025b^6}{c_1c_2}A_3
	=\frac{1492992d_2^2b^2}{c_1^3c_2^3} \\
	& + \frac{6912(ad_1-b^2)^4}{b^2c_1^2c_2^2}\Big(209ad_1 + 524b^2 + d_2^2(209 + 99b^2) + 396b^4\Big) \\
	& + \frac{4b^2}{c_1c_2}\Big(625080b^2 + 1047320b^4 + d_2^2(768 + 29840d_2^2 + 367663b^2)\Big) \\
	& + 722304ad_1b^2 + 583200(ad_1-b^2)^2 + (2603105 + 171072ad_1)b^4 \\
	& + b^2(1-100b^2)(41472 + 1711b^2) + 28b^6 \\
	&> 0,
\end{align*}
\begin{align*}
	\qquad&\!\!\!\!\!\!\!\!\!\!\!\!
		\frac{225b^{10}}{4096E^2D^2}(A_1B_1^2+A_2B_2^2-2B_1B_2B_3) \\
	&=\frac{432b^4(ad_1-b^2)^2}{c_1c_2}\left\{
		\frac{864d_2^2(1+2b^2)}{c_1c_2}
			+ 209 + 2334b^2 + 4696b^4 + 1728b^6 \right. \\
	&\qquad \left.\phantom{\frac{()^2}{c_1}}
		+ d_2^2(411 + 1542b^2 + 432b^4)\right\} \\
	& + b^2(1-100b^2)(300b^4 + 3225b^6 + 1867b^8 + 392040a^4d_1^4) \\
	& + 4b^4(1-10ad_1)^2(93271a^2d_1^2 + 3707d_2^2b^2) + 1166400a^6d_1^6 \\
	& + 1426680a^4d_1^4d_2^2b^2 + 54a^3d_1^3b^4(49418 + 230061ad_1) \\
	& + ad_1b^6\Big(12564d_2^2 + ad_1(14187+713783d_2^2)\Big) \\
	& + b^8( 93d_2^2 + 3075550a^2d_1^2 + 3649d_2^4) + b^{10}(89 + 48752d_2^2) + 76b^{12} \\
	&> 0,
\end{align*}
\begin{align*}
	\qquad&\!\!\!\!\!\!\!\!\!\!\!\!
		\frac{8303765625b^{24}c_1^4c_2^4}{16384}(A_1A_2A_3 + 2B_1B_2B_3 - A_1B_1^2 - A_2B_2^2 - A_3B_3^2) \\
	&=\frac{406239826673664d_2^6b^{18}}{c_1^2c_2^2} \\
	& + \frac{940369969152d_2^4b^{14}}{c_1c_2}\Big(6b^2(5ad_1-18b^2)^2 + a^2d_1^2b^2(2576b^2 + 29(1-100b^2)) \\
	&\qquad + 1253a^2d_1^2(ad_1-b^2)^2 + 6b^4(ad_1+54(ad_1-2b^2)^2)\Big) \\
	&+ b^4(1-100b^2)\Big(12121552738440000a^{14}d_1^{14} + 38379608724728421a^{12}d_1^{12}b^2 \\
	&\qquad + 18971342095206083a^{10}d_1^{10}b^4 + 2090743154936166a^8d_1^8b^6 \\
	&\qquad + 279697348741361a^6d_1^6b^8 + 3822652869396a^4d_1^4b^{10} + 756253238998a^2d_1^2b^{12} \\
	&\qquad + 4091161727505b^{14} + 12150003729052b^{16} + 50821685560135b^{18} \\
	&\qquad + 98411555030530b^{20} + 122686738307665b^{22} + 97941102999784b^{24} \\
	&\qquad + 50251546065961b^{26} + 92177583947338b^{28} + 7977409739887423b^{30}\Big) \\
	&+ a^3d_1^3b^8(1-8ad_1)^2\Big(a^7d_1^7(201474903892777042 + 90643664222052075d_2^4) \\
	&\qquad + 87413325433574170a^5d_1^5d_2^4b^2 + a^3d_1^3d_2^4b^4(26280162838987044ad_1 \\
	&\qquad\qquad + 18393062792105615d_2^2 + 263384776326212272a^2d_1^2) \\
	&\qquad  + 12ad_1d_2^2b^6(921487390368760ad_1 + 349514116301673d_2^2) \\
	&\qquad  + 1983138862306248d_2^{10}b^8\Big) \\
	&+ 47569496486400000a^{18}d_1^{18}(3ad_1 + d_2^2) \\
	&+ 6802444800000b^2a^{16}d_1^{16}(13631ad_1 + 38284d_2^2 + 108214a^2d_1^2) \\
	&+ 24603750b^4a^{14}d_1^{14}(900363843 + 7650287352d_2^2 + 10182035079d_2^4 + 3366302686d_2^6) \\
	&+ 9b^6a^{12}d_1^{12}\Big(60274379589129731d_2^6 + 647576875800691136a^4d_1^4 \\
	&\qquad + ad_1d_2^2(16761821 + 32817780980698951d_2^2 + 123184084009465804ad_1d_2^2)\Big) \\
	&+ 4b^8a^{11}d_1^{11}d_2^2\Big(89625675600532618d_2^4 \\
	&\qquad + ad_1(1232319731556914497 + 172d_2^2 + 282706757366186700d_2^4)\Big) \\
	&+ b^{10}a^9d_1^9\Big(252a^2d_1^2d_2^2 + 8d_2^4(680730978067758 + 176303141792417999ad_1) \\
	&\qquad + a^3d_1^3(28350046403323666988 + 48947181654129068587d_2^2 \\
	&\qquad\qquad + 10169754750431798896d_2^4)\Big) \\
	&+ 2b^{12}a^9d_1^9\Big(738d_2^4 + a^2d_1^2(6476622057673320575 + 16775086553269092583d_2^2) \\
	&\qquad + ad_1d_2^4(3205280764781424730 + 2163731183135420737d_2^2)\Big) \\
	&+ b^{14}a^6d_1^6\Big(7a^3d_1^3(267606134502121283 + 6725d_2^2) \\
	&\qquad + 3ad_1d_2^6(41287243358734152 + 1035751844957493491ad_1) \\
	&\qquad + 102133612824349884d_2^{10} \\
	&\qquad + 2a^4d_1^4d_2^2(4444298389065673601 + 11374900226453699636d_2^2)\Big) \\
	&+ 2b^{16}a^2d_1^2\Big(38a^7d_1^7(1217 + 119048568159419288ad_1d_2^2 + 150837118919724652d_2^4) \\
	&\qquad + 2834407261786822999a^6d_1^6d_2^6 + 1114301757844765956a^5d_1^5d_2^8 \\
	&\qquad + 3612664492939608a^2d_1^2d_2^{14} + d_2^{10}(485477278694229 \\
	&\qquad\qquad + 11576717008928092a^3d_1^3 + 297856692052747442a^4d_1^4)\Big) \\
	&+ b^{18}\Big(a^7d_1^7(80161 + 613280621514985073ad_1d_2^2 + 2168086120174861319d_2^4) \\
	&\qquad + 3a^5d_1^5d_2^6(201266834139916751 + 140164857867870021d_2^2) \\
	&\qquad + 134092514706817166a^4d_1^4d_2^{10} + a^2d_1^2d_2^{12}(958467215549044 \\
	&\qquad\qquad + 23606915947923884ad_1 + 113721778713257259a^2d_1^2) \\
	&\qquad + 850988675449304ad_1d_2^{18} + 55106777425135d_2^{20}\Big) \\
	&+ ad_1b^{20}\Big(5005973317536822ad_1 + 6878225893133844d_2^{10} \\
	&\qquad + 581926730554137620a^2d_1^2d_2^8 + 6500684793939788360a^7d_1^7 \\
	&\qquad + 19917465388756768960a^6d_1^6d_2^2 + 22114487883990056318a^5d_1^5d_2^4 \\
	&\qquad + 2a^3d_1^3d_2^6(697437936152400250 + 2785841237498890944ad_1 \\
	&\qquad\qquad + 1482271309993033463a^2d_1^2)\Big) \\
	&+ 2ad_1b^{22}\Big(23130096864387570d_2^8 + 2315937521698808738a^5d_1^5 \\
	&\qquad + 840904065784953540a^2d_1^2d_2^6 \\
	&\qquad + ad_1d_2^2(14037263072295615 + 2867348900259792053a^3d_1^3) \\
	&\qquad + 3a^3d_1^3d_2^4(332465786752201665 + 682141502084226293ad_1 \\
	&\qquad\qquad + 977253515068504924a^2d_1^2)\Big) \\
	&+ ad_1b^{24}\Big(1278193172773325750a^3d_1^3 + 577141788962864220ad_1d_2^4 \\
	&\qquad + a^2d_1^2d_2^2(794463032340756283 + 1790154667867305990ad_1) \\
	&\qquad + 136816263813704700d_2^8 \\
	&\qquad + 3a^2d_1^2d_2^6(295254266683834159 + 189245075752900744ad_1)\Big) \\
	&+ ad_1b^{26}(7342158644044464784a^4d_1^4 + 13195017945764608985a^3d_1^3d_2^2 \\
	&\qquad + 6273782071117534828a^2d_1^2d_2^4 + 1505967078540509160ad_1d_2^6 \\
	&\qquad + 222276557026854540d_2^8 + 11142172580460000000a^5d_1^5) \\
	&+ 2ad_1b^{28}\Big(371333521783157905a^2d_1^2 + 110859693910065990d_2^4 \\
	&\qquad + ad_1d_2^2(113043479305733967 + 681897555397501901ad_1 \\
	&\qquad\qquad + 350076133701617264a^2d_1^2)\Big) \\
	&+ ad_1b^{30}\Big(43083707928314887 + 94188165135443957d_2^4 \\
	&\qquad + 2ad_1d_2^2(38270182911979807 + 19065464961694444d_2^2)\Big) \\
	&+ 4ad_1b^{32}\Big(59079847399876688a^2d_1^2 \\
	&\qquad + 3d_2^2(3833975366004819 + 7362306726062380ad_1)\Big) \\
	&+ 4ad_1b^{34}(1097149099950981 + 4490808002806144ad_1) \\
	&+ 4b^{36}(199324506182122989 + 113055729475640d_2^2) + 84587928319744b^{38} \\
	&> 0.
\end{align*}

\end{proof}

\section{MATLAB code for verification} \label{Appendix2}
In this appendix, we present MATLAB codes used in this study.
To run these codes, Symbolic Math Toolbox, an optional add-on to MATLAB, is also required.
All programs listed here are available for download on GitHub~\cite{GitHub}.

{\tt C1\_verify.m}, {\tt C2\_verify.m}, and {\tt C3\_verify.m} are codes that verify Theorem~\ref{Verified_Result1}
and Theorem~\ref{Verified_Result2} for $C_1, C_2, C_3$, and {\tt C4\_verify.m} is a code that verifies
Theorem~\ref{Verified_Result1} for $C_4$.
INTALB \cite{Rump2,MooreKearfottCloud} is required to execute these codes.

Codes {\tt C1\_verify.m} to {\tt C4\_verify.m} include many subprograms;
the relationship among the inclusions is shown in Table~\ref{inclusion}.

\begin{table}[h!]
	\centering
{\footnotesize
\begin{tabular}{p{120pt}cccc}
\hline
 Subprograms & {\tt C1\_verify.m} & {\tt C2\_verify.m} & {\tt C3\_verify.m} & {\tt C4\_verify.m} \\
\hline
 {\tt create\_vertex\_list.m} & -- & -- & $\bigcirc$ & $\bigcirc$ \\
 {\tt create\_edge\_list.m} & $\bigcirc$ & $\bigcirc$ & $\bigcirc$ & $\bigcirc$ \\
 {\tt create\_face\_list.m} & $\bigcirc$ & $\bigcirc$ & -- & -- \\
 {\tt set\_triangle\_alpha.m} & $\bigcirc$ & $\bigcirc$ & -- & -- \\
 {\tt set\_triangle\_beta.m} & -- & -- & $\bigcirc$ & $\bigcirc$ \\
 {\tt F\_alpha\_0.m} & $\bigcirc$ & $\bigcirc$ & -- & -- \\
 {\tt F\_alpha\_1.m} & $\bigcirc$ & $\bigcirc$ & -- & -- \\
 {\tt F\_beta\_0.m} & -- & -- & $\bigcirc$ & -- \\
 {\tt F\_beta\_1.m} & -- & -- & -- & $\bigcirc$ \\
 {\tt F\_beta\_2.m} & -- & -- & $\bigcirc$ & $\bigcirc$ \\
 {\tt create\_coefficient\_matrix.m} & $\bigcirc$ & $\bigcirc$ & $\bigcirc$ & $\bigcirc$ \\
 {\tt L1ab.m} & $\bigcirc$ & -- & -- & -- \\
 {\tt L2ab.m} & -- & $\bigcirc$ & -- & $\bigcirc$ \\
 {\tt L3ab.m} & -- & -- & $\bigcirc$ & -- \\
 {\tt L4ab.m} & -- & -- & -- & $\bigcirc$ \\
\hline
\end{tabular}
}
\caption{Relationship among inclusions}
\label{inclusion}
\end{table}

Below are the codes  {\tt C1\_verify.m} to {\tt C4\_verify.m} for verification regarding $C_1$ to $C_4$.

\newpage

\begin{lstlisting}[caption={\tt C1\_verify.m}, label=C1verify]
function C1_verify
    format long g;

    %%  Parameters  %%
    n = 20;
    ha = 1 / sym(50); hb = 1 / sym(50);
    ca = sym(5); cb = sym(3);

    disp('Step 1: Create index');

    [edge_list, max_index] = create_edge_list(n, 1);
    [face_list, max_index] = create_face_list(n, max_index + 1);

    %%  "constraints" is an array of the index  %%
    %%  that constraints are imposed.           %%
    constraints = face_list(face_list > 0);

    %%  "triangle(i,k), k=1,2,3,4" denotes indexes    %%
    %%  of w1,w2,w3,u0 associate with i-th triangle.  %%
    triangle = set_triangle_alpha(n, edge_list, face_list);

    disp('Step 2: Create coefficient matrix');

    a = sym('a'); b = sym('b'); h = sym('h');
    w1 = sym('w1'); w2 = sym('w2'); w3 = sym('w3');
    u0 = sym('u0');
    variables = [w1 w2 w3 u0];

    s = b/2*h^2;
    c1 = 2*((1-a)^2+b^2)/b; c2 = 2*(a^2+b^2)/b; c3 = 2/b;

    z = F_alpha_0(s, c1, c2, c3, w1, w2, w3, u0);
    A = create_coefficient_matrix(z, variables);

    z = F_alpha_1(s, c1, c2, c3, w1, w2, w3, u0);
    B = create_coefficient_matrix(z, variables);

	lambda = L1ab(a, b) / sym(n^2) * sym(n^2-1) ...
	        / (1+ca*ha^2) / (1+cb*hb^2);
    Mab = lambda * B - A;

	lambda = L1ab(a, 0) / sym(n^2) * sym(n^2-1) / (1+ca*ha^2);
    Ma0 = lambda * B - A;

    disp('Step 3: Verify the inequality for 1/10 <= b <= 1');

    h = 1 / intval(n);

    yk = 1000;
    for k = 1:119
        b = 1000 / intval(yk);
        xk = ceil(yk/40);
        for l = 0:xk
            a = l / intval(2*xk);
            M = eval(Mab);
            verify();
        end
        yk = floor(yk*51/50);
    end

    disp('Step 4: Verify the inequality for 0 < b <= 1/10');

    b = 1 / intval(10);
    for l = 0:250
        a = l / intval(500);
        M = eval(Ma0);
  	    verify();
    end

    disp('Verification completed.');

    %%  Nested function for the verification  %%
    function verify()
        disp(['a=', num2str(mid(a)), ', b=', num2str(mid(b))]);
        CM = intval(zeros(max_index));
        for i = 1:n^2
            c = triangle(i, :);
            CM(c, c) = CM(c, c) + M;
        end

        %%  Impose constraints  %%
        s = constraints(1);
        t = constraints(2:end);
        len = length(t);
        p = CM(:, s); q = CM(s, :); r = CM(s, s);
        CM(:, t) = CM(:, t) - repmat(p, 1, len);
        CM(t, :) = CM(t, :) - repmat(q, len, 1);
        CM(t, t) = CM(t, t) + repmat(r, len, len);
        CM(:, s) = []; CM(s, :) = [];

        if isspd(CM) == 1
            disp('Verified');
        else
            error('Verification failed.');
        end
    end
end
\end{lstlisting}

\begin{lstlisting}[caption={\tt C2\_verify.m}, label=C2verify]
function C2_verify
    format long g;

    %%  Parameters  %%
    n = 20;
    ha = 1 / sym(50); hb = 1 / sym(50);
    ca = sym(5); cb = sym(3);

    disp('Step 1: Create index');

    [edge_list, max_index] = create_edge_list(n, 1);
    [face_list, max_index] = create_face_list(n, max_index + 1);

    %%  "constraints" is an array of the index  %%
    %%  that constraints are imposed.           %%
    constraints = zeros(3, n);

    t = zeros(3, 2*n);
    for i = 2:2*n
        t(1, i) = edge_list(i, 2*(n+1) - i);
    end
    constraints(1, :) = t(t > 0);

    t = edge_list(1, :);
    constraints(2, :) = t(t > 0);

    t = edge_list(:, 1);
    constraints(3, :) = t(t > 0);

    %%  "triangle(i,k), k=1,2,3,4" denotes indexes    %%
    %%  of w1,w2,w3,u0 associate with i-th triangle.  %%
    triangle = set_triangle_alpha(n, edge_list, face_list);

    disp('Step 2: Create coefficient matrix');

    a = sym('a'); b = sym('b'); h = sym('h');
    w1 = sym('w1'); w2 = sym('w2'); w3 = sym('w3');
    u0 = sym('u0');
    variables = [w1 w2 w3 u0];

    s = b/2*h^2;
    c1 = 2*((1-a)^2+b^2)/b; c2 = 2*(a^2+b^2)/b; c3 = 2/b;

    z = F_alpha_0(s, c1, c2, c3, w1, w2, w3, u0);
    A = create_coefficient_matrix(z, variables);

    z = F_alpha_1(s, c1, c2, c3, w1, w2, w3, u0);
    B = create_coefficient_matrix(z, variables);

	lambda = L2ab(a, b) / sym(n^2) * sym(n^2-1) ...
	        / (1+ca*ha^2) / (1+cb*hb^2);
    Mab = lambda * B - A;

	lambda = L2ab(a, 0) / sym(n^2) * sym(n^2-1) / (1+ca*ha^2);
    Ma0 = lambda * B - A;

    disp('Step 3: Verify the inequality for 1/10 <= b <= 1');

    h = 1 / intval(n);

    yk = 1000;
    for k = 1:119
        b = 1000 / intval(yk);
        xk = ceil(yk/40);
        for l = 0:xk
            a = l / intval(2*xk);
            M = eval(Mab);
            verify();
        end
        yk = floor(yk*51/50);
    end

    disp('Step 4: Verify the inequality for 0 < b <= 1/10');

    b = 1 / intval(10);
    for l = 0:250
        a = l / intval(500);
        M = eval(Ma0);
  	    verify();
    end

    disp('Verification completed.');

    %%  Nested function for verification  %%
    function z = verify()
        disp(['a=', num2str(mid(a)), ', b=', num2str(mid(b))]);
        CM = intval(zeros(max_index));
        for i = 1:n^2
            c = triangle(i, :);
            CM(c, c) = CM(c, c) + M;
        end

        %%  Impose constraints  %%
        s = zeros(1, 3);
        for j = 1:3
            s(j) = constraints(j, 1);
            t = constraints(j, 2:end);
            len = length(t);
            p = CM(:, s(j)); q = CM(s(j), :); r = CM(s(j), s(j));
            CM(:, t) = CM(:, t) - repmat(p, 1, len);
            CM(t, :) = CM(t, :) - repmat(q, len, 1);
            CM(t, t) = CM(t, t) + repmat(r, len, len);
        end
        CM(:, s) = [];
        CM(s, :) = [];
        
        if isspd(CM) == 1
            disp('Verified');
        else
            error('Verification failed.');
        end
    end
end
\end{lstlisting}

\begin{lstlisting}[caption={\tt C3\_verify.m}, label=C3verify]
function C3_verify
    format long g;

    %%  Parameters  %%
    n = 20;
    ha = 1 / sym(50); hb = 1 / sym(50);
    ca = sym(6); cb = sym(8);

    disp('Step 1: Create index');

    [vertex_list, max_index] = create_vertex_list(n, 1);
    [edge_list, max_index] = create_edge_list(n, max_index + 1);

    %%  Variable "constraints" is an array of the index  %%
    %%  that constraints are imposed.                    %%
    constraints ...
    = [vertex_list(1, 1) vertex_list(n+1, 1) vertex_list(1, n+1)];

    %%  "triangle(i,k), k=1,2,3,4,5,6" denotes indexes      %%
    %%  of u1,u2,u3,w1,w2,w3 associate with i-th triangle.  %%
    [triangle, orientation] ...
    = set_triangle_beta(n, vertex_list, edge_list);

    disp('Step 2: Create coefficient matrix');

    a = sym('a'); b = sym('b'); h = sym('h');
    u1 = sym('u1'); u2 = sym('u2'); u3 = sym('u3');
    w1 = sym('w1'); w2 = sym('w2'); w3 = sym('w3');
    variables = [u1 u2 u3 w1 w2 w3];
    e = sym(ones(3, 3));

    s = b/2*h^2;
    c1 = 2*((1-a)^2+b^2)/b; c2 = 2*(a^2+b^2)/b; c3 = 2/b;

    z = F_beta_0(s, c1, c2, c3, u1, u2, u3, w1, w2, w3);
    A1 = create_coefficient_matrix(z, variables);
    A2 = A1.*[e -e; -e e];

    z = F_beta_2(s, c1, c2, c3, u1, u2, u3, w1, w2, w3);
    B1 = create_coefficient_matrix(z, variables);
    B2 = B1.*[e -e; -e e];

	lambda = L3ab(a, b) / sym(n^4) * sym(n^4-1) ...
	        / (1+ca*ha^2) / (1+cb*hb^2);
	M1ab = lambda * B1 - A1;
    M2ab = lambda * B2 - A2;

	lambda = L3ab(a, 0) / sym(n^4) * sym(n^4-1) / (1+ca*ha^2);
    M1a0 = lambda * B1 - A1;
    M2a0 = lambda * B2 - A2;

    disp('Step 3: Verify the inequality for 1/10 <= b <= 1');

    h = 1 / intval(n);

    yk = 1000;
    for k = 1:119
        b = 1000 / intval(yk);
        xk = ceil(yk/40);
        for l = 0:xk
            a = l / intval(2*xk);
	        M1 = eval(M1ab);
    	    M2 = eval(M2ab);
            verify();
        end
        yk = floor(yk*51/50);
    end

    disp('Step 4: Verify the inequality for 0 < b <= 1/10');

    b = 1 / intval(10);
    for l = 0:250
        a = l / intval(500);
        M1 = eval(M1a0);
   	    M2 = eval(M2a0);
  	    verify();
    end

    disp('Verification completed.');

    %%  Nested function for the verification  %%
    function verify()
        disp(['a=', num2str(mid(a)), ', b=', num2str(mid(b))]);
        CM = intval(zeros(max_index));
        for i = 1:n^2
            c = triangle(i, :);
            if orientation(i) == 1
                CM(c, c) = CM(c, c) + M1;
            else
                CM(c, c) = CM(c, c) + M2;
            end
        end

        %%  Impose constraints  %%
        CM(constraints, :) = [];
        CM(:, constraints) = [];

        if isspd(CM) == 1
            disp('Verified');
        else
            error('Verification failed.');
        end
    end
end
\end{lstlisting}

\begin{lstlisting}[caption={\tt C4\_verify.m}, label=C4verify]
function C4_verify
    format long g;

    %%  Parameters  %%
    n = 20;
    ha = 1 / sym(50); hb = 1 / sym(50);
    ca = sym(9); cb = sym(9);

    disp('Step 1: Create index');

    [vertex_list, max_index] = create_vertex_list(n, 1);
    [edge_list, max_index] = create_edge_list(n, max_index + 1);

    %%  "constraints" is an array of the index  %%
    %%  that constraints are imposed.           %%
    constraints ...
    = [vertex_list(1, 1) vertex_list(n+1, 1) vertex_list(1, n+1)];

    %%  "triangle(i,k), k=1,2,3,4,5,6" denotes indexes      %%
    %%  of u1,u2,u3,w1,w2,w3 associate with i-th triangle.  %%
    [triangle, orientation] ...
    = set_triangle_beta(n, vertex_list, edge_list);

    disp('Step 2: Create coefficient matrix');

    a = sym('a'); b = sym('b'); h = sym('h');
    u1 = sym('u1'); u2 = sym('u2'); u3 = sym('u3');
    w1 = sym('w1'); w2 = sym('w2'); w3 = sym('w3');
    variables = [u1 u2 u3 w1 w2 w3];
    e = sym(ones(3, 3));

    s = b/2*h^2;
    c1 = 2*((1-a)^2+b^2)/b; c2 = 2*(a^2+b^2)/b; c3 = 2/b;

    z = F_beta_1(s, c1, c2, c3, u1, u2, u3, w1, w2, w3);
    A1 = create_coefficient_matrix(z, variables);
    A2 = A1.*[e -e; -e e];

    z = F_beta_2(s, c1, c2, c3, u1, u2, u3, w1, w2, w3);
    B1 = create_coefficient_matrix(z, variables);
    B2 = B1.*[e -e; -e e];

    lambda = L4ab(a, b) / (1+ca*ha^2) / (1+cb*hb^2) ...
            - L2ab(a, b) / sym(n^2);
    M1ab = lambda * B1 - A1;
    M2ab = lambda * B2 - A2;

    disp('Step 3: Verify the inequality for 1/10 <= b <= 1');

    h = 1 / intval(n);
    yk = 1000;
    for k = 1:119
        b = 1000 / intval(yk);
        xk = ceil(yk/40);
        for l = 0:xk
            a = l / intval(2*xk);
            M1 = eval(M1ab);
            M2 = eval(M2ab);
            verify();
        end
        yk = floor(yk*51/50);
    end

    disp('Verification completed.');

    %%  Nested function for the verification  %%
    function verify()
        disp(['a=', num2str(mid(a)), ', b=', num2str(mid(b))]);
        CM = intval(zeros(max_index));
        for i = 1:n^2
            c = triangle(i, :);
            if orientation(i) == 1
                CM(c, c) = CM(c, c) + M1;
            else
                CM(c, c) = CM(c, c) + M2;
            end
        end

        %%  Impose constraints  %%
        CM(constraints, :) = [];
        CM(:, constraints) = [];

        if isspd(CM) == 1
            disp('Verified');
        else
            error('Verification failed.');
        end
    end
end
\end{lstlisting}

\bigskip

The function {\tt create\_vertex\_list(n, start\_index)} gives the mapping from the oblique coordinates of the nodal points
to the edge number.
For {\tt n=2, start\_index=1}, the mapping is given as shown in Fig.~\ref{vertexlist}.

\begin{figure}[!h]
	\centering
{\footnotesize
\begin{tikzpicture}[scale=1.2]
  \coordinate (A) at (0.0,0.0);
  \coordinate (B) at (5.0,0.0);
  \coordinate (C) at (3.5,3.0);
  \draw [line width = 0.8pt](A) -- (B) -- (C) -- cycle;
  \coordinate (C2) at ($(A)!0.5!(B)$);
  \coordinate (A2) at ($(B)!0.5!(C)$);
  \coordinate (B2) at ($(C)!0.5!(A)$);
  \draw [line width = 0.8pt] (A2) -- (B2) -- (C2) -- cycle;
  \draw(A)  circle [radius=0.1];
  \draw(B)  circle [radius=0.1];
  \draw(C)  circle [radius=0.1];
  \draw(C2) circle [radius=0.1];
  \draw(A2) circle [radius=0.1];
  \draw(B2) circle [radius=0.1];
  \draw (A)  node[below=0.7mm]{$(1,1)\rightarrow 1$};
  \draw (C2) node[below=0.7mm]{$(2,1)\rightarrow 2$};
  \draw (B)  node[below=0.7mm]{$(3,1)\rightarrow 3$};
  \draw (B2) node[left=1.3mm]{$(1,2)\rightarrow 4$};
  \draw (A2) node[right=0.7mm]{$(2,2)\rightarrow 5$};
  \draw (C)  node[above=0.7mm]{$(1,3)\rightarrow 6$};
\end{tikzpicture}
}
	\caption{Mapping obtained using {\tt create\_vertex\_list}.}
	\label{vertexlist}
\end{figure}

\newpage

\begin{lstlisting}[caption={\tt create\_vertex\_list.m}, label=CreateVertexList]
%%  "vertex_list(p+1, q+1)" denotes the index of vertex  %%
%%  whose location is  p/n*[1 0] + q/n*[a b] .           %%
function [vertex_list, end_index] ...
        = create_vertex_list(n, start_index)
    vertex_list = zeros(n+1);
    i = start_index;
    for q = 0:n
        for p = 0:n-q
            vertex_list(p+1, q+1) = i;
            i = i + 1;
        end
    end
    end_index = i - 1;
end
\end{lstlisting}

\bigskip

The function {\tt create\_edge\_list(n, start\_index)} gives the mapping from the oblique coordinates of the edge's midpoint
to the edge number.
For {\tt n=2, start\_index=1}, the mapping is given as shown in Fig.~\ref{edgelist}.

\begin{figure}[h]
	\centering
{\footnotesize
\begin{tikzpicture}[scale=1.2]
  \coordinate (A) at (0.0,0.0);
  \coordinate (B) at (5.0,0.0);
  \coordinate (C) at (3.5,3.0);
  \draw [line width = 0.8pt](A) -- (B) -- (C) -- cycle;
  \coordinate (C2) at ($(A)!0.5!(B)$);
  \coordinate (A2) at ($(B)!0.5!(C)$);
  \coordinate (B2) at ($(C)!0.5!(A)$);
  \draw [line width = 0.8pt] (A2) -- (B2) -- (C2) -- cycle;
  \draw($(C2)!0.5!(B2)$)  circle [radius=0.1];
  \draw($(B2)!0.5!(A)$)  circle [radius=0.1];
  \draw($(A)!0.5!(C2)$)  circle [radius=0.1];
  \draw($(B)!0.5!(A2)$)  circle [radius=0.1];
  \draw($(A2)!0.5!(C2)$)  circle [radius=0.1];
  \draw($(C2)!0.5!(B)$)  circle [radius=0.1];
  \draw($(A2)!0.5!(C)$)  circle [radius=0.1];
  \draw($(C)!0.5!(B2)$)  circle [radius=0.1];
  \draw($(B2)!0.5!(A2)$)  circle [radius=0.1];
  \draw($(C2)!0.5!(B2)$)  node[below=0.7mm]{$(2,2)\rightarrow 1$};
  \draw($(B2)!0.5!(A)$)  node[left=1.3mm]{$(1,2)\rightarrow 2$};
  \draw($(A)!0.5!(C2)$)  node[below=0.7mm]{$(2,1)\rightarrow 3$};
  \draw($(B)!0.5!(A2)$)  node[right=0.7mm]{$(4,2)\rightarrow 4$};
  \draw($(A2)!0.5!(C2)$)  node[right=5mm, below=0.7mm]{$(3,2)\rightarrow 5$};
  \draw($(C2)!0.5!(B)$)  node[below=0.7mm]{$(4,1)\rightarrow 6$};
  \draw($(A2)!0.5!(C)$)  node[right=0.7mm]{$(2,4)\rightarrow 7$};
  \draw($(C)!0.5!(B2)$)  node[left=1.3mm]{$(1,4)\rightarrow 8$};
  \draw($(B2)!0.5!(A2)$)  node[below=0.7mm]{$(2,3)\rightarrow 9$};
\end{tikzpicture}
}
	\caption{Mapping obtained using {\tt create\_edge\_list}.}
	\label{edgelist}
\end{figure}

\begin{lstlisting}[caption={\tt create\_edge\_list.m}, label=CreateEdgeList]
%%  "edge_list(p+1, q+1)" denotes the index of edge        %%
%%  whose mid point is  p/(2*n)*[1 0] + q/(2*n)*[a b] .    %%
function [edge_list, end_index] = create_edge_list(n, start_index)
    edge_list = zeros(2*n);
    i = start_index;
    for q = 0:n-1
        for p = 0:n-1-q
            cp = [p p+1 p];
            cq = [q q q+1];
            edge_list(cp(2)+cp(3)+1, cq(2)+cq(3)+1) = i;
            i = i + 1;
            edge_list(cp(3)+cp(1)+1, cq(3)+cq(1)+1) = i;
            i = i + 1;
            edge_list(cp(1)+cp(2)+1, cq(1)+cq(2)+1) = i;
            i = i + 1;
        end
    end
    end_index = i - 1;
end
\end{lstlisting}

\bigskip

The function {\tt create\_face\_list(n, start\_index)} gives the mapping from the oblique coordinates of the face's center of gravity
to the face number.
For {\tt n=2, start\_index=1}, the mapping is given as shown in Fig.~\ref{facelist}.

\begin{figure}[h]
	\centering
{\footnotesize
\begin{tikzpicture}[scale=1.2]
  \coordinate (A) at (0.0,0.0);
  \coordinate (B) at (5.0,0.0);
  \coordinate (C) at (3.5,3.0);
  \draw [line width = 0.8pt](A) -- (B) -- (C) -- cycle;
  \coordinate (C2) at ($(A)!0.5!(B)$);
  \coordinate (A2) at ($(B)!0.5!(C)$);
  \coordinate (B2) at ($(C)!0.5!(A)$);
  \draw [line width = 0.8pt] (A2) -- (B2) -- (C2) -- cycle;
  \coordinate (E1) at ($0.3333*($($(A)+(C2)$)+(B2)$)$);
  \coordinate (E2) at ($0.3333*($($(A2)+(B2)$)+(C2)$)$);
  \coordinate (E3) at ($0.3333*($($(C2)+(B)$)+(A2)$)$);
  \coordinate (E4) at ($0.3333*($($(B2)+(A2)$)+(C)$)$);
  \draw(E1)  circle [radius=0.1];
  \draw(E2)  circle [radius=0.1];
  \draw(E3)  circle [radius=0.1];
  \draw(E4)  circle [radius=0.1];
  \draw(E1)  node[below=0.7mm]{$(2,2)\rightarrow 1$};
  \draw(E2)  node[right=3mm,below=0.7mm]{$(3,3)\rightarrow 2$};
  \draw(E3)  node[below=0.7mm]{$(5,2)\rightarrow 3$};
  \draw(E4)  node[below=0.7mm]{$(2,5)\rightarrow 4$};
\end{tikzpicture}
}
	\caption{Mapping obtained using {\tt create\_face\_list}.}
	\label{facelist}
\end{figure}

\begin{lstlisting}[caption={\tt create\_face\_list.m}, label=CreateFaceList]
%%  "face_index(p+1, q+1)" denotes the index of face             %%
%%  whose center of gravity is  p/(3*n)*[1 0] + q/(3*n)*[a b] .  %%
function [face_list, end_index] = create_face_list(n, start_index)
    face_list= zeros(3*n);
    i = start_index;
    for q = 0:n-1
        for p = 0:n-1-q
            cp = [p p+1 p];
            cq = [q q q+1];
            face_list(sum(cp)+1, sum(cq)+1) = i;
            i = i + 1;
            if p+q < n-1
	            cp = [p+1 p p+1];
    	        cq = [q+1 q+1 q];
                face_list(sum(cp)+1, sum(cq)+1) = i;
                i = i + 1;
            end
        end
    end
    end_index = i - 1;
end
\end{lstlisting}

\bigskip

The function {\tt set\_triangle\_alpha} returns the correspondence from each triangle element to the edge
and face numbers that make up the element.

\begin{lstlisting}[caption={\tt set\_triangle\_alpha.m}, label=SetTriangleAlpha]
function triangle = set_triangle_alpha(n, edge_list, face_list)
    triangle = zeros(n^2, 4);
    i = 0;
    for q = 0:n-1
        for p = 0:n-1-q
            cp = [p p+1 p];
            cq = [q q q+1];
            edge1 = edge_list(cp(2)+cp(3)+1, cq(2)+cq(3)+1);
            edge2 = edge_list(cp(3)+cp(1)+1, cq(3)+cq(1)+1);
            edge3 = edge_list(cp(1)+cp(2)+1, cq(1)+cq(2)+1);
            face  = face_list(sum(cp)+1, sum(cq)+1);
            i = i + 1;
            triangle(i, :) = [edge1, edge2, edge3, face];
            if p+q < n-1
                cp = [p+1 p p+1];
                cq = [q+1 q+1 q];
                edge1 = edge_list(cp(2)+cp(3)+1, cq(2)+cq(3)+1);
                edge2 = edge_list(cp(3)+cp(1)+1, cq(3)+cq(1)+1);
                edge3 = edge_list(cp(1)+cp(2)+1, cq(1)+cq(2)+1);
                face  = face_list(sum(cp)+1, sum(cq)+1);
                i = i + 1;
                triangle(i, :) = [edge1, edge2, edge3, face];
            end
        end
    end
end
\end{lstlisting}

\bigskip

The function {\tt set\_triangle\_beta} returns, for each triangle element, the correspondence between
the node and edge numbers that make up the element, as well as its orientation.

\begin{lstlisting}[caption={\tt set\_triangle\_beta.m}, label=SetTriangleBeta]
function [triangle, orientation] ...
        = set_triangle_beta(n, vertex_list, edge_list)
    triangle = zeros(n^2, 6);
    orientation = zeros(n^2, 1);
    i = 0;
    for q = 0:n-1
        for p = 0:n-1-q
            cp = [p p+1 p];
            cq = [q q q+1];
            vertex1  = vertex_list(cp(1)+1, cq(1)+1);
            vertex2  = vertex_list(cp(2)+1, cq(2)+1);
            vertex3  = vertex_list(cp(3)+1, cq(3)+1);
            edge1 = edge_list(cp(2)+cp(3)+1, cq(2)+cq(3)+1);
            edge2 = edge_list(cp(3)+cp(1)+1, cq(3)+cq(1)+1);
            edge3 = edge_list(cp(1)+cp(2)+1, cq(1)+cq(2)+1);
            i = i + 1;
            triangle(i, :) ...
            = [vertex1, vertex2, vertex3, edge1, edge2, edge3];
            orientation(i) = 1;
            if p+q < n-1
	            cp = [p+1 p p+1];
    	        cq = [q+1 q+1 q];
                vertex1  = vertex_list(cp(1)+1, cq(1)+1);
                vertex2  = vertex_list(cp(2)+1, cq(2)+1);
                vertex3  = vertex_list(cp(3)+1, cq(3)+1);
                edge1 = edge_list(cp(2)+cp(3)+1, cq(2)+cq(3)+1);
                edge2 = edge_list(cp(3)+cp(1)+1, cq(3)+cq(1)+1);
                edge3 = edge_list(cp(1)+cp(2)+1, cq(1)+cq(2)+1);
                i = i + 1;
                triangle(i, :) ...
                = [vertex1, vertex2, vertex3, edge1, edge2, edge3];
                orientation(i) = 2;
            end
        end
    end
end
\end{lstlisting}

\bigskip

The functions {\tt F\_alpha\_0} and {\tt F\_alpha\_1} compute
$F^{(\alpha)}_0(s,w_1,w_2,w_3,u_0)$ and $F^{(\alpha)}_1(s,w_1,w_2,w_3,u_0)$, respectively,
that appear in Lemma~\ref{F1}.

\begin{lstlisting}[caption={\tt F\_alpha\_0.m}, label=Fa0]
function z = F_alpha_0(s, c1, c2, c3, w1, w2, w3, u0)
    v0 =  (3*c2 + 3*c3 - c1) * w1 ...
        + (3*c3 + 3*c1 - c2) * w2 ...
        + (3*c1 + 3*c2 - c3) * w3;
    z = s / 15 * ( ...
        8 * (c1^2 + c2^2 + c3^2) * (w1 + w2 + w3 - 3*u0)^2 ...
                / (c1 + c2 + c3)^2 ...
        - 2 * (w1 + w2 + w3 - 3*u0) * v0 / (c1 + c2 + c3) ...
        + 5 * (w1^2 + w2^2 + w3^2) ...
    );
end
\end{lstlisting}

\begin{lstlisting}[caption={\tt F\_alpha\_1.m}, label=Fa1]
function z = F_alpha_1(s, c1, c2, c3, w1, w2, w3, u0)
    z = ( ...
        32 * (w1 + w2 + w3 - 3*u0)^2 / (c1 + c2 + c3) ...
        + (c1 + c2 - c3) * (w1 - w2)^2 ...
        + (c2 + c3 - c1) * (w2 - w3)^2 ...
        + (c3 + c1 - c2) * (w3 - w1)^2 ...
    ) / 2;
end
\end{lstlisting}

\bigskip

The functions {\tt F\_beta\_0}, {\tt F\_beta\_1}, and {\tt F\_beta\_2} compute
$F^{(\beta)}_0(s,u_1,u_2,u_3,w_1,w_2,w_3)$, $F^{(\beta)}_1(s,u_1,u_2,u_3,w_1,w_2,w_3)$,
and $F^{(\beta)}_2(s,u_1,u_2,u_3,w_1,w_2,w_3)$, respectively,
that appear in Lemma~\ref{F2}.

\begin{lstlisting}[caption={\tt F\_beta\_0.m}, label=Fb0]
function z = F_beta_0(s, c1, c2, c3, u1, u2, u3, w1, w2, w3)
    v1 = (13*u1 + u2 + u3) / 4 - (4*w1 - (c2 - c3)*(u2 - u3)) / c1;
    v2 = (13*u2 + u3 + u1) / 4 - (4*w2 - (c3 - c1)*(u3 - u1)) / c2;
    v3 = (13*u3 + u1 + u2) / 4 - (4*w3 - (c1 - c2)*(u1 - u2)) / c3;
    z = s / 720 * ( ...
        21 * (u1^2 + u2^2 + u3^2) - 6 * (u1*u2 + u2*u3 + u3*u1) ...
        + 6 * (v1^2 + v2^2 + v3^2) + 10 * (v1*v2 + v2*v3 + v3*v1) ...
    );
end
\end{lstlisting}

\begin{lstlisting}[caption={\tt F\_beta\_1.m}, label=Fb1]
function z = F_beta_1(s, c1, c2, c3, u1, u2, u3, w1, w2, w3)
    z = ( ...
          ((u2 - u3)^2 + w1^2) / c1 ...
        + ((u3 - u1)^2 + w2^2) / c2 ...
        + ((u1 - u2)^2 + w3^2) / c3 ...
    ) / 3;
end
\end{lstlisting}

\begin{lstlisting}[caption={\tt F\_beta\_2.m}, label=Fb2]
function z = F_beta_2(s, c1, c2, c3, u1, u2, u3, w1, w2, w3)
    v4 = 2*u1 - u2 - u3 + (4*w1 - (c2 - c3)*(u2 - u3)) / c1;
    v5 = 2*u2 - u3 - u1 + (4*w2 - (c3 - c1)*(u3 - u1)) / c2;
    v6 = 2*u3 - u1 - u2 + (4*w3 - (c1 - c2)*(u1 - u2)) / c3;
    z = ( ...
        (c1*v4 + c2*v5 + c3*v6)^2 - 8*(v4*v5 + v5*v6 + v6*v4) ...
    ) / s / 16;
end
\end{lstlisting}

\bigskip

{\tt create\_coefficient\_matrix} is a function that computes the representation matrix
from the corresponding polynomial of quadratic form.

\begin{lstlisting}[caption={\tt create\_coefficient\_matrix.m}, label=CreateCoefficientMatrix]
function A = create_coefficient_matrix(quadratic_form, variables)
    m = length(variables);
    A = sym(zeros(m));
    for p = 1:m
        df = diff(quadratic_form, variables(p));
        for q = 1:p
            A(p, q) = simplifyFraction(diff(df, variables(q)) / 2);
            if q<p
                A(q, p) = A(p, q);
            end
        end
    end
end
\end{lstlisting}

\bigskip

The functions {\tt L1ab(a,b)}, {\tt L2ab(a,b)}, {\tt L3ab(a,b)}, and {\tt L4ab(a,b)} compute
$L_j(a,b),\;j=1,2,3,4$ defined by \eqref{Ljab}.

\begin{lstlisting}[caption={\tt L1ab.m}, label=L1ab]
function z = L1ab(a, b)
    AA = (1-a)^2 + b^2;
    BB = a^2 + b^2;
    CC = 1;
    SS = b^2 / 4;
    z = (AA + BB + CC) / 28 - SS^2 / (AA * BB * CC);
end
\end{lstlisting}

\begin{lstlisting}[caption={\tt L2ab.m}, label=L2ab]
function z = L2ab(a, b)
    AA = (1-a)^2 + b^2;
    BB = a^2 + b^2;
    CC = 1;
    SS = b^2 / 4;
    z = (AA + BB + CC) / 54 - SS^2 / 2 / (AA * BB * CC);
end
\end{lstlisting}

\begin{lstlisting}[caption={\tt L3ab.m}, label=L3ab]
function z = L3ab(a, b)
    AA = (1-a)^2 + b^2;
    BB = a^2 + b^2;
    CC = 1;
    SS = b^2 / 4;
    z = (AA*BB + BB*CC + CC*AA) / 83 ...
        - (AA*BB*CC / (AA + BB + CC) + SS) / 24;
end
\end{lstlisting}

\begin{lstlisting}[caption={\tt L4ab.m}, label=L4ab]
function z = L4ab(a, b)
    AA = (1-a)^2 + b^2;
    BB = a^2 + b^2;
    CC = 1;
    SS = b^2 / 4;
    z = AA*BB*CC / 16 / SS - (AA + BB + CC) / 30 ...
        - SS / 5 * (1/AA + 1/BB + 1/CC);
end
\end{lstlisting}

\bigskip

Hereafter, functions of the form {\tt Lemma*\_*} in {\tt Lemma*\_*.m} are codes to check the correctness of the corresponding
lemmas *.*.
In these codes, operations such as {\tt 2/3} are written as {\tt 2/sym(3)} to avoid being computed as floating-point numbers.
In addition, large integers that cannot be expressed in the range of double-precision floating-point numbers
are not accurately converted to symbolic integers if written as is, so to be safe,
integers with 10 or more digits are defined using strings such as {\tt sym('1234567890')}.

\begin{lstlisting}[caption={\tt Lemma5\_1.m}, label=Lemma5-1]
function Lemma5_1
    pair = [];

    a = sym('a'); b = sym('b'); h = sym('h');
    w1 = sym('w1'); w2 = sym('w2'); w3 = sym('w3');
    u0 = sym('u0');

    x = sym('x'); y = sym('y');
    t = sym('t'); t1 = sym('t1'); t2 = sym('t2');

    Pu = 1/b/h*( ...
            (b*(2*x - h) + 2*(1 - a)*y)*w1 ...
            - (b*(2*x - h) - 2*a*y)*w2 ...
            - (2*y - b*h)*w3 ...
        ) + 2*(3*(x^2 + y^2) - 2*(1 + a)*h*x - 2*b*h*y + a*h^2) ...
            /(1 - a + a^2 + b^2)/h^2*(w1 + w2 + w3 - 3*u0);
    Pux = diff(Pu, x); Puy = diff(Pu, y);

    W1 = int(subs(Pu, {x, y}, {h*(1 - (1 - a)*t), h*b*t}), t, 0, 1);
    W2 = int(subs(Pu, {x, y}, {h*a*t, h*b*t}), t, 0, 1);
    W3 = int(subs(Pu, {x, y}, {h*t, 0}), t, 0, 1);
    U0 = 2*int(int(subs(Pu, {x, y}, {h*(t1 + a*t2), h*b*t2}), ...
            t1, 0, 1 - t2), t2, 0, 1);
    pair = [pair [W1; w1] [W2; w2] [W3; w3] [U0; u0]];

    s = b*h^2/2;
    c1 = 2*((1 - a)^2 + b^2)/b; c2 = 2*(a^2 + b^2)/b; c3 = 2/b;

    v = subs(Pu^2, {x, y}, {h*(t1 + a*t2), h*b*t2});
    L2 = 2*s*int(int(v, t1, 0, 1 - t2), t2, 0, 1);
    l2 = F_alpha_0(s, c1, c2, c3, w1, w2, w3, u0);
    pair = [pair [L2; l2]];

    v = subs(Pux^2 + Puy^2, {x, y}, {h*(t1 + a*t2), h*b*t2});
    H1 = 2*s*int(int(v, t1, 0, 1 - t2), t2, 0, 1);
    h1 = F_alpha_1(s, c1, c2, c3, w1, w2, w3, u0);
    pair = [pair [H1; h1]];

    for p = pair
        if simplifyFraction(p(1) - p(2)) ~= 0
            error('There is something wrong.');
        end
    end
    disp('It is all right.');
end
\end{lstlisting}

\begin{lstlisting}[caption={\tt Lemma5\_2.m}, label=Lemma5-2]
function Lemma5_2
    pair = [];

    a = sym('a'); b = sym('b'); h = sym('h');
    u1 = sym('u1'); u2 = sym('u2'); u3 = sym('u3');
    w1 = sym('w1'); w2 = sym('w2'); w3 = sym('w3');

    x = sym('x'); y = sym('y');
    t = sym('t'); t1 = sym('t1'); t2 = sym('t2');

    Pu = 1/b/h*( ...
            - (b*(x - h) + (1 - a)*y)*u1 ...
            + (b*x - a*y)*u2 + y*u3 ...
        ) + (b*(x - h) + (1 - a)*y)*(b*x + (1 - a)*y) ...
            /((1 - a)^2 + b^2)/b^2/h^2 ...
            *(b*w1 + ((1 - a)^2 + b^2)*u1 ...
                + (a*(1 - a) - b^2)*u2 - (1 - a)*u3) ...
        + (b*(x - h) - a*y)*(b*x - a*y) ...
            /(a^2 + b^2)/b^2/h^2 ...
            * (b*w2 + (a*(1-a) - b^2)*u1 ...
                + (a^2 + b^2)*u2 - a*u3) ...
        + (y - b*h)*y/b^2/h^2 ...
            * (b*w3 - (1-a)*u1 - a*u2 + u3);
    Pux = diff(Pu, x); Puy = diff(Pu, y);
    Puxx = diff(Pux, x); Puxy = diff(Pux, y); Puyy = diff(Puy, y);

    U1 = subs(Pu, {x, y}, {0, 0});
    U2 = subs(Pu, {x, y}, {h, 0});
    U3 = subs(Pu, {x, y}, {h*a, h*b});
    pair = [pair [U1; u1] [U2; u2] [U3; u3]];

    nx = b*h; ny = (1-a)*h;
    W1 = int(subs(Pux*nx + Puy*ny, ...
            {x, y}, {h*(1-(1-a)*t), h*b*t}), t, 0, 1);
    nx = -b*h; ny = a*h;
    W2 = int(subs(Pux*nx + Puy*ny, {x, y}, {h*a*t, h*b*t}), t, 0, 1);
    nx = 0; ny = -h;
    W3 = int(subs(Pux*nx + Puy*ny, {x, y}, {h*t, 0}), t, 0, 1);
    pair = [pair [W1; w1] [W2; w2] [W3; w3]];

    s = b*h^2/2;
    c1 = 2*((1 - a)^2 + b^2)/b; c2 = 2*(a^2 + b^2)/b; c3 = 2/b;

    v = subs(Pu^2, {x, y}, {h*(t1 + a*t2), h*b*t2});
    L2 = 2*s*int(int(v, t1, 0, 1 - t2), t2, 0, 1);
    l2 = F_beta_0(s, c1, c2, c3, u1, u2, u3, w1, w2, w3);
    pair = [pair [L2; l2]];

    v = subs(Pux^2 + Puy^2, {x, y}, {h*(t1 + a*t2), h*b*t2});
    H1 = 2*s*int(int(v, t1, 0, 1 - t2), t2, 0, 1);
    h1 = F_beta_1(s, c1, c2, c3, u1, u2, u3, w1, w2, w3);
    pair = [pair [H1; h1]];

    v = subs(Puxx^2 + 2*Puxy^2 + Puyy^2, ...
            {x, y}, {h*(t1 + a*t2), h*b*t2});
    H2 = 2*s*int(int(v, t1, 0, 1 - t2), t2, 0, 1);
    h2 = F_beta_2(s, c1, c2, c3, u1, u2, u3, w1, w2, w3);
    pair = [pair [H2; h2]];

    for p = pair
        if simplifyFraction(p(1) - p(2)) ~= 0
            error('There is something wrong.');
        end
    end
    disp('It is all right.');
end
\end{lstlisting}

\begin{lstlisting}[caption={\tt LemmaA\_1.m}, label=LemmaA-1]
function LemmaA_1
    pair = [];

    a = sym('a'); b = sym('b');
    phi = b^2/(a^2 + b^2);

    f = diff(phi, a);
    g = - 2*a*b^2/(a^2 + b^2)^2;
    h = - 2/sym(3)/b ...
        + (2*(a - b)^4 + 2*a*b*(2*a - b)^2)/3/b/(a^2 + b^2)^2;
    pair = [pair [f; g] [f; h]];

    f = diff(phi, a, 2);
    g = 2*b^2*(3*a^2 - b^2)/(a^2 + b^2)^3;
    pair = [pair [f; g]];

    f = diff(phi, b);
    g = 2*a^2*b/(a^2 + b^2)^2;
    pair = [pair [f; g]];

    f = diff(phi, b, 2);
    g = 2*a^2*(a^2 - 3*b^2)/(a^2 + b^2)^3;
    h = - 7/sym(9)/b^2 + ( ...
            3*a^2*b^2*(a^2 + b^2) ...
            + (7*a^2 + 3*b^2)*(2*a^2 - b^2)^2 ...
            + b^2*(13*a^2 - 5*b^2)^2 ...
        )/36/b^2/(a^2 + b^2)^3;
    pair = [pair [f; g] [f; h]];

    for p = pair
        if simplifyFraction(p(1) - p(2)) ~= 0
            error('There is something wrong.');
        end
    end
    disp('It is all right.');
end
\end{lstlisting}

\begin{lstlisting}[caption={\tt LemmaA\_2.m}, label=LemmaA-2]
function LemmaA_2
    pair = [];

    a = sym('a'); b = sym('b');
    phi = b^2/(a^2 + b^2);
    psi = b^2*(1 + 2*a)/(1 + 4*b^2)*phi;

    f = diff(psi, a);
    g = 2*b^2/(1 + 4*b^2)*phi ...
        + b^2*(1 + 2*a)/(1 + 4*b^2)*diff(phi, a);
    pair = [pair [f; g]];

    f = diff(psi, a, 2);
    g = 4*b^2/(1 + 4*b^2)*diff(phi, a) ...
        + b^2*(1 + 2*a)/(1 + 4*b^2)*diff(phi, a, 2);
    pair = [pair [f; g]];

    f = 4*b^2/(1 + 4*b^2)*(-2)/3/b ...
        + b^2*(1 + 2*a)/(1 + 4*b^2)*(-2)/b^2;
    g = -2*(3 + 6*a + 4*b)/3/(1 + 4*b^2);
    pair = [pair [f; g]];

    f = diff(psi, b);
    g = 2*b*(1 + 2*a)/(1 + 4*b^2)^2*phi ...
        + b^2*(1 + 2*a)/(1 + 4*b^2)*diff(phi, b);
    pair = [pair [f; g]];

    f = 2*b*(1 + 2*a)/(1 + 4*b^2)^2 ...
        + b^2*(1 + 2*a)/(1 + 4*b^2)*1/2/b;
    g = (1 + 2*a)*(2 + 16*b - 9*b^2)/10/(1 + 4*b^2) ...
        - (1 + 2*a)*(1 - b)/50/(1 + 4*b^2)^2 ...
            *(11*b^3 + 10*(1 - 3*b)^2 + b*(5 - 13*b)^2);
    pair = [pair [f; g]];

    f = diff(psi, b, 2);
    g = 2*(1 + 2*a)*(1 - 12*b^2)/(1 + 4*b^2)^3*phi ...
        + 4*b*(1 + 2*a)/(1 + 4*b^2)^2*diff(phi, b) ...
        + b^2*(1 + 2*a)/(1 + 4*b^2)*diff(phi, b, 2);
    pair = [pair [f; g]];

    f = 2*(1 + 2*a)*(1 - 12*b^2 - (5/sym(4) - 3*b^2)^2) ...
            /(1 + 4*b^2)^3*phi ...
        - b^2*(1 + 2*a)/(1 + 4*b^2)*7/sym(9)/b^2;
    g = -9*(1 + 2*a)/8/(1 + 4*b^2)*phi ...
        - 7*(1 + 2*a)/9/(1 + 4*b^2);
    pair = [pair [f; g]];

    for p = pair
        if simplifyFraction(p(1) - p(2)) ~= 0
            error('There is something wrong.');
        end
    end
    disp('It is all right.');
end
\end{lstlisting}

\begin{lstlisting}[caption={\tt LemmaA\_3.m}, label=LemmaA-3]
function LemmaA_3
    pair = [];

    a = sym('a'); b = sym('b');
    d1 = 1 - a; d2 = 1 - 2*a;
    omega = a*(1 - a)/(1 - a + a^2 + b^2);

    f = omega;
    g = 4*a*(1 - a)*(7 - 4*b^2)/21 ...
        - a*d1/21/(1 - a*d1 + b^2) ...
            *(3*d2^2 + 4*(1 - b^2)*(d2^2 + 4*b^2));
    pair = [pair [f; g]];

    f = diff(omega, a);
    g = 4*(1 - 2*a)*(2 - b^2 + 5*a*(1-a))/7 ...
        - d2/7/(1 - a*d1 + b^2)^2*( ...
            d2^2 + a*d1*(8*d2^2*(1 + 4*b^2) + 28*b^4) ...
            + 4*a^2*d1^2*(5*a*d1 + 21*b^2) + b^2*(5 - 4*b^4) ...
        );
    h = (1 - 2*a)*(1 + b^2)/(1 - a*d1 + b^2)^2;
    pair = [pair [f; g] [f; h]];

    f = diff(omega, a, 2);
    g = -2*(11*a*(1-a) + 2*b^2)/7/b^2 ...
        + 2*(1 + b^2)/7/b^2/(1 - a*d1 + b^2)^3*( ...
            a^3*d1^3*(a*d1 + 3*d2^2 + b^2)/(1 + b^2) ...
            + a^2*d1^2*(10*a*d1 + 19*d2^2 + b^2) ...
            + a*d1*d2^2*(1 + 6*d2^2 + 7*b^2) ...
            + a*d1*(2 - 3*b^2)^2 + d2^2*b^4 + 2*b^2*(1 - b^2)^2 ...
        );
    pair = [pair [f; g]];

    f = diff(omega, b);
    g = - 2*a*(1-a)*b/(1 - a*d1 + b^2)^2;
    h = - 1/sym(6)/b*(1 - ((1 - 7*a*d1 + b^2)^2 + 12*a*d1*d2^2) ...
                /(1 - a*d1 + b^2)^2 ...
        );
    pair = [pair [f; g] [g; h]];

    f = diff(omega, b, 2);
    g = - (24 - 64*a*(1 - a) - 17*b^2)/44/b^2 ...
        + 1/sym(44)/b^2/(1 - a*d1 + b^2)^3*( ...
            13*d2^2*b^2 + a^3*d1^3*(8 + 17*b^2) ...
            + d2^4*(1 + 6*b^2 + (3 - 2*a*d1+3*b^2)^2) ...
            + a^2*d1^2*(4 + 10*b^2 + 3*b^2*(1 - b^2)) ...
            + 2*(6*a*d1 + 7*d2^2)*(1 - b^2)*(1 - 3*b^2)^2 ...
            + b^2*(1 - b^2)*(34*d2^2 + 3*a*d1*(13 + 27*b^2) ...
                + 82*d2^2*(1 - b^2) + 17*b^4) ...
        );
    pair = [pair [f; g]];

    for p = pair
        if simplifyFraction(p(1) - p(2)) ~= 0
            error('There is something wrong.');
        end
    end
    disp('It is all right.');
end
\end{lstlisting}

\begin{lstlisting}[caption={\tt LemmaA\_4.m}, label=LemmaA-4]
function LemmaA_4
    pair = [];

    a = sym('a'); at = sym('at'); b = sym('b'); bt = sym('bt');
    d1 = 1 - a; d2 = 1 - 2*a;
    phi = b^2/(a^2 + b^2);
    psi = b^2*(1 + 2*a)/(1 + 4*b^2)*phi;
    psi_a = diff(psi, a);
    psi_b = diff(psi, b);
    psi_aa = diff(psi_a, a);
    psi_bb = diff(psi_b, b);

    f = L1ab(a, b);
    g = (1 + a^2 + (1 - a)^2 + 2*b^2)/28 ...
        - b^4/16/(a^2 + b^2)/((1 - a)^2 + b^2);
    h = (1 - a + a^2 + b^2)/14 ...
        - b^4*(1 + 2*a)/16/(a^2 + b^2)/(1 + 4*b^2) ...
        - b^4*(3 - 2*a)/16/((1 - a)^2 + b^2)/(1 + 4*b^2);
    j = (1 - a + a^2 + b^2)/14 - (psi + subs(psi, a, 1 - a))/16;
    pair = [pair [f; g] [f; h] [f; j]];

    f = (1 - a + a^2 + b^2)/14 - 1/sym(16)*( ...
            b^2*(1 + 2*a)/(1 + 4*b^2) ...
            + b^2*(1 + 2*(1 - a))/(1 + 4*b^2) ...
        );
    g = (1 - a + a^2 + b^2)/14 - b^2/4/(1 + 4*b^2);
    pair = [pair [f; g]];

    f = diff(L1ab(a, b), a);
    g = - (1 - 2*a)/14 - (psi_a - subs(psi_a, a, 1 - a))/16;
    pair = [pair [f; g]];

    f = - (1 - 2*a)/14 - 1/sym(16)*( ...
            2*b^2/(1 + 4*b^2) ...
            + 2*b*(1 + 2*(1 - a))/3/(1 + 4*b^2) ...
        );
    g = - (1 - 2*a)/14 - b*(3 - 2*a + 3*b)/24/(1 + 4*b^2);
    pair = [pair [f; g]];

    f = - (1 - 2*a)/14 - 1/sym(16)*( ...
            -2*b*(1 + 2*a)/3/(1 + 4*b^2) - 2*b^2/(1 + 4*b^2) ...
        );
    g = - (1 - 2*a)/14 + b*(1 + 2*a + 3*b)/24/(1 + 4*b^2);
    pair = [pair [f; g]];

    f = subs(diff(L1ab(a, b), a, 2), a, at);
    g = 1/sym(7) ...
        - (subs(psi_aa, a, at) + subs(psi_aa, a, 1 - at))/16;
    pair = [pair [f; g]];

    f = 1/sym(7) - 1/sym(16)*( ...
            - 2*(3 + 6*at + 4*b)/3/(1 + 4*b^2) ...
            - 2*(3 + 6*(1 - at) + 4*b)/3/(1 + 4*b^2) ...
        );
    g = 1/sym(7) + (3 + 2*b)/6/(1 + 4*b^2);
    pair = [pair [f; g]];

    f = diff(L1ab(a, b), b);
    g = b/7 - (psi_b + subs(psi_b, a, 1 - a))/16;
    pair = [pair [f; g]];

    f = b/7 - 1/sym(16)*( ...
            (1 + 2*a)*(2 + 16*b - 9*b^2)/10/(1 + 4*b^2) ...
            + (1 + 2*(1 - a))*(2 + 16*b - 9*b^2) ...
                /10/(1 + 4*b^2) ...
        );
    g = b/7 - (2 + 16*b - 9*b^2)/40/(1 + 4*b^2);
    pair = [pair [f; g]];

    f = subs(diff(L1ab(a, b), b, 2), b, bt);
    g = 1/sym(7) ...
        - subs(psi_bb + subs(psi_bb, a, 1 - a), b, bt)/16;
    pair = [pair [f; g]];

    f = 1/sym(7) - 1/sym(16)*( ...
            - 21*(1 + 2*a)/11/(1 + 4*bt^2) ...
            - 21*(1 + 2*(1 - a))/11/(1 + 4*bt^2) ...
        );
    g = 1/sym(7) + 21/sym(44)/(1 + 4*bt^2);
    pair = [pair [f; g]];

    f = 168*(1 + 4*b^2)*( ...
            2*((1 - a + a^2 + b^2)/14 - b^2/4/(1 + 4*b^2)) ...
            + b*( ...
                - (1 - 2*a)/14 ...
                - b*(3 - 2*a + 3*b)/24/(1 + 4*b^2) ...
            ) ...
        );
    g = 6*(d2^2 + 4*a*b)*(1 + 4*b^2) ...
        + b^2*(3 + 14*a + 3*b + 46*(1 - b)^2) ...
        + 2*(3 - b - 5*b^2)^2;
    pair = [pair [f; g]];

    f = 168*(1 + 4*b^2)*( ...
            2*((1 - a + a^2 + b^2)/14 - b^2/4/(1 + 4*b^2)) ...
            - b*( ...
                - (1 - 2*a)/14 ...
                + b*(1 + 2*a + 3*b)/24/(1 + 4*b^2) ...
            ) ...
        );
    g = 6*(d2^2 + 4*d1*b)*(1 + 4*b^2) ...
        + b^2*(3 + 14*d1 + 3*b + 46*(1 - b)^2) ...
        + 2*(3 - b - 5*b^2)^2;
    pair = [pair [f; g]];

    f = 84*(1 + 4*b^2)*( ...
            5*((1 - a + a^2 + b^2)/14 - b^2/4/(1 + 4*b^2)) ...
            - b^2*(1/sym(7) + (3 + 2*b)/6/(1 + 4*b^2)) ...
        );
    g = 2*a*d1 + 2*d2^2*(4 + 15*b^2) ...
        + 22*(1 + 2*b)*(1 - b)^2 + 9*b^2*(1 + 2*(1 - 2*b)^2);
    pair = [pair [f; g]];

    f = 280*(1 + 4*b^2)*( ...
            2*((1 - a + a^2 + b^2)/14 - b^2/4/(1 + 4*b^2)) ...
            + b*(b/7 - (2 + 16*b - 9*b^2)/40/(1 + 4*b^2)) ...
        );
    g = 64*a*d1*b*(1 - 2*b)^2 ...
        + 2*(2*a*d1*(30 + b^2) + d2^2*(20 + 13*b))*(1 - b) ...
        + b^2*(11 + 3*d2^2 + 320*b^2 + 63*d2^2*b);
    pair = [pair [f; g]];

    f = 28*(1 + 4*b^2)*( ...
            2*((1 - a + a^2 + b^2)/14 - b^2/4/(1 + 4*b^2)) ...
            - b*b/7 ...
        );
    g = 1 + 2*(1 - b^2) + d2^2*(1 + 4*b^2);
    pair = [pair [f; g]];

    f = 14*(1 + 4*b^2)*( ...
            4*((1 - a + a^2 + b^2)/14 - b^2/4/(1 + 4*b^2)) ...
            - b^2*(1/sym(7) + 1/sym(2)/(1 + 4*b^2)) ...
        );
    g = 1 + b^2 + d2^2*(1 + 4*b^2) + 2*(1 - 2*b^2)^2;
    pair = [pair [f; g]];

    for p = pair
        if simplifyFraction(p(1) - p(2)) ~= 0
            error('There is something wrong.');
        end
    end
    disp('It is all right.');
end
\end{lstlisting}

\begin{lstlisting}[caption={\tt LemmaA\_5.m}, label=LemmaA-5]
function LemmaA_5
    pair = [];

    a = sym('a'); at = sym('at'); b = sym('b'); bt = sym('bt');
    d1 = 1 - a; d2 = 1 - 2*a;
    phi = b^2/(a^2 + b^2);
    psi = b^2*(1 + 2*a)/(1 + 4*b^2)*phi;
    psi_a = diff(psi, a);
    psi_b = diff(psi, b);
    psi_aa = diff(psi_a, a);
    psi_bb = diff(psi_b, b);

    f = L2ab(a, b);
    g = (1 + a^2 + (1 - a)^2 + 2*b^2)/54 ...
        - b^4/32/(a^2 + b^2)/((1 - a)^2 + b^2);
    h = (1 - a + a^2 + b^2)/27 ...
        - b^4*(1 + 2*a)/32/(a^2 + b^2)/(1 + 4*b^2) ...
        - b^4*(3 - 2*a)/32/((1 - a)^2 + b^2)/(1 + 4*b^2);
    j = (1 - a + a^2 + b^2)/27 - (psi + subs(psi, a, 1 - a))/32;
    pair = [pair [f; g] [f; h] [f; j]];

    f = (1 - a + a^2 + b^2)/27 - 1/sym(32)*( ...
            b^2*(1 + 2*a)/(1 + 4*b^2) ...
            + b^2*(1 + 2*(1 - a))/(1 + 4*b^2) ...
        );
    g = (1 - a + a^2 + b^2)/27 - b^2/8/(1 + 4*b^2);
    pair = [pair [f; g]];

    f = diff(L2ab(a, b), a);
    g = - (1 - 2*a)/27 - (psi_a - subs(psi_a, a, 1 - a))/32;
    pair = [pair [f; g]];

    f = - (1 - 2*a)/27 - 1/sym(32)*( ...
            2*b^2/(1 + 4*b^2) ...
            + 2*b*(1 + 2*(1 - a))/3/(1 + 4*b^2) ...
        );
    g = - (1 - 2*a)/27 - b*(3 - 2*a + 3*b)/48/(1 + 4*b^2);
    pair = [pair [f; g]];

    f = - (1 - 2*a)/27 - 1/sym(32)*( ...
            -2*b*(1 + 2*a)/3/(1 + 4*b^2) - 2*b^2/(1 + 4*b^2) ...
        );
    g = - (1 - 2*a)/27 + b*(1 + 2*a + 3*b)/48/(1 + 4*b^2);
    pair = [pair [f; g]];

    f = subs(diff(L2ab(a, b), a, 2), a, at);
    g = 2/sym(27) ...
        - (subs(psi_aa, a, at) + subs(psi_aa, a, 1 - at))/32;
    pair = [pair [f; g]];

    f = 2/sym(27) - 1/sym(32)*( ...
            - 2*(3 + 6*at + 4*b)/3/(1 + 4*b^2) ...
            - 2*(3 + 6*(1 - at) + 4*b)/3/(1 + 4*b^2) ...
        );
    g = 2/sym(27) + (3 + 2*b)/12/(1 + 4*b^2);
    pair = [pair [f; g]];

    f = diff(L2ab(a, b), b);
    g = 2*b/27 - (psi_b + subs(psi_b, a, 1 - a))/32;
    pair = [pair [f; g]];

    f = 2*b/27 - 1/sym(32)*( ...
            (1 + 2*a)*(2 + 16*b - 9*b^2)/10/(1 + 4*b^2) ...
            + (1 + 2*(1 - a))*(2 + 16*b - 9*b^2) ...
                /10/(1 + 4*b^2) ...
        );
    g = 2*b/27 - (2 + 16*b - 9*b^2)/80/(1 + 4*b^2);
    pair = [pair [f; g]];

    f = subs(diff(L2ab(a, b), b, 2), b, bt);
    g = 2/sym(27) ...
        - subs(psi_bb + subs(psi_bb, a, 1 - a), b, bt)/32;
    pair = [pair [f; g]];

    f = 2/sym(27) - 1/sym(32)*( ...
            - 21*(1 + 2*a)/11/(1 + 4*bt^2) ...
            - 21*(1 + 2*(1 - a))/11/(1 + 4*bt^2) ...
        );
    g = 2/sym(27) + 21/sym(88)/(1 + 4*bt^2);
    pair = [pair [f; g]];

    f = 432*(1 + 4*b^2)*( ...
            2*((1 - a + a^2 + b^2)/27 - b^2/8/(1 + 4*b^2)) ...
            + b*( ...
                - (1 - 2*a)/27 ...
                - b*(3 - 2*a + 3*b)/48/(1 + 4*b^2) ...
            ) ...
        );
    g = 8 + 8*d2^2 + 16*(1 - b) ...
        + b^2*(18*a + 128*a^2 + 9*b + 28*b^2) ...
        + b*(32*a + 25*b)*(1 - 2*b)^2;
    pair = [pair [f; g]];

    f = 432*(1 + 4*b^2)*( ...
            2*((1 - a + a^2 + b^2)/27 - b^2/8/(1 + 4*b^2)) ...
            - b*( ...
                - (1 - 2*a)/27 ...
                + b*(1 + 2*a + 3*b)/48/(1 + 4*b^2) ...
            ) ...
        );
    g = 8 + 8*d2^2 + 16*(1 - b) ...
        + b^2*(18*d1 + 128*d1^2 + 9*b + 28*b^2) ...
        + b*(32*d1 + 25*b)*(1 - 2*b)^2;
    pair = [pair [f; g]];

    f = 216*(1 + 4*b^2)*( ...
            5*((1 - a + a^2 + b^2)/27 - b^2/8/(1 + 4*b^2)) ...
            - b^2*(2/sym(27) + (3 + 2*b)/12/(1 + 4*b^2)) ...
        );
    g = 9*b*(1 - b) + 15*(2 + b)*(1 - 2*b)^2 ...
        + 10*d2^2*(1 + 4*b^2) ...
        + 96*b*(1 + b)*(1 - b)^2;
    pair = [pair [f; g]];

    f = 2160*(1 + 4*b^2)*( ...
            2*((1 - a + a^2 + b^2)/27 - b^2/8/(1 + 4*b^2)) ...
            + b*(2*b/27 - (2 + 16*b - 9*b^2)/80/(1 + 4*b^2)) ...
        );
    g = 337*b^3 + 40*d2^2*(1 + 4*b^2) + 30*(2 - b - 7*b^2)^2 ...
        + 2*b*(1 - b)*(33 + 352*b + 95*b^2);
    pair = [pair [f; g]];

    f = 108*(1 + 4*b^2)*( ...
            2*((1 - a + a^2 + b^2)/27 - b^2/8/(1 + 4*b^2)) ...
            - b*2*b/27 ...
        );
    g = 3 + 3*(1 - b^2) + 2*d2^2*(1 + 4*b^2);
    pair = [pair [f; g]];

    f = 108*(1 + 4*b^2)*( ...
            4*((1 - a + a^2 + b^2)/27 - b^2/8/(1 + 4*b^2)) ...
            - b^2*(2/sym(27) + 1/sym(4)/(1 + 4*b^2)) ...
        );
    g = 4 + 7*b^2 + 4*d2^2*(1 + 4*b^2) + 8*(1 - 2*b^2)^2;
    pair = [pair [f; g]];

    for p = pair
        if simplifyFraction(p(1) - p(2)) ~= 0
            error('There is something wrong.');
        end
    end
    disp('It is all right.');
end
\end{lstlisting}

\begin{lstlisting}[caption={\tt LemmaA\_6.m}, label=LemmaA-6]
function LemmaA_6
    pair = [];

    a = sym('a'); at = sym('at'); b = sym('b'); bt = sym('bt');
    d1 = 1 - a; d2 = 1 - 2*a;
    omega = a*(1 - a)/(1 - a + a^2 + b^2);
    omega_a = diff(omega, a);
    omega_b = diff(omega, b);
    omega_aa = diff(omega_a, a);
    omega_bb = diff(omega_b, b);

    f = L3ab(a, b);
    g = (b^2 + (1 - a + a^2 + b^2)^2)/83 ...
        + (2*a*(1 - a) - 3*b^2)/96 ...
        - omega/48;
    pair = [pair [f; g]];

    f = (b^2 + (1 - a + a^2 + b^2)^2)/84 ...
        + (2*a*(1 - a) - 3*b^2)/96 ...
        - a*(1 - a)*(7 - 4*b^2)/252;
    g = (24*(1 - a + a^2 + b^2)^2 ...
            - 39*b^2 - 2*a*(1 - a)*(7 - 16*b^2))/2016;
    pair = [pair [f; g]];

    f = diff(L3ab(a, b), a);
    g = - 2*(1 - 2*a)*(1 - a + a^2 + b^2)/83 ...
        + (1 - 2*a)/48 ...
        - omega_a/48;
    pair = [pair [f; g]];

    f = - 25*(1 - 2*a)*(1 - a + a^2 + b^2)/1008 ...
        + (1 - 2*a)/48 ...
        - (1 - 2*a)*(2 - b^2 + 5*a*(1 - a))/84;
    g = - (1 - 2*a)*(28 + 13*b^2 + 35*a*(1 - a))/1008;
    pair = [pair [f; g]];

    f = - 2*(1 - 2*a)*(1 - a + a^2 + b^2)/144 ...
        + (1 - 2*a)/48 ...
        - (1 - 2*a)/96;
    g = - d2*(d2^2 + 4*b^2)/288;
    pair = [pair [f; g]];

    f = subs(diff(L3ab(a, b), a, 2), a, at);
    g = (3 + 3*(1 - 2*at)^2 + 4*b^2)/83 - 1/sym(24) ...
         - subs(omega_aa, a, at)/48;
    pair = [pair [f; g]];

    f = 25*(3 + 3*(1 - 2*at)^2 + 4*b^2)/2016 - 1/sym(24) ...
        + (11*at*(1 - at) + 2*b^2)/168/b^2;
    g = (11 - 25*b^2)*(at - a)*(1 - at - a)/168/b^2 ...
        + (2*a*(1 - a)*(33 - 75*b^2) + 5*b^2*(9 + 10*b^2)) ...
            /1008/b^2;
    pair = [pair [f; g]];

    f = 1/sym(504)/b^2 + ( ...
            2*a*(1 - a)*(33 - 75*b^2) ...
            + 5*b^2*(9 + 10*b^2) ...
        )/1008/b^2;
    g = (2 + 2*a*(1 - a)*(33 - 75*b^2) ...
            + 5*b^2*(9 + 10*b^2))/1008/b^2;
    pair = [pair [f; g]];

    f = diff(L3ab(a, b), b);
    g = b*(5 + (1 - 2*a)^2 + 4*b^2)/83 - b/16 - omega_b/48;
    pair = [pair [f; g]];

    f = 23*b*(5 + (1 - 2*a)^2 + 4*b^2)/1920 - b/16 ...
        + a*(1 - a)*b/96;
    g = b*(9*d2^2 + 46*b^2)/960;
    pair = [pair [f; g]];

    f = 17*b*(5 + (1 - 2*a)^2 + 4*b^2)/1344 - b/16 ...
        + 1/sym(168)/b;
    g = (4 + 9*b^2 - 34*a*(1 - a)*b^2 + 34*b^4)/672/b;
    pair = [pair [f; g]];

    f = subs(diff(L3ab(a, b), b, 2), b, bt);
    g = (5 + (1 - 2*a)^2 + 12*bt^2)/83 - 1/sym(16) ...
        - subs(omega_bb, b, bt)/48;
    pair = [pair [f; g]];

    f = (5 + (1 - 2*a)^2 + 12*bt^2)/83 - 1/sym(16) ...
        + (24 - 64*a*(1 - a) - 17*bt^2)/2112/bt^2;
    g = ((1 - 2*a)^2 + 12*bt^2)/83 ...
        + (4*a*(1 - a) + 3*(1 - 2*a)^2)/264/bt^2 ...
        - 1807/sym(175296);
    pair = [pair [f; g]];

    f = (2*(1 - 2*a)^2 + 19*b^2)/126 ...
        + (4*a*(1 - a) + 3*(1 - 2*a)^2)/252/b^2 - 1/sym(126);
    g = (3 + 2*b^2 + 38*b^4 - 8*a*(1 - a)*(1 + 2*b^2)) ...
            /252/b^2;
    pair = [pair [f; g]];

    f = 1008*( ...
            2*(24*(1 - a + a^2 + b^2)^2 - 39*b^2 ...
                - 2*a*(1 - a)*(7 - 16*b^2))/2016 ...
            + b*(- (1 - 2*a)*(28 + 13*b^2 ...
                    + 35*a*(1 - a))/1008) ...
        );
    g = 6*d2^4 + 6*a*d1*(4 - b^2) + d2*b*(a*d1 + b^2) ...
        + 4*(1 - 2*b^2)^2 + b^2*(4 + b^2) ...
        + (d2 - b)^2*(14 + 18*a*d1 + 7*b^2);
    pair = [pair [f; g]];

    f = 1008*( ...
            2*(24*(1 - a + a^2 + b^2)^2 - 39*b^2 ...
                - 2*a*(1 - a)*(7 - 16*b^2))/2016 ...
            - b*0 ...
        );
    g = 34*a*d1 + 5*b^2 + 24*a^2*d1^2 + 4*d2^2*(6+b^2) + 24*b^4;
    pair = [pair [f; g]];

    f = 1008*( ...
            4*(24*(1 - a + a^2 + b^2)^2 - 39*b^2 ...
                - 2*a*(1 - a)*(7 - 16*b^2))/2016 ...
            - b^2*(2 + 2*a*(1 - a)*(33 - 75*b^2) ...
                + 5*b^2*(9 + 10*b^2))/1008/b^2 ...
        );
    g = 17*d2^4 + 2*a*d1*(12*a*d1+31*d2^2+b^2) ...
        + (1 - b^2)*(29*d2^2+2*b^2);
    pair = [pair [f; g]];

    f = 672*( ...
            3*(24*(1 - a + a^2 + b^2)^2 - 39*b^2 ...
                - 2*a*(1 - a)*(7 - 16*b^2))/2016 ...
            + b*0 ...
        );
    g = 34*a*d1 + 5*b^2 + 24*a^2*d1^2 + 4*d2^2*(6+b^2) + 24*b^4;
    pair = [pair [f; g]];

    f = 336*( ...
            3*(24*(1 - a + a^2 + b^2)^2 - 39*b^2 ...
                - 2*a*(1 - a)*(7 - 16*b^2))/2016 ...
            - b*(4 + 9*b^2 - 34*a*(1 - a)*b^2 + 34*b^4) ...
                /672/b ...
        );
    g = a*d1*(1 + 17*d2^2) + 5*d2^4 ...
        + (1 - b^2)*(11*a*d1 + 5*d2^2 + 5*b^2);
    pair = [pair [f; g]];

    f = 252*( ...
            8*(24*(1 - a + a^2 + b^2)^2 - 39*b^2 ...
                - 2*a*(1 - a)*(7 - 16*b^2))/2016 ...
            - b^2*(3 + 2*b^2 + 38*b^4 - 8*a*(1 - a)*(1 + 2*b^2)) ...
                /252/b^2 ...
        );
    g = 14*d2^2 + 2*a*d1*(1 + 12*a*d1) + 7*(1 - b^2)*(1 + 2*b^2);
    pair = [pair [f; g]];

    for p = pair
        if simplifyFraction(p(1) - p(2)) ~= 0
            error('There is something wrong.');
        end
    end
    disp('It is all right.');
end
\end{lstlisting}

\begin{lstlisting}[caption={\tt LemmaA\_7.m}, label=LemmaA-7]
function LemmaA_7
    pair = [];

    a = sym('a'); at = sym('at'); b = sym('b'); bt = sym('bt');
    d1 = 1 - a; d2 = 1 - 2*a;
    phi = b^2/(a^2 + b^2);
    phi_a = diff(phi, a);
    phi_b = diff(phi, b);
    phi_aa = diff(phi_a, a);
    phi_bb = diff(phi_b, b);

    f = L4ab(a, b);
    g = (a^2 + b^2)*((1 - a)^2 + b^2)/4/b^2 ...
        - (1 + a^2 + (1 - a)^2 + 2*b^2)/30 ...
        - b^2/20*(1 + 1/(a^2 + b^2) + 1/((1-a)^2 + b^2));
    h = a^2*(1 - a)^2/4/b^2 + (11 - 26*a*(1 - a))/60 ...
        + 2*b^2/15 - (phi + subs(phi, a, 1 - a))/20;
    pair = [pair [f; g] [f; h]];

    f = a^2*(1 - a)^2/4/b^2 + (11 - 26*a*(1 - a))/60 ...
        + 2*b^2/15 - (1 + 1)/20;
    g = a^2*(1 - a)^2/4/b^2 + (5 - 26*a*(1 - a) + 8*b^2)/60;
    pair = [pair [f; g]];

    f = diff(L4ab(a, b), a);
    g = (1 - 2*a)*(a*(1 - a)/2/b^2 - 13/sym(30)) ...
        - (phi_a - subs(phi_a, a, 1 - a))/20;
    pair = [pair [f; g]];

    f = subs(diff(L4ab(a, b), a, 2), a, at);
    g = (1 - 6*at*(1 - at))/2/b^2 + 13/sym(15) ...
        - (subs(phi_aa, a, at) + subs(phi_aa, a, 1 - at))/20;
    pair = [pair [f; g]];

    f = (1 - 6*at*(1 - at))/2/b^2 + 13/sym(15) ...
        + 1/sym(20)*(2/b^2 + 2/b^2);
    g = (7 - 30*a*(1 - a) - 30*(at - a)*(1 - at - a))/10/b^2 ...
        + 13/sym(15);
    pair = [pair [f; g]];

    f = (7 - 30*a*(1 - a) + 30*(b/45))/10/b^2 + 13/sym(15);
    g = (21 - 90*a*(1 - a) + 2*b)/30/b^2 + 13/sym(15);
    pair = [pair [f; g]];

    f = diff(L4ab(a, b), b);
    g = - a^2*(1 - a)^2/2/b^3 + 4*b/15 ...
        - (phi_b + subs(phi_b, a, 1 - a))/20;
    pair = [pair [f; g]];

    f = - a^2*(1 - a)^2/2/b^3 + 4*b/15 ...
        - 1/sym(20)*(1/2/b + 1/2/b);
    g = - a^2*(1 - a)^2/2/b^3 - 1/sym(20)/b + 4*b/15;
    pair = [pair [f; g]];

    f = subs(diff(L4ab(a, b), b, 2), b, bt);
    g = 3*a^2*(1 - a)^2/2/bt^4 + 4/sym(15) ...
        - subs(phi_bb + subs(phi_bb, a, 1 - a), b, bt)/20;
    pair = [pair [f; g]];

    f = 3*a^2*(1 - a)^2/2/bt^4 + 4/sym(15) ...
        + 1/sym(20)*(7/9/bt^2 + 7/9/bt^2);
    g = 3*a^2*(1 - a)^2/2/bt^4 + 7/sym(90)/bt^2 + 4/sym(15);
    pair = [pair [f; g]];

    f = 60*b^2*( ...
            3*(a^2*(1 - a)^2/4/b^2 ...
                + (5 - 26*a*(1 - a))/60 + 2*b^2/15) ...
            + b*((1 - 2*a)*(a*(1-a)/2/b^2 - 13/sym(30)) ...
                - 1/sym(30)/b) ...
        );
    g = 11*a*d1^2*b + 5*(3*a*d1 - b)^2 + a*b*(7*d1 - 5*b)^2 ...
        + b^2*(2*(1 - b) + a*(4 + 2*a + 3*b) + 6*(d1 - 2*b)^2);
    pair = [pair [f; g]];

    f = 60*b^2*( ...
            3*(a^2*(1 - a)^2/4/b^2 ...
                + (5 - 26*a*(1 - a))/60 + 2*b^2/15) ...
            - b*((1 - 2*a)*(a*(1 - a)/2/b^2 - 13/sym(30)) ...
                + 1/sym(30)/b) ...
        );
    g = 11*a^2*d1*b + 5*(3*a*d1 - b)^2 + d1*b*(7*a - 5*b)^2 ...
        + b^2*(2*(1 - b) + d1*(4 + 2*d1 + 3*b) + 6*(a - 2*b)^2);
    pair = [pair [f; g]];

    f = 60*b^2*( ...
            9*(a^2*(1 - a)^2/4/b^2 ...
                + (5 - 26*a*(1 - a))/60 + 2*b^2/15) ...
            - b^2*((21 - 90*a*(1-a) + 2*b)/30/b^2 + 13/sym(15))...
        );
    g = 86*a^2*d1^2 + (7*a*d1 - 4*b^2)^2 ...
        + b^2*(2 + 2*a*d1 + (1 - 2*b)^2);
    pair = [pair [f; g]];

    f = 60*b^2*( ...
            3*(a^2*(1 - a)^2/4/b^2 ...
                + (5 - 26*a*(1 - a))/60 + 2*b^2/15) ...
            + b*(- a^2*(1 - a)^2/2/b^3 - 1/sym(20)/b + 4*b/15) ...
        );
    g = b^2*(12*d2^2 + 25*b^2) + 15*(a*d1 - b^2)^2;
    pair = [pair [f; g]];

    f = 60*b^2*( ...
            3*(a^2*(1 - a)^2/4/b^2 ...
                + (5 - 26*a*(1 - a))/60 + 2*b^2/15) ...
            - b*(- a^2*(1 - a)^2/2/b^3 + 4*b/15) ...
        );
    g = b^2*(3 + 12*d2^2 + 5*b^2) + 3*(5*a*d1 - b^2)^2;
    pair = [pair [f; g]];

    f = 60*b^2*( ...
            9*(a^2*(1 - a)^2/4/b^2 ...
                + (5 - 26*a*(1 - a))/60 + 2*b^2/15) ...
            - b^2*(49*a^2*(1 - a)^2/30/b^4 ...
                + 1/sym(12)/b^2 + 4/sym(15)) ...
        );
    g = b^2*(40*d2^2 + 19*b^2) + 37*(a*d1 - b^2)^2;
    pair = [pair [f; g]];

    for p = pair
        if simplifyFraction(p(1) - p(2)) ~= 0
            error('There is something wrong.');
        end
    end
    disp('It is all right.');
end
\end{lstlisting}

\begin{lstlisting}[caption={\tt LemmaA\_8.m}, label=LemmaA-8]
function LemmaA_8
    pair = [];

    a = sym('a'); b = sym('b');
    d1 = 1 - a; d2 = 1 - 2*a;

    f = 112*(a^2 + b^2)*(d1^2 + b^2)/b^2*(L1ab(a,b) - L1ab(a,0));
    g = 8*(a*d1 - b^2)^2 + b^2;
    pair = [pair [f; g]];

    f = 864*(a^2 + b^2)*(d1^2 + b^2)/b^2*(L2ab(a,b) - L2ab(a,0));
    g = 32*(a*d1 - b^2)^2 + 5*b^2;
    pair = [pair [f; g]];

    f = 7968*(d1 + a^2 + b^2)*(d1 + a^2)/b^2 ...
        *(L3ab(a,b) - L3ab(a,0));
    g = 52*a*d1 + 39*d2^2 + 3*a^2*d1^2*(125 + 16*d2^2) ...
         + 3*b^2*(d1 + a^2)*(21 + 32*b^2 + 24*d2^2);
    pair = [pair [f; g]];

    for p = pair
        if simplifyFraction(p(1) - p(2)) ~= 0
            error('There is something wrong.');
        end
    end
    disp('It is all right.');
end
\end{lstlisting}

\begin{lstlisting}[caption={\tt LemmaA\_9.m}, label=LemmaA-9]
function LemmaA_9
    a = sym('a'); b = sym('b');
    d1 = 1 - a; d2 = 1 - 2*a;

    S = b/2;
    c1 = 2*((1-a)^2+b^2)/b; c2 = 2*(a^2+b^2)/b; c3 = 2/b;
    u1 = sym('u1'); u2 = sym('u2'); u3 = sym('u3');
    w1 = sym('w1'); w2 = sym('w2'); w3 = sym('w3');
    w4 = sym('w4'); w5 = sym('w5'); w6 = sym('w6');
    w7 = sym('w7'); w8 = sym('w8'); w9 = sym('w9');

    G1 =  F_beta_1(S/4, c1, c2, c3,  0, u3, u2, w7, w5, w3) ...
        + F_beta_1(S/4, c1, c2, c3, u3,  0, u1, w1, w8, w6) ...
        + F_beta_1(S/4, c1, c2, c3, u2, u1,  0, w4, w2, w9) ...
        + F_beta_1(S/4, c1, c2, c3, u1, u2, u3,-w7,-w8,-w9);

    G2 =  F_beta_2(S/4, c1, c2, c3,  0, u3, u2, w7, w5, w3) ...
        + F_beta_2(S/4, c1, c2, c3, u3,  0, u1, w1, w8, w6) ...
        + F_beta_2(S/4, c1, c2, c3, u2, u1,  0, w4, w2, w9) ...
        + F_beta_2(S/4, c1, c2, c3, u1, u2, u3,-w7,-w8,-w9);

    D = c1*c2*c3/16 - 41*(c1 + c2 + c3)/1080 ...
            - (1/c1 + 1/c2 + 1/c3)/5;
    H = 48*c1^2*c2^2*c3^2*(S*D*G2 - G1);

    E = 48*D - c1 - c2 - c3;
    q1 = 2*E + 2*c1 + c2 + c3;
    q2 = 2*E + 2*c2 + c3 + c1;
    q3 = 2*E + 2*c3 + c1 + c2;
    r1 = c1*(12*D*(E + c1) + 1)/(2*E + c1 + 2*c2 + 2*c3)^2;
    r2 = c2*(12*D*(E + c2) + 1)/(2*E + c2 + 2*c3 + 2*c1)^2;
    r3 = c3*(12*D*(E + c3) + 1)/(2*E + c3 + 2*c1 + 2*c2)^2;
    v1 = c2*c3*(w1 + w4);
    v2 = c3*c1*(w2 + w5);
    v3 = c1*c2*(w3 + w6);
    v4 = w1 + w2 - w3 + w4 - w5 + w6 + 2*(w8 + w9);
    v5 = w2 + w3 - w1 + w5 - w6 + w4 + 2*(w9 + w7);
    v6 = w3 + w1 - w2 + w6 - w4 + w5 + 2*(w7 + w8);
    v7 = c1*c2*c3*(4*q1*(w4 - w8 + w9) - q1*(v5 - v6) ...
                - (c2 - c3)*v4);
    v8 = c1*c2*c3*(4*q2*(w5 - w9 + w7) - q2*(v6 - v4) ...
                - (c3 - c1)*v5);
    v9 = c1*c2*c3*(4*q3*(w6 - w7 + w8) - q3*(v4 - v5) ...
                - (c1 - c2)*v6);

    H2 = 16*( ...
          1/c2/c3/q1*(c1*c2*c3*(q1*(u1 - u2 - u3) ...
                - 6*D*(c1 - c2 - c3)*v4) - 96*D*v1)^2 ...
        + 1/c3/c1/q2*(c1*c2*c3*(q2*(u2 - u3 - u1) ...
                - 6*D*(c2 - c3 - c1)*v5) - 96*D*v2)^2 ...
        + 1/c1/c2/q3*(c1*c2*c3*(q3*(u3 - u1 - u2) ...
                - 6*D*(c3 - c1 - c2)*v6) - 96*D*v3)^2 ...
    ) ...
    + ( ...
          1/c2/c3/q1*(q1*(2*c1*v1 - c2*v2 + c3*v3) ...
                + (c2 - c3)*(96*D - c1)*v1 - v7)^2 ...
        + 1/c3/c1/q2*(q2*(2*c2*v2 - c3*v3 + c1*v1) ...
                + (c3 - c1)*(96*D - c2)*v2 - v8)^2 ...
        + 1/c1/c2/q3*(q3*(2*c3*v3 - c1*v1 + c2*v2) ...
                + (c1 - c2)*(96*D - c3)*v3 - v9)^2 ...
    ) ...
    + 4*( ...
          c1/q1/r1*(v1 + r1*c2*c3*(96*D - c1)*v4)^2 ...
        + c2/q2/r2*(v2 + r2*c3*c1*(96*D - c2)*v5)^2 ...
        + c3/q3/r3*(v3 + r3*c1*c2*(96*D - c3)*v6)^2 ...
    ) ...
    + 8*( ...
          (q1*c1 + 3*D*(c2 + c3)*(c2 - c3)^2) ...
                /q1/(12*D*(E + c1) + 1)*v1^2 ...
        + (q2*c2 + 3*D*(c3 + c1)*(c3 - c1)^2) ...
                /q2/(12*D*(E + c2) + 1)*v2^2 ...
        + (q3*c3 + 3*D*(c1 + c2)*(c1 - c2)^2) ...
                /q3/(12*D*(E + c3) + 1)*v3^2 ...
    ) ...
    + 2/E*( ...
          ((c2 - c3)^2 + 32)/(E + c1) ...
                *(c1 + E/(12*D*(E + c1) + 1))*v1^2 ...
        + ((c3 - c1)^2 + 32)/(E + c2) ...
                *(c2 + E/(12*D*(E + c2) + 1))*v2^2 ...
        + ((c1 - c2)^2 + 32)/(E + c3) ...
                *(c3 + E/(12*D*(E + c3) + 1))*v3^2 ...
    ) ...
    + 2/E*( ...
          (8*E*(3*D*c1 - 1)*c1 - (c2 - c3)^2 - 32)*v1^2 ...
        + (8*E*(3*D*c2 - 1)*c2 - (c3 - c1)^2 - 32)*v2^2 ...
        + (8*E*(3*D*c3 - 1)*c3 - (c1 - c2)^2 - 32)*v3^2 ...
        + 48*D*E*( ...
              (c1*c2 - 4)*v1*v2 ...
            + (c2*c3 - 4)*v2*v3 ...
            + (c3*c1 - 4)*v3*v1 ...
        ) ...
    );

    e1 = 45*b^3*E/4;
    e2 = 108*(a*d1 - b^2)^2*(1 + d2^2 + 4*b^2)/c1/c2 ...
        + 270*a^2*d1^2 ...
        + b^2*(47*a*d1 + 88*d2^2 + (1 - 100*b^2)) + 81*b^4;
    if simplifyFraction(e1 - e2) ~= 0
        error('There is something wrong.');
    end

    %%  Verify the equality term by term.  %%
    variables = [u1 u2 u3 w1 w2 w3 w4 w5 w6 w7 w8 w9];
    for p = 1:12
        df1 = diff(H - H2, variables(p));
        for q = p:12
            df2 = simplifyFraction(diff(df1, variables(q)));
            if df2 ~= 0
                error('There is something wrong.');
            end
        end
    end
    disp('It is all right.');
end
\end{lstlisting}

\begin{lstlisting}[caption={\tt LemmaA\_11.m}, label=LemmaA-11]
function LemmaA_11
    pair = [];

    a = sym('a'); b = sym('b');
    d1 = 1 - a; d2 = 1 - 2*a;
    c1 = 2*((1-a)^2 + b^2)/b; c2 = 2*(a^2 + b^2)/b; c3 = 2/b;
    D = c1*c2*c3/16 - 41*(c1 + c2 + c3)/1080 ...
        - (1/c1 + 1/c2 + 1/c3)/5;
    E = 48*D - c1 - c2 - c3;
    A1 = 8*E*(3*D*c1 - 1)*c1 - (c2 - c3)^2 - 32;
    A2 = 8*E*(3*D*c2 - 1)*c2 - (c3 - c1)^2 - 32;
    A3 = 8*E*(3*D*c3 - 1)*c3 - (c1 - c2)^2 - 32;
    B1 = 24*D*E*(c2*c3 - 4);
    B2 = 24*D*E*(c3*c1 - 4);
    B3 = 24*D*E*(c1*c2 - 4);
    
    f = b^6*B1*B2*B3;
    g = 884736*D^3*E^3*a^2*d1^2*(a*d1 - b^2)^2;
    pair = [pair [f; g]];

    f = 2025*b^6/c1/c2*A3;
    g = 1492992*d2^2*b^2/c1^3/c2^3 ...
        + 6912*(a*d1 - b^2)^4/b^2/c1^2/c2^2*( ...
            209*a*d1 + 524*b^2 + d2^2*(209 + 99*b^2) + 396*b^4 ...
        ) + 4*b^2/c1/c2*( ...
            625080*b^2 + 1047320*b^4 ...
            + d2^2*(768 + 29840*d2^2 + 367663*b^2) ...
        ) + 722304*a*d1*b^2 + 583200*(a*d1 - b^2)^2 ...
        + (2603105 + 171072*a*d1)*b^4 ...
        + b^2*(1 - 100*b^2)*(41472 + 1711*b^2) + 28*b^6;
    pair = [pair [f; g]];

    f = 225*b^10/4096/E^2/D^2*(A1*B1^2 + A2*B2^2 - 2*B1*B2*B3);
    g = 432*b^4*(a*d1 - b^2)^2/c1/c2*( ...
            + 864*d2^2*(1 + 2*b^2)/c1/c2 + 209 + 2334*b^2 ...
            + 4696*b^4 + 1728*b^6 ...
            + d2^2*(411 + 1542*b^2 + 432*b^4) ...
        ) ...
        + b^2*(1 - 100*b^2)*(300*b^4 + 3225*b^6 + 1867*b^8 ...
            + 392040*a^4*d1^4) ...
        + 4*b^4*(1 - 10*a*d1)^2*(93271*a^2*d1^2 + 3707*d2^2*b^2) ...
        + 1166400*a^6*d1^6 + 1426680*a^4*d1^4*d2^2*b^2 ...
        + 54*a^3*d1^3*b^4*(49418 + 230061*a*d1) ...
        + a*d1*b^6*(12564*d2^2 + a*d1*(14187 + 713783*d2^2)) ...
        + b^8*(93*d2^2 + 3075550*a^2*d1^2 + 3649*d2^4) ...
        + b^10*(89 + 48752*d2^2) + 76*b^12;
    pair = [pair [f; g]];

    f = sym('8303765625')*b^24*c1^4*c2^4/16384*( ...
        A1*A2*A3 + 2*B1*B2*B3 - A1*B1^2 - A2*B2^2 - A3*B3^2 ...
    );
    g = sym('406239826673664')*d2^6*b^18/c1^2/c2^2 ...
        + sym('940369969152')*d2^4*b^14/c1/c2*( ...
            6*b^2*(5*a*d1 - 18*b^2)^2 ...
            + a^2*d1^2*b^2*(2576*b^2 + 29*(1 - 100*b^2)) ...
            + 1253*a^2*d1^2*(a*d1 - b^2)^2 ...
            + 6*b^4*(a*d1 + 54*(a*d1 - 2*b^2)^2) ...
        ) ...
        + b^4*(1 - 100*b^2)*( ...
            sym('12121552738440000')*a^14*d1^14 ...
            + sym('38379608724728421')*a^12*d1^12*b^2 ...
            + sym('18971342095206083')*a^10*d1^10*b^4 ...
            + sym('2090743154936166')*a^8*d1^8*b^6 ...
            + sym('279697348741361')*a^6*d1^6*b^8 ...
            + sym('3822652869396')*a^4*d1^4*b^10 ...
            + sym('756253238998')*a^2*d1^2*b^12 ...
            + sym('4091161727505')*b^14 ...
            + sym('12150003729052')*b^16 ...
            + sym('50821685560135')*b^18 ...
            + sym('98411555030530')*b^20 ...
            + sym('122686738307665')*b^22 ...
            + sym('97941102999784')*b^24 ...
            + sym('50251546065961')*b^26 ...
            + sym('92177583947338')*b^28 ...
            + sym('7977409739887423')*b^30 ...
        ) + a^3*d1^3*b^8*(1 - 8*a*d1)^2*( ...
            + a^7*d1^7*( ...
                sym('201474903892777042') ...
                + sym('90643664222052075')*d2^4 ...
            ) ...
            + sym('87413325433574170')*a^5*d1^5*d2^4*b^2 ...
            + a^3*d1^3*d2^4*b^4*( ...
                sym('26280162838987044')*a*d1 ...
                + sym('18393062792105615')*d2^2 ...
                + sym('263384776326212272')*a^2*d1^2 ...
            ) + 12*a*d1*d2^2*b^6*( ...
                sym('921487390368760')*a*d1 ...
                + sym('349514116301673')*d2^2 ...
            ) + sym('1983138862306248')*d2^10*b^8 ...
        ) + sym('47569496486400000')*a^18*d1^18*(3*a*d1 + d2^2) ...
        + sym('6802444800000')*b^2*a^16*d1^16*( ...
            13631*a*d1 + 38284*d2^2 + 108214*a^2*d1^2 ...
        ) + 24603750*b^4*a^14*d1^14*( ...
            900363843 + sym('7650287352')*d2^2 ...
            + sym('10182035079')*d2^4 ...
            + sym('3366302686')*d2^6 ...
        ) + 9*b^6*a^12*d1^12*( ...
            sym('60274379589129731')*d2^6 ...
            + sym('647576875800691136')*a^4*d1^4 ...
            + a*d1*d2^2*( ...
                16761821 + sym('32817780980698951')*d2^2 ...
                + sym('123184084009465804')*a*d1*d2^2 ...
            ) ...
        ) + 4*b^8*a^11*d1^11*d2^2*( ...
            sym('89625675600532618')*d2^4 + a*d1*( ...
                sym('1232319731556914497') + 172*d2^2 ...
                + sym('282706757366186700')*d2^4 ...
            ) ...
        ) + b^10*a^9*d1^9*( ...
            252*a^2*d1^2*d2^2 + 8*d2^4*( ...
                sym('680730978067758') ...
                + sym('176303141792417999')*a*d1 ...
            ) + a^3*d1^3*( ...
                sym('28350046403323666988') ...
                + sym('48947181654129068587')*d2^2 ...
                + sym('10169754750431798896')*d2^4 ...
            ) ...
        ) + 2*b^12*a^9*d1^9*( ...
            738*d2^4 + a^2*d1^2*( ...
                sym('6476622057673320575') ...
                + sym('16775086553269092583')*d2^2 ...
            ) + a*d1*d2^4*( ...
                sym('3205280764781424730') ...
                 + sym('2163731183135420737')*d2^2 ...
            ) ...
        ) + b^14*a^6*d1^6*( ...
            7*a^3*d1^3*( ...
               sym('267606134502121283') + 6725*d2^2 ...
            ) ...
            + 3*a*d1*d2^6*( ...
                sym('41287243358734152') ...
                + sym('1035751844957493491')*a*d1 ...
            ) + sym('102133612824349884')*d2^10 ...
            + 2*a^4*d1^4*d2^2*( ...
                sym('4444298389065673601') ...
                + sym('11374900226453699636')*d2^2 ...
            ) ...
        ) + 2*b^16*a^2*d1^2*( ...
            38*a^7*d1^7*( ...
                1217 + sym('119048568159419288')*a*d1*d2^2 ...
                + sym('150837118919724652')*d2^4 ...
            ) + sym('2834407261786822999')*a^6*d1^6*d2^6 ...
            + sym('1114301757844765956')*a^5*d1^5*d2^8 ...
            + sym('3612664492939608')*a^2*d1^2*d2^14 ...
            + d2^10*( ...
                sym('485477278694229') ...
                + sym('11576717008928092')*a^3*d1^3 ...
                + sym('297856692052747442')*a^4*d1^4 ...
            ) ...
        ) + b^18*( ...
            a^7*d1^7*( ...
                80161 + sym('613280621514985073')*a*d1*d2^2 ...
                + sym('2168086120174861319')*d2^4 ...
            ) + 3*a^5*d1^5*d2^6*( ...
                sym('201266834139916751') ...
                + sym('140164857867870021')*d2^2 ...
            ) + sym('134092514706817166')*a^4*d1^4*d2^10 ...
            + a^2*d1^2*d2^12*( ...
                sym('958467215549044') ...
                + sym('23606915947923884')*a*d1 ...
                + sym('113721778713257259')*a^2*d1^2 ...
            ) + sym('850988675449304')*a*d1*d2^18 ...
            + sym('55106777425135')*d2^20 ...
        ) + a*d1*b^20*( ...
            sym('5005973317536822')*a*d1 ...
            + sym('6878225893133844')*d2^10 ...
            + sym('581926730554137620')*a^2*d1^2*d2^8 ...
            + sym('6500684793939788360')*a^7*d1^7 ...
            + sym('19917465388756768960')*a^6*d1^6*d2^2 ...
            + sym('22114487883990056318')*a^5*d1^5*d2^4 ...
            + 2*a^3*d1^3*d2^6*( ...
                sym('697437936152400250') ...
                + sym('2785841237498890944')*a*d1 ...
                + sym('1482271309993033463')*a^2*d1^2 ...
            ) ...
        ) + 2*a*d1*b^22*( ...
            sym('23130096864387570')*d2^8 ...
            + sym('2315937521698808738')*a^5*d1^5 ...
            + sym('840904065784953540')*a^2*d1^2*d2^6 ...
            + a*d1*d2^2*( ...
                sym('14037263072295615') ...
                + sym('2867348900259792053')*a^3*d1^3 ...
            ) ...
            + 3*a^3*d1^3*d2^4*( ...
                sym('332465786752201665') ...
                + sym('682141502084226293')*a*d1 ...
                + sym('977253515068504924')*a^2*d1^2 ...
            ) ...
        ) + a*d1*b^24*( ...
            sym('1278193172773325750')*a^3*d1^3 ...
            + sym('577141788962864220')*a*d1*d2^4 ...
            + a^2*d1^2*d2^2*( ...
                sym('794463032340756283') ...
                + sym('1790154667867305990')*a*d1 ...
            ) + sym('136816263813704700')*d2^8 ...
            + 3*a^2*d1^2*d2^6*( ...
                sym('295254266683834159') ...
                + sym('189245075752900744')*a*d1 ...
            ) ...
        ) + a*d1*b^26*( ...
            sym('7342158644044464784')*a^4*d1^4 ...
            + sym('13195017945764608985')*a^3*d1^3*d2^2 ...
            + sym('6273782071117534828')*a^2*d1^2*d2^4 ...
            + sym('1505967078540509160')*a*d1*d2^6 ...
            + sym('222276557026854540')*d2^8 ...
            + sym('11142172580460000000')*a^5*d1^5 ...
        ) + 2*a*d1*b^28*( ...
            sym('371333521783157905')*a^2*d1^2 ...
            + sym('110859693910065990')*d2^4 ...
            + a*d1*d2^2*( ...
                sym('113043479305733967') ...
                + sym('681897555397501901')*a*d1 ...
                + sym('350076133701617264')*a^2*d1^2 ...
            ) ...
        ) + a*d1*b^30*( ...
            sym('43083707928314887') ...
            + sym('94188165135443957')*d2^4 ...
            + 2*a*d1*d2^2*( ...
                sym('38270182911979807') ...
                + sym('19065464961694444')*d2^2 ...
            ) ...
        ) + 4*a*d1*b^32*( ...
            sym('59079847399876688')*a^2*d1^2 ...
            + 3*d2^2*( ...
                sym('3833975366004819') ...
                + sym('7362306726062380')*a*d1 ...
            ) ...
        ) + 4*a*d1*b^34*( ...
            sym('1097149099950981') ...
            + sym('4490808002806144')*a*d1 ...
        ) + 4*b^36*( ...
            sym('199324506182122989') ...
            + sym('113055729475640')*d2^2 ...
        ) + sym('84587928319744')*b^38;
    pair = [pair [f; g]];

    for p = pair
        if simplifyFraction(p(1) - p(2)) ~= 0
            error('There is something wrong.');
        end
    end
    disp('It is all right.');
end
\end{lstlisting}


\begin{thebibliography}{10}

\bibitem{ArcangeliGout}
{\sc Arcangeli, R. \& Gout, J. L.} (1976) {Sur l'\'evaluation de i'erreur d'interpolation de Lagrange dans un ouvert de $\mathbb{R}^n$}.
{\em R.A.I.R.O. Analyse Num\'erique}, {\bf 10}, 5--27. (in French)

\bibitem{BabuskaAziz}
{\sc Babu\v{s}ka, I. \& Aziz, A. K.} (1976) {On the angle condition in the finite element method}.
{\em SIAM Journal on Numerical Analysis}, {\bf 13}, 214--226.

\bibitem{BrennerScott}
{\sc Brenner, S. C. \& Scott, L. R.} (2002)
{\em The Mathematical Theory of Finite Element Methods}, Springer.

\bibitem{CarstensenGedicke}
{\sc Carstensen, C. \& Gedicke, J.} (2014) {Guaranteed lower bounds for eigenvalues}.
{\em Math. Comp.},  {\bf 83}(290), 2605--2629.

\bibitem{Ciarlet}
{\sc Ciarlet, P. G.} (2002)
{\em The Finite Element Method for Elliptic Problems}, SIAM.

\bibitem{HornJohnson}
{\sc Horn, R. \& Johnson, C.} (2013)
{\em Matrix Analysis}, Cambridge University Press, Cambridge; New York, 2nd edition.

\bibitem{Hu}
{\sc Hu, J., Huang, Y. \& Lin, Q.} (2014) {Lower Bounds for Eigenvalues of Elliptic Operators: By Nonconforming Finite Element Methods}.
{\em J. Sci. Comput.}, {\bf 61}, 196--221.

\bibitem{KikuchiLiu2007}
{\sc Kikuchi, F. \& Liu, X.} (2007) {Estimation of interpolation error constants for the $p_0$ and $p_1$ triangular finite elements}.
{\em Comput. Methods Appl. Mech. Engrg.}, {\bf 196}, 3750--3758.

\bibitem{KikuchiSaito}
{\sc Kikuchi, F. \& Saito, H.} (2007) {Remarks on a posteriori error estimation for finite element solutions}.
{\em J. Comp. Appl. Math.}, {\bf 199}, 329--336.

\bibitem{Kobayashi}
{\sc Kobayashi, K.} (2015) {On the interpolation constants over triangular elements}.
{\em Application of Mathematics}, Proceedings. Prague, November 18-21, 2015. Institute of Mathematics CAS, Prague, 110--124.

\bibitem{KobayashiTsuchiya}
{\sc Kobayashi, K. \& Tsuchiya, T.} (2014) {A Babu\v{s}ka-Aziz type proof of the circumradius condition}.
{\em Japan J. Indust. Appl. Math.}, {\bf 31}, 193--210.

\bibitem{KobayashiTsuchiya2}
{\sc Kobayashi, K. \& Tsuchiya, T.} (2015) {On the circumradius condition for piecewise linear triangular elements}.
{\em Japan J. Indust. Appl. Math.}, {\bf 32}, 65--76.

\bibitem{LaugesenSiudeja}
{\sc Laugesen, R. S. \& Siudeja, B. A.} (2010) {Minimizing Neumann fundamental tones of triangles: An optimal Poincar\'e inequality}.
{\em J. Differential Equations}, {\bf 249}, 118--135.

\bibitem{Lehmann}
{\sc Lehmann, R.} (1986) {Computable error bounds in finite-element method}.
{\em IMA Journal of Numerical Analysis}, {\bf 6}, 265--271.

\bibitem{LiuKikuchi}
{\sc Liu, X. \& Kikuchi, F.} (2010) {Analysis and estimation of error constants for $p_0$ and $p_1$ interpolations over triangular finite elements}.
{\em J. Math. Sci. Univ. Tokyo}, {\bf 17}, 27--78.

\bibitem{LiuOishi}
{\sc Liu, X. \& Oishi, S.} (2013) {Guaranteed high-precision estimation for {$P_0$} interpolation constants on triangular finite elements}.
{\em Japan J. Indust. Appl. Math.}, {\bf 30}, 635--652.

\bibitem{LuoLinXie}
{\sc Luo, F., Lin, Q. \& Xie, H.} (2012) {Computing the lower and upper bounds of Laplace eigenvalue problem: by combining conforming and non-conforming finite element methods}.
{\em Science China Mathematics}, {\bf 55}, 1069--1082.

\bibitem{MeinguetDescloux}
{\sc Meinguet, J. \& Descloux, J.} (1977) {An operator-theoretical approach to error estimation}.
{\em Numer. Math.}, {\bf 27}, 307--326.

\bibitem{Morley}
{\sc Morley, L. S. D.} (1968) {The triangular equilibrium element in solving plate-bending problems}.
{\em Aero Quart}, {\bf 19}, 149--169.

\bibitem{MooreKearfottCloud}
{\sc Moore R. E., Kearfott, R. B. \& Cloud, M. J.} (2009)
{\em Introduction to Interval Analysis}, Cambridge Univ. Press.

\bibitem{NakaoYamamoto}
{\sc Nakao, M. T. \& Yamamoto, N.} (2001) {A guaranteed bound of the optimal constant in the error estimates for linear triangular element}.
{\em Computing Supplementum}, {\bf 15}, 163--173.

\bibitem{Natterer}
{\sc Natterer, F.} (1975) {Berechenbare Fehlerschranken f\"ur die Methode der Finite Elemente}.
{\em International Series of Numerical Mathematics}, {\bf 28}, 109--121. (in German)

\bibitem{PayneWeinberger}
{\sc Payne, L. E. \& Weinberger, H. F.} (1960) {An optimal Poincar\'e inequality for convex domains}.
{\em Arch. Rat. Mech. Anal.}, {\bf 5}, 286--292.

\bibitem{Repin}
{\sc Repin, S. I.} (2012) {Computable majorants of constants in the Poincar\'e and Friedrichs inequalities}.
{\em Journal of Mathematical Sciences}, {\bf 186}, 307--321.

\bibitem{Rump1}
{\sc Rump, S. M.} (2006) {Verification of positive definiteness}.
{\em BIT Numer. Math.}, {\bf 46}, 433--452.

\bibitem{Rump2}
{\sc Rump, S. M.} (2010) {Verification methods: Rigorous results using floating-point arithmetic}.
{\em Acta Numerica}, {\bf 19}, 287--449.

\bibitem{SebestovaVejchodsky}
{\sc Sebestova, I. \& Vejchodsky, T.} (2014) {Two-sided bounds for eigenvalues of differential operators with applications to Friedrichs', Poincar\'e, trace, and similar constants}.
{\em SIAM J. Numer. Anal.}, {\bf 52}, 308--329.

\bibitem{Zlamal}
{\sc Zl\'amal, M.} (1968) {On the Finite Element Method}.
{\em Numerische Mathematik}, {\bf 12}, 394--409.

\bibitem{GitHub}
\url{https://github.com/kobayashi-kenta/remarkable-upper-bounds}

\end{thebibliography}
\end{document}